\documentclass[a4paper, 11pt]{amsart}

\usepackage{enumerate}
\usepackage{mathrsfs}
\usepackage{amssymb}
\usepackage[all]{xy}
\usepackage[utf8]{inputenc}
\usepackage[toc,page]{appendix}

\usepackage{color} 

\newcommand{\C}{{\mathcal{C}}}
\newcommand{\rep}{{\mathrm{rep}}}
\newcommand{\vect}{{\mathrm{vect}}}
\newcommand{\Coh}{{\mathrm{Coh}}}
\newcommand{\dimvec}{\underline{\dim}}

\renewcommand{\mod}{{\mathrm{mod}}}
\newcommand{\gr}{\mathrm{gr}}
\newcommand{\Hom}{{\mathrm{Hom}}}
\newcommand{\GL}{{\mathrm{GL}}}
\newcommand{\Gm}{{\mathbb{G}_m}}
\newcommand{\sst}{{\mathrm{sst}}}
\newcommand{\st}{{\mathrm{st}}}
\newcommand{\ssp}{{\mathrm{ssp}}}
\newcommand{\Lul}{{\underline{L}}}

\newcommand{\ch}{{\mathrm{ch}}}
\newcommand{\todd}{{\mathrm{todd}}}

\newcommand{\E}{{\mathcal{E}}}
\newcommand{\F}{{\mathcal{F}}}
\newcommand{\rk}{{\mathrm{rk}}} 
\newcommand{\sat}{{\mathrm{sat}}}
\newcommand{\hmu}{{\widehat{\mu}}}
\newcommand{\Ocal}{{\mathcal{O}}}

\newcommand{\links}{{\mathrm{left}}}
\newcommand{\rechts}{{\mathrm{right}}}

\newcommand{\G}{{\mathcal{G}}}
\newcommand{\K}{{\mathcal{K}}}
\newcommand{\U}{{\mathcal{U}}}
\newcommand{\M}{{\mathcal{M}}}

\newcommand{\reg}{{\mathrm{reg}}}

\newcommand{\imgreg}{{[\mathrm{reg}]}}
\newcommand{\imgsst}{{[\mathrm{sst}]}}
\newcommand{\imgst}{{[\mathrm{st}]}}

\newcommand{\Cbar}{{\overline{C}}}
\newcommand{\ev}{{\mathrm{ev}}}

\newcommand{\Mbb}{{\mathbb{M}}}
\newcommand{\Ebb}{{\mathbb{E}}}

\renewcommand{\flat}{{\mathrm{flat}}}
\newcommand{\R}{{\mathbb{R}}}
\newcommand{\Q}{{\mathbb{Q}}}
\newcommand{\Z}{{\mathbb{Z}}}
\newcommand{\N}{{\mathbb{N}}}

\newcommand{\Spec}{{\mathrm{Spec}}}
\renewcommand{\ker}{{\mathrm{ker}}}
\newcommand{\im}{{\mathrm{im}}}

\newcommand{\Ext}{{\mathrm{Ext}}}
\newcommand{\HN}{{\mathrm{HN}}}
\newcommand{\Gal}{{\mathrm{Gal}}}

\newcommand{\Aut}{\mathrm{Aut}}
\newcommand{\I}{{\mathcal{I}}}
\newcommand{\m}{{\mathfrak{m}}}

\newcommand{\op}{{\mathrm{op}}}

\newcommand{\End}{{\mathrm{End}}}

\newcommand{\Quotfunc}{{\mathcal{Q}\mathrm{uot}}}
\newcommand{\Quot}{{\mathrm{Quot}}}
\newcommand{\Grassfunc}{{\mathcal{G}\mathrm{rass}}}

\newcommand{\Sym}{{\mathrm{Sym}}}
\renewcommand{\P}{{\mathbb{P}}}

\newtheorem{proposition}{Proposition}[section]
\newtheorem{lemma}[proposition]{Lemma}
\newtheorem{corollary}[proposition]{Corollary}
\newtheorem{theorem}[proposition]{Theorem}
\theoremstyle{definition}
\newtheorem{thmx}{Theorem}

\newtheorem{definition}[proposition]{Definition}
\newtheorem{remark}[proposition]{Remark}
\newtheorem{example}[proposition]{Example}
\newtheorem{assumption}[proposition]{Assumption}
\newtheorem{convention}[proposition]{Convention}

\title[]{Multi-Gieseker semistability and moduli of quiver sheaves}
\author{Marcel Maslovari\'c and Henrik Sepp\"anen} 

\thanks{The first author was supported by the DFG Research Training Group 1493 "Mathematical Structures in Modern Quantum Physics".}

\address{Marcel Maslovari\'c,
Mathematisches Institut,
Georg-August-Univer\-sit\"at G\"ot\-tingen,
Bunsenstra\ss e 3-5, 
D-37073 G\"ottingen,
Germany}
\email{marcel-helmut.maslovaric@mathematik.uni-goettingen.de}

\address{Henrik Sepp\"{a}nen,
Mathematisches Institut,
Georg-August-Univer\-sit\"at G\"ot\-tingen,
Bunsenstra\ss e 3-5, 
D-37073 G\"ottingen,
Germany}
\email{henrik.seppaenen@mathematik.uni-goettingen.de}

\begin{document}
 
\maketitle

%

\begin{abstract}
 We generalize the notion of multi-Gieseker semistability for coherent sheaves, introduced by Greb, Ross, and Toma, to quiver sheaves for a quiver $Q$. We construct coarse moduli 
 spaces for semistable quiver sheaves using a functorial method that realizes these as subschemes of moduli spaces of representations of 
 a twisted quiver, depending on $Q$, with relations. We also show the projectivity of the moduli space in the case when $Q$ has no oriented cycles. Further, we construct moduli spaces of quiver sheaves which satisfy a given set of relations as closed subvarieties. 
 Finally, we investigate the parameter dependence of the moduli.
\end{abstract}

\tableofcontents

\section{Introduction}
Mumfords's Geometric Invariant Theory (GIT) (cf. \cite{mumford}) plays an important role in the construction of moduli spaces of vector bundles, and, more generally, coherent sheaves, on projective 
varieties. Instead of considering all coherent sheaves of a given topological type, that is, with some predescribed Chern character, one only looks at the \emph{semistable} ones, a 
condition that compares a Hilbert polynomial of a sheaf with that of its subsheaves. 

The idea of Gieseker's construction is to construct the moduli space as a quotient of a locally closed subset of a Grassmannian, or a Quot-scheme, by the action of a general linear group. The 
semistable sheaves will then appear as orbit closures of semistable points--in the sense of GIT--with respect to a line bundle on the Grassmannian.

Another example of the role of GIT in the study of moduli is King's construction of moduli spaces of semistable representations of a quiver, $Q$, with a given 
dimension vector $d \in \N^{Q_0}$ (\cite{king}). Again, there is 
a natural notion of semistability, depending on some tuple $\theta$ of real parameters, for a quiver representation, and King identifies this notion with Mumford's notion of 
semistability for points on the so-called representation variety $R_d(Q)$, where a suitable linearization is defined.  

The notions of semistability from the two above examples meet in the functorial construction of moduli spaces of Gieseker-semistable coherent sheaves due to 
\'Alvarez-C{\'o}nsul and King (\cite{ack}). Here, a functor from the category $\mbox{Coh}(X)$ of coherent sheaves on $X$ to the category of representations of a 
Kronecker quiver is defined in such a way that the moduli space of coherent sheaves is given as a GIT quotient of a locally closed subset of a 
representation variety $R_d(Q)$ of this quiver. Thus, the moduli of sheaves becomes embedded as a closed subscheme of a moduli space of semistable representations of $Q$.

In both cases, some parameter needs to be chosen in order to define the appropriate notion of semistability. In the case of sheaves, this parameter is an ample line bundle, $L$, on $X$, so that 
the Hilbert polynomials are calculated with respect to this line bundle. In the second case, the parameter is a tuple of real numbers. These parameters can naturally vary in a real vector space.
When Hilbert polynomials are considered as parameters defining semistability, it is natural to fix a tuple $\Lul=(L_1,\ldots, L_N)$ 
of ample line bundles and to compute Hilbert polynomials with respect to $\Q$-line bundles $L_1^{q_1} \otimes \cdots \otimes L_N^{q_N}$, for tuples $(q_1,\ldots, q_N)$ of nonnegative rational 
numbers. In the case when $X$ is a surface, it was shown by Matsuki and Wentworth (\cite{matsukiwentworth}) that the cone of parameters is divided up by locally finitely many rational walls
into chambers in such a way that two rational points in the interior of the same full-dimensional cone define the same moduli space. Moreover, the moduli spaces 
from two chambers separated by one of these walls are related via a Thaddeus flip, a notion from GIT which produces a birational map if both spaces are 
irreducible (cf. \cite{thaddeus}). In fact, similar results had previously been proven for slope semistability by several authors 
(cf. the introduction in \cite{grt}). In higher dimensions, the picture is more complicated. First of all, the walls dividing the 
cone of parameters into subcones with constant moduli need neither be linear, nor rational. In fact, Schmitt (\cite{schmitt-wall}) found an example of a wall 
separating two full-dimensional cones that contains no rational points. For points in such a wall there is no obvious definition of a Hilbert polynomial, and 
moduli spaces of the type above do not make sense.

Motivated by this lack of a notion of semistability for irrational parameters, Greb, Ross, and Toma (\cite{grt}) introduced the notion of multi-Gieseker semistablity. Again, a tuple 
$\Lul=(L_1,\ldots, L_N)$ of ample line bundles is fixed. However, instead of looking at tensor powers, the authors consider 
nonnegative linear combinations $\sum_{j=1}^N \sigma_j P^{L_j}_{E}$, with $\sigma_1,\ldots, \sigma_N \in \R$, of the Hilbert polynomials 
$P^{L_j}_{E}$ with respect to the $L_j$, of  a coherent sheaf $E$. Moreover, they provide a functorial construction of a moduli space--generalizing that of \cite{ack}--by defining a functor from 
the category $\mbox{Coh}(X)$ to the category of representations of a certain quiver which depends on the line bundles $L_1,\ldots, L_N$ and the section spaces of tensor powers of these line bundles.

In this paper, we address the question, asked in \cite{grt}, whether the construction of multi-Gieseker semistable sheaves has a counterpart for quiver sheaves, that is, for 
representations of a quiver, $Q$, in the category $\mbox{Coh}(X)$ of coherent sheaves. As a simple example of a quiver sheaf, we could consider triples $(\mathcal{E}_1, \mathcal{E}_2, \varphi)$, 
where the $\mathcal{E}_i$ are vector bundles on $X$, and $\varphi: \mathcal{E}_1 \to \mathcal{E}_2$ is a morphism of $\mathcal{O}_X$-modules.

We introduce a notion of multi-Gieseker semistability which depends on a tuple $\Lul=(L_1,\ldots, L_N)$ of ample line bundles, as well as on a pair, $(\sigma, \rho)$, 
of tuples in $\R^{Q_0 \times N}$. This generalizes both the above mentioned multi-Gieseker semistability for sheaves and King's notion of semistability for quiver representations. 

To be more precise, let $Q=(Q_0, Q_1)$ be a quiver, where $Q_0$ is the set of vertices, and $Q_1$ is the set of arrows. Let 
$\sigma, \rho \in (\R_{\geq 0})^{Q_0 \times N} \setminus 0$ be tuples of nonnegative real numbers. For 
a coherent quiver sheaf $\mathcal{E}$ with sheaves $\mathcal{E}_i, i \in Q_0$, attached to the vertices of $Q$ we define 
the multi-Hilbert polynomial with respect to $(\Lul, \sigma)$ by
\begin{align*}
P_{\mathcal{E}}^{(\sigma, \rho)}(m)=\sum_{i\in Q_0}\sum_{j=1}^N \sigma_{ij} \chi(\mathcal{E}_i \otimes L_j^m), \quad m \in \N.
\end{align*}
Moreover, we define the rank of $\mathcal{E}_i$ with respect to $L_j$ as 
\begin{align*}
 \rk^{L_j}(\mathcal{E}_i)=\rk(\mathcal{E}_i)\int_X c_1(L_j)^n,
\end{align*}
where $n=\dim(X)$. 
The reduced multi-Hilbert polynomial of $\mathcal{E}$ with respect to $(\Lul, \sigma, \rho)$ is then
\begin{align*}
 p_{\mathcal{E}}^{(\sigma, \rho)}=\frac{P_{\mathcal{E}}^{\sigma}}{\sum_{i,j} \rho_{ij}\rk^{L_j}_{\mathcal{E}_i}}.
\end{align*}
We require from a quiver sheaf $\mathcal{E}$ that all the sheaves $\mathcal{E}_i$ be pure of the same dimension, $d$. In this case, 
we say that $\mathcal{E}$ is pure of dimension $d$.
A quiver sheaf $\mathcal{E}$ is multi-Gieseker semistable with respect to $(\Lul,
\sigma, \rho)$ if it is pure and for all nontrivial quiver subsheaves 
$\mathcal{F}$ we have
\begin{align*}
 p_{\mathcal{E}}^{(\sigma, \rho)}\leq p_{\mathcal{F}}^{(\sigma, \rho)}, 
\end{align*}
where the inequality is with respect to the lexicographic order of the coefficients of the polynomials. We say that $\mathcal{E}$ is 
stable if the inequality is strict for all nontrivial proper quiver subsheaves $\mathcal{F}$. 

We then see that King's notion of a semistable representation of $Q$ in the category of vector spaces over $k$ is indeed 
the special case when $X$ is a point, and $\rho=(1,\ldots, 1)$, whereas the semistability condition for sheaves in \cite{grt} is 
the special case of a trivial quiver and $\rho=\sigma$.

Being inspired by the work \cite{grt}, in the present paper we focus on the case $\rho=\sigma$, for an arbitrary quiver $Q$.
We then write $p_{\mathcal{E}}^\sigma$ instead of $p_{\mathcal{E}}^{(\sigma, \sigma)}$, and say that a quiver sheaf is $\sigma$-semistable. 
We let $\tau=(\tau_i)_{i \in Q_0} \in B(X)_\Q^{Q_0}$ denote a topological type. Our goal is to construct moduli spaces 
for families of $\sigma$-semistable quiver sheaves of a fixed topological type $\tau$. The main theorem of the paper is the following. 
(We present it here in a simplified version for the sake of readability. The precise formulation concerns quiver sheaves satisfying a set $I$ of relations.)
\begin{thmx} (cf. Theorem \ref{T: existmoduli}, Corollary \ref{C:C:relations_closed})
Let $Q$ be a quiver, $k$ an algebraically closed field of characteristic zero, and $X$ a smooth projective variety over $k$. Let $\tau$ be a fixed topological type. 
If the family of $\sigma$-semistable quiver sheaves of topological type $\tau$ is bounded, then the coarse moduli space of $\sigma$-semistable quiver sheaves of 
topological type $\tau$ exists 
as a quasi-projective scheme, $M^\sigma$, over $k$. If $Q$ has no oriented cycles, and $\sigma$ is rational, $M^\sigma$ is projective over $k$.
\end{thmx}

In order to ensure the boundedness condition, we consider the special case when the parameter $\sigma_{ij}$ is \emph{symmetric}, that is, it 
only depends on the second index $j$. We can then 
identify $\sigma$ with a parameter $\hat{\sigma} \in \R_{\geq 0}^N$. In this case, a $\sigma$-semistable quiver sheaf $\mathcal{E}$ consists of semistable sheaves 
$\mathcal{E}_i$ in the sense of Greb, Ross, and Toma. More precisely, we have the following theorem.

\begin{thmx} (cf. Proposition \ref{P: semistab-tuple-quivsheaf}) \label{T: sheaf-tuple-intro}
Assume that $\sigma$ is symmetric, and let $\hat{\sigma} \in \R_{\geq 0}^N$ be 
the associated tuple. A quiver sheaf $\mathcal{E}$ is $\sigma$-semistable if and only if each $\mathcal{E}_i$ is $\hat{\sigma}$-semistable and 
the multi-Hilbert polynomials $p^{\hat{\sigma}}_{\mathcal{E}_i}$ of the $\mathcal{E}_i$ all satisfy 
$p^{\hat{\sigma}}_{\mathcal{E}_i}=p^{\sigma}_{\mathcal{E}}$.
\end{thmx}
This means that the boundedness results from \cite{grt} (Cor. 6.12) are applicable. In particular, if $\dim(X) \leq 3$, then the 
family of $\sigma$-semistable quiver sheaves is bounded for any symmetric $\sigma$.

The construction of the moduli space follows the functorial method by \'Alvarez-C{\'o}nsul and King (\cite{ack}), which was further developed 
by Greb, Ross, and Toma in \cite{grt}. As in \cite{grt}, 
the tuple $\underline{L}$ determines a certain auxiliary labeled quiver $(Q', H)$, where the labels $H_\alpha$ on the arrows are section spaces of 
products of powers of the $L_j$. We then use our quiver $Q$ to define a ``twisted'' quiver $Q(Q')$. Roughly speaking, each vertex $i$ of $Q$ 
defines a copy of $Q'$, and the arrows of $Q$ define ``morphisms'' of various copies of $Q'$, meaning that we are interested in 
representations of $Q(Q')$ where the diagrams induced by the arrows of $Q$ are required to commute. 
As a generalization of the construction in \cite{ack} and \cite{grt}, we define a functor 
\begin{align*}
 \mbox{Hom}(T, \bullet): Q-\mbox{Coh}(X) \rightarrow (Q(Q'), H, I_1')-\mbox{rep},
\end{align*}
from the category of coherent quiver sheaves on $X$ to the category of representations of the labeled quiver $(Q(Q'), H)$ in vector spaces with 
a set of relations $I_1'$, where the last condition means the commutativity of the diagrams in $Q(Q')$ defined by the arrows in $Q_1$.
This functor identifies $\sigma$--semistable quiver sheaves with certain semistable representations of $(Q(Q'), H, I_1')$ in the sense of GIT, 
and this allows us to obtain the moduli space $M^\sigma$ as a locally closed subscheme of a GIT quotient, namely of a 
moduli space of semistable representations of $(Q(Q'), H, I_1')$.
Finally, if we consider quiver sheaves satisfying a set of relations $I$, as in the remark before the above theorem, the relations $I$ 
define a set of relations $I_2'$ on the ``morphisms'' of $Q'$ that also has to be taken into account on the right hand side that concerns representations of $Q(Q')$.


Just as the motivation for the approach in \cite{grt} was to give a meaning to moduli spaces for nonrational stability conditions, and to study their dependence on this 
stability condition, we obtain a similar result in the case of symmetric parameters $\sigma$. For this, we need the moduli spaces to be projective. Again, we
present a simplified version of the result and remark that, in contrast to the case in \cite{grt}, the walls are in general given as quadratic hypersurfaces 
(cf. Proposition \ref{V:P:chambers}).
\begin{thmx} (cf. Proposition \ref{SYM:P:wallslinear}, Corollary \ref{C: linearwalls}) 
 For torsion-free quiver sheaves, the set of symmetric stability conditions $\sigma$ is partitioned into equivalence classes by a finite set of rational linear hyperplanes 
 in such a way that the set of $\sigma$--semistable quiver sheaves is constant on these classes. Moreover, if $Q$ has no oriented cycles, this implies that the moduli spaces 
 $M^\sigma$ are projective over $k$ for
 arbitrary bounded nonnegative symmetric $\sigma$. 
 
 If $\sigma$ and $\sigma'$ are two bounded symmetric stability conditions, the 
 moduli spaces $M^\sigma$ and $M^{\sigma'}$ are related via a finite sequence of Thaddeus flips. These flips are the ``wall-crossings'' defined by the VGIT of 
 representations of $Q(Q')$.
\end{thmx}

We end this introduction by remarking that there are other approaches to semistability conditions on quiver sheaves (cf. \cite{alvarezconsul}, \cite{schmitt-tuple}, \cite{schmitt-qsheaf}). 
However, these all have the common feature that only one ample line bundle is used to define the Hilbert polynomials for the sheaves involved, and then linear combinations of these 
are formed where the coefficients are allowed to be certain polynomials. In contrast to our multi-Gieseker approach, the boundedness of the families of semistable 
quiver sheaves holds for much the same reason as for the classical Gieseker semistability for sheaves. 

\bigskip
{\small\noindent {\bf Acknowledgements.} } The authors would like to thank Daniel Greb, Markus Reineke, Matei Toma, and Georg Merz for helpful discussions about the 
topic of this paper, as well as Manfred Lehn for having kindly answered our questions about semistability and moduli.
The first author is grateful to the DFG for financial support within the framework of the Research Training Group 
1493 ``Mathematical Structures in Modern Quantum Physics''.

\section{Quiver sheaves}

We start by introducing the notion of multi--Gieseker stability and the category of quiver sheaves. Also, we establish some foundational results about this notion of stability.~\\

\subsection{Basic notions}~\\

Let $X$ denote any scheme, and let $Q$ denote an (unlabeled) quiver.

\begin{definition}
A quiver sheaf $\E$ on $X$ associated to the quiver $Q$ is a representation of $Q$ in the category of coherent sheaves on $X$. That is, we have a family of coherent sheaves $\E_i,~i\in Q_0$ associated to the vertices, and a family $\E_\alpha:\E_i\to \E_j,~\alpha:i\to j$ of morphisms of coherent sheaves associated to the arrows.\\

Occasionally, we will also consider representations of $Q$ in the category of quasi--coherent sheaves, and call them quasi--coherent quiver sheaves.\\

A morphism of quiver sheaves $\varphi:\E\to \F$ is a collection of morphisms $\varphi_i:\E_i\to \F_i$ for each vertex $i\in Q_0$, such that the canonical diagram
$$\xymatrix{\E_i \ar[rr]^{\varphi_i} \ar[d]_{\E_\alpha} && \F_i \ar[d]^{\F_\alpha} \\
\E_j \ar[rr]_{\varphi_j} && \F_j}$$
for each arrow $\alpha:i\to j$ commutes.
\end{definition}

For a path $\gamma=\alpha_l\ldots\alpha_1$ and a quiver sheaf $\E$ recall that we denote
$$\E_\gamma=\E_{\alpha_l}\ldots\E_{\alpha_1},$$
and say that $\E$ satisfies a relation $\sum_{k=1}^r \lambda_k \gamma_k$ if the identity
$$\sum_{k=1}^r \lambda_k \E_{\gamma_k}=0$$
of morphisms of sheaves holds. This property is inherited by quiver subsheaves and quiver quotient sheaves.\\

To give a first idea about how to interpret quiver sheaves as configurations of morphisms of sheaves, we provide some simple examples.

\begin{example}
The simplest non--trivial example of a quiver sheaf is a morphism
$$\E_1\xrightarrow{\E_\alpha}\E_2,$$
associated to the quiver consisting of two vertices and an arrow between them. Another example is that of a pair $(E,\phi)$, consisting of a sheaf $E$ and an endomorphism $\phi$. By using relations we could require
\begin{center}
$\phi^r=0$, or $\phi^r=\phi$,
\end{center}
ie. that $\phi$ is nilpotent or idempotent respectively. A slightly more sophisticated example is that of two composable morphisms
$$\E_1\xrightarrow{\E_\alpha}\E_2\xrightarrow{\E_\beta}\E_3,$$
where we could require $\E_\beta\E_\alpha=0$ by using relations. A commuting square of morphisms of sheaves can also be expressed as a quiver sheaf satisfying a relation.
\end{example}

Let us recall the notion of a topological type $\tau$ of a sheaf $E$ (compare with Definition 1.4 and Remark 1.5 of \cite{grt}). It is given as
$$\tau(E)=\ch(E)\todd(X)\in B(X)_\Q,$$
where $B(X)$ is the group of cycles on $X$ modulo algebraic equivalence (consider the proof of \cite{fulton} Theorem 15.2). By definition, the topological type thus remains constant in flat families over a connected and noetherian base scheme. Further, it determines the Hilbert polynomial with respect to any ample line bundle (\cite{fulton}, Example 18.3.6).\\
By a slight abuse of notation we use the same letter for the topological type of a quiver sheaf.
\begin{definition}
We say that a quiver sheaf $\E$ has topological type
$$\tau\in B(X)_\Q^{Q_0}$$
if for all $i\in Q_0$ the sheaf $\E_i$ has topological type $\tau_i$.
\end{definition}

There are some other properties of sheaves which transfer to quiver sheaves simply by requiring the sheaves at the vertices to fulfill this property.
\begin{definition}
A quiver sheaf $\E$ is called pure of dimension $d$ if all the shaves $\E_i$ are pure of dimension $d$.
\end{definition}

Assume that $X$ is projective over $k$, and fix a tuple
$$\Lul=\left(L_1,\ldots,L_N\right)$$
of ample line bundles. Following \cite{grt}, we say that a sheaf $E$ is $(n,\Lul)$--regular if it is $n$--regular with respect to each line bundle $L_j$.
\begin{definition}
A quiver sheaf $\E$ is called $(n,\Lul)$--regular if all the sheaves $\E_i$ are $(n,\Lul)$--regular. Equivalently, the quiver sheaf $\E$ is $(n,\Lul)$--regular if $\E_i$ is $n$--regular with respect to $L_j$ for all $i\in Q_0$ and all $j=1,\ldots,N$.
\end{definition}

Another such property is the saturation of sheaves.
\begin{lemma}\label{QS:L:saturation_functorial}
Let $\phi:E_1\to E_2$ denote a morphism of sheaves on $X$ which are of the same dimension $d$, and let $F_1\subset E_1$, $F_2\subset E_2$ denote subsheaves which are respected by the morphism, i.e. $\phi:F_1\to F_2$. Then the saturations
$$\phi:F_{1,\sat}\to F_{2,\sat}$$
are respected as well.

\begin{proof}
Recall (\cite{huybrechtslehn} Definition 1.1.5), that the saturation is given as
$$F_{1,\sat}=\ker\left(E_1\to E_1/F_1\to \left(E_1/F_1\right)/\left(T_{d-1}\left(E_1/F_1\right)\right)\right),$$
where $T_{d-1}$ of a sheaf is defined to be the sum over all subsheaves of dimension at most $d-1$. We have an induced morphism $\phi':E_1/F_1\to E_2/F_2$. Since the dimension of a sheaf does not grow under taking the image, we have an induced morphism between the second quotients occurring in the above description of the saturation as well. The obvious commuting diagram tells us that $\phi$ thus induces a morphism between the kernels, which gives the result.
\end{proof}
\end{lemma}

Lemma \ref{QS:L:saturation_functorial} shows that saturation of a quiver subsheaf is well--defined.
\begin{definition}
Consider a quiver subsheaf $\F\subset \E$ of a quiver sheaf purely of dimension $d$. The saturation of $\F$ in $\E$ is defined as the quiver subsheaf $\F_\sat\subset \E$ such that
$$\left(\F_\sat\right)_i=\F_{i,sat}$$
for all vertices $i\in Q_0$.
\end{definition}

\begin{definition}
A family of quiver sheaves on $X$ is called flat or bounded over a $k$--scheme $S$ of finite type, if the families of sheaves $\E_i$ are flat or bounded for all $i\in Q_0$ respectively.
\end{definition}
~\\

\subsection{Stability conditions}~\\

Now we introduce the notion of stability for quiver sheaves. To relate them to the already established notions of stability for vector space representations of a quiver (\cite{king}), and to the stability condition for sheaves as in \cite{grt}, we first give a slightly more general definition which interpolates between these two. However, our construction and variation results only hold in a special case.\\

Consider a projective scheme $X$ over $k$.
\begin{definition}
A multi--Gieseker stability condition $(\Lul,\sigma,\rho)$ for quiver sheaves associated to $Q$ over $X$ consists of a tuple
$$\Lul=\left(L_1,\ldots,L_N\right)$$
of ample line bundles on $X$, and two tuples
$$\sigma,\rho\in\left(\R^{Q_0\times N}\right)_+,$$
where the right hand side is the subset of $\R^{Q_0\times N}$ consisting of tuples $\sigma$ of non--negative real numbers such that for any fixed $i$ not all $\sigma_{ij}$ vanish simultaneously. We denote the entries of such tuples $\sigma$ as
$$\sigma=\left(\sigma_{ij}\right)_{(i,j)\in Q_0\times N}.$$
If all entries of $\sigma$ are strictly positive we say that $\sigma$ is positive.
\end{definition}

Typically, we think of $\Lul$ as being fixed, and just refer to $(\sigma,\rho)$ as the stability condition if $\Lul$ is clear from the context.\\

\begin{definition}
The multi--Hilbert polynomial of a quiver sheaf $\E$ with respect to $\sigma$ is defined as
$$P^\sigma_\E=\sum_{i\in Q_0}\sum_{j=1}^N \sigma_{ij}P^{L_j}_{\E_i},$$
where $P^{L_j}_{\E_i}$ is the usual Hilbert polynomial of $\E_i$ computed with respect to the ample line bundle $L_j$.
\end{definition}

The multi--Hilbert polynomial can be written as
$$P^\sigma_\E(T)=\frac{\alpha^\sigma_d(\E)}{d!}T^d+\frac{\alpha^\sigma_{d-1}(E)}{(d-1)!}T^{d-1}+\ldots+\alpha^\sigma_0(\E)$$
for real numbers $\alpha^\sigma_k(\E)$, and strictly positive leading coefficient. Here, $d$ is the maximum of the dimensions of the sheaves $\E_i$. The coefficients can be further expressed as
$$\alpha_k^\sigma(\E)=\sum_{i\in Q_0}\sum_{j=1}^N \sigma_{ij}\alpha^{L_j}_k(\E_i),$$
where $\alpha^{L_j}_k(\E_i)/k!$ is the coefficient to the monomial $T^k$ in the Hilbert polynomial $P^{L_j}_{\E_i}$, computed with respect to $L_j$.\\

Using these coefficients allows us to introduce the reduced version of the multi--Hilbert polynomial and, later, the slope.
\begin{definition}
The multi--rank of a quiver sheaf $\E$ with respect to $\rho$ is defined as
$$\rk^\rho(\E)=\sum_{i\in Q_0}\sum_{j=1}^N \rho_{ij}\alpha^{L_j}_d(\E_i),$$
where $d$ is the maximum of the dimensions of the sheaves $\E_i$. The reduced multi--Hilbert polynomial of $\E$ with respect to $(\sigma,\rho)$ is defined as
$$p^{(\sigma,\rho)}_\E=\frac{P^\sigma_\E}{\rk^\rho(\E)}.$$
\end{definition}

\begin{definition}
A quiver sheaf $\E$ is called multi--Gieseker semistable with respect to $(\sigma,\rho)$, or simply semistable, if it is pure and for all non--trivial quiver subsheaves $\F\subset \E$ the inequality
$$p^{(\sigma,\rho)}_\F\leq p^{(\sigma,\rho)}_\E$$
holds. If all such inequalities are strict we call $\E$ stable.
\end{definition}

\begin{remark}\label{QS:R:stabcond_rudakov}
This stability condition is a condition in the sense of \cite{rudakov}. Additionally, any quiver sheaf $\E$ is noetherian, allowing us to adopt general results from \cite{rudakov} to our setting.
\end{remark}

There are three special cases of this stability condition.
\begin{enumerate}
\item In case $X=\Spec(k)$, coherent sheaves are just vector spaces $V$, and the Hilbert polynomial is just its dimension. Consequently, quiver sheaves are simply vector space representations $M$ of $Q$. Considering $\Lul=(L)$, where $L$ is the trivial line bundle, $\sigma=\theta$ and $\rho=(1,1,\ldots,1)$, we compute
$$p^{(\sigma,\rho)}(M)=\frac{\sum_{i\in Q_0}\theta_i \dim(M_i)}{\sum_{i\in Q_0}\dim(M_i)}.$$
In other words, $(\sigma,\rho)$--stability is the same as the well--known slope stability for vector space representations of $Q$ as introduced in \cite{king}.
\item If $Q=\bullet$ is the trivial quiver, quiver sheaves are the same as sheaves $E$ on $X$. Considering $\sigma=\rho$, the reduced multi--Hilbert polynomial for a $d$--dimensional sheaf $E$ reads as
$$p^{(\sigma,\sigma)}_E=\frac{\sum_{j=1}^N\sigma_jP^{L_j}_E}{\sum_{j=1}^N\sigma_j \alpha^{L_j}_d(E)}.$$
Thus this stability condition is the one introduced by \cite{grt}.
\item From now on we consider the case of an arbitrary scheme $X$, an arbitrary quiver $Q$, but $\sigma=\rho$. The reduced multi--Hilbert polynomial of a quiver sheaf $\E$ then reads as
$$p^\sigma_\E=\frac{\sum_{i\in Q_0}\sum_{j=1}^N\sigma_{ij}P^{L_j}_{\E_i}}{\sum_{i\in Q_0}\sum_{j=1}^N\sigma_{ij}\alpha^{\sigma}_d(\E_i)}.$$
Of course it is convenient to just write $\sigma$ instead of $(\sigma,\sigma)$ in this case.
\end{enumerate}

Sometimes we need to compare our stability condition to stability conditions for the sheaves at the vertices in the sense of \cite{grt}. 
\begin{definition}
Let $(\Lul,\sigma)$ denote a stability condition for quiver sheaves on $X$ associated to a quiver $Q$. We use the notation
$$\sigma_{i_0}=\left(\sigma_{i_0j}\right)_{j=1,\ldots,N}$$
for the restriction of $\sigma$ to the vertex $i_0\in Q_0$. The tuple $\left(\Lul,\sigma_{i_0}\right)$ then is a stability condition for sheaves on $X$ in the sense of \cite{grt}.
\end{definition}~\\

\subsection{Properties of stability}~\\

From now on we restrict ourselves to the special case of $\rho=\sigma$, and examine some of the properties of such conditions. In this case, the leading coefficient of the reduced multi--Hilbert polynomial is $1/d!$, and thus not very interesting. The second highest order coefficient is of more importance.  

\begin{definition}
The slope of a quiver sheaf $\E$ with respect to a stability condition $\sigma$ is
$$\hmu^\sigma(\E)=\frac{\alpha^\sigma_{d-1}(\E)}{\alpha^\sigma_d(\E)},$$
where $d$ is the degree of the multi--Hilbert polynomial.
\end{definition}

We need the following easy technical result.
\begin{lemma}\label{QS:L:slope_inequality_globalvertex}
Let $\E$ denote a quiver sheaf such that
$$\hmu^\sigma(\E)\geq \mu$$
for some real number $\mu$. Then there exist indices $1\leq j\leq N$ and $i\in Q_0$ such that $\sigma_{ij}\neq 0$ and $\hmu^{L_j}(\E_i)\geq \mu$.

\begin{proof}
With the obvious modifications the elementary proof of \cite{grt} Lemma 2.11 applies.
\end{proof}
\end{lemma}

The reduced multi--Hilbert polynomial can now be written as
$$p^\sigma_\E(T)=\frac{1}{d!}T^d+\frac{\hmu^\sigma(\E)}{(d-1)!}T^{d-1}+\Ocal(T^{d-2}).$$

As an instance of the philosophy mentioned in Remark \ref{QS:R:stabcond_rudakov}, there are variants of the Harder--Narasimhan and Jordan--Hölder filtration for quiver sheaves. Alternatively, one could check that the proofs of the respective results for sheaves (eg. consider Proposition 1.5.2 and Theorem 1.3.4 in \cite{huybrechtslehn}) can be generalized. The Harder--Narasimhan filtration measures to what extent a quiver sheaf fails to be semistable.
\begin{theorem}
For a pure quiver sheaf $\E$ and a stability condition $\sigma$ there is a uniquely determined filtration
$$0=\HN^0(\E)\subsetneq \HN^1(\E)\subsetneq \ldots \subsetneq \HN^l(\E)=\E,$$
called the Harder--Narasimhan filtration of $\E$, such that all the subquotients $\F_{k+1}=\HN^{k+1}(\E)/\HN^k(\E)$ are semistable and such that
$$p^\sigma(\F^1)>p^\sigma(\F^2)>\ldots>p^\sigma(\F^l).$$
Further, $p^\sigma_{\max}(\E)=p^\sigma(\HN^1(\E))$ is maximal among the reduced Hilbert polynomials of quiver subsheaves of $\E$, and $\HN^1(\E)$ is maximal among the quiver subsheaves attaining this limit.
\end{theorem}

If a quiver sheaf is semistable, the Jordan--Hölder filtration describes if and how it fails to be stable.
\begin{theorem}
Suppose that $\E$ is a semistable quiver sheaf with respect to a stability condition $\sigma$. There exists a filtration
$$0=\E^0\subsetneq \E^1\subsetneq \ldots \subsetneq \E^l=\E,$$
called a Jordan--Hölder filtration, such that all $\E^k$ have the same reduced multi--Hilbert polynomial and such that all subquotients are stable.\\
The subquotients in any such filtration are uniquely determined up to isomorphism and permutation.
\end{theorem}

The uniqueness of the subquotients in a Jordan--Hölder filtration implies that the following notion is well--defined.
\begin{definition}
Fix a stability condition $\sigma$ and consider two semistable quiver sheaves $\E$ and $\F$. We define
$$\gr(\E)=\bigoplus_{k=1}^l \E^k/\E^{k-1},$$
where the $\E^k$ are the quiver subsheaves occurring in some Jordan--Hölder filtration of $\E$. We say that $\E$ and $\F$ are $S$--equivalent if
$$\gr(\E)\simeq \gr(\F).$$
\end{definition}

As in the case of sheaves, it is enough to check the inequalities on saturated quiver subsheaves.
\begin{lemma}
A pure quiver sheaf $\E$ on $X$ is semistable with respect to a stability condition $\sigma$ if and only if
$$p^\sigma_\F\leq p^\sigma_\E$$
holds for all nontrivial saturated quiver subsheaves $\F\subset \E$. Similarly, $\E$ is stable if strict inequality holds for all such quiver subsheaves.

\begin{proof}
Consider any quiver subsheaf $\F\subset \E$. By construction, the saturation of $\F$ differs from $\F$ only in codimension one, and is a supersheaf of $\F$. Hence
$$\alpha_d^{L_j}(\F_{i,\sat})=\alpha_d^{L_j}(\F_i) ~\mathrm{and}~ P^{L_j}_{\F_i}\leq P^{L_j}_{\F_{i,sat}}$$
for all $i\in Q_0$ and $j=1,\ldots,N$, where $d$ is the dimension of $\E$. Thus the reduced multi--Hilbert polynomial of $\F$ is always smaller than the polynomial of the saturation.
\end{proof}
\end{lemma}

We can prove a generalized version of a lemma of Grothendieck, which is a central boundedness result in our work. To prepare for the variation of moduli spaces as treated in Section \ref{S:variation}, we formulate this result for a whole set of stability conditions. Consider a subset
$$\Sigma\subset \left(\R^{Q_0\times N}\right)_+$$
which is a closed cone without origin in $\R^{Q_0\times N}$. Clearly, we could also consider subsets which are contained in such a closed cone, or simply a single parameter $\sigma$.
\begin{lemma}\label{QS:L:grothendieck}
Fix a stability condition $\sigma$, integers $p$ and $d$, a real number $\mu$, a topological type $\tau$, and a cone $\Sigma$ as given above. Consider the family $S$ of quiver sheaves $\F$ which arise as quiver subsheaves of some quiver sheaf $\E$ such that
\begin{enumerate}
\item $\E$ is $(p,\Lul)$--regular of topological type $\tau$,
\item $\F$ is saturated in $\E$,
\item $\E$ is pure of dimension $d$,
\item and $\hmu^\sigma(\F)\geq \mu$ for some $\sigma\in \Sigma$.
\end{enumerate}
This family is bounded.

\begin{proof}
We proceed by induction on $|Q_0|$.\\
In the case $|Q_0|=1$ the quiver subsheaves of some quiver sheaf $\E$ are clearly a subfamily of the family of subsheaves, and the boundedness of this family is settled by \cite{grt}, Lemma 4.5, invoking a lemma of Grothendieck (\cite{huybrechtslehn} Lemma 1.7.9). Note that this proof also works in the case of pure sheaves of arbitrary dimension.\\
In the general case, we get by Lemma \ref{QS:L:slope_inequality_globalvertex} that $S$ decomposes into a finite union of families $S_{ij}$ where at least one sheaf satisfies $\hmu^{L_j}(\F_i)\geq \mu$. Thus by the lemma of Grothendieck the family of sheaves $\F_i$ occurring at the vertex $i$ in quiver sheaves contained in $S_{ij}$ is bounded (note that we can equivalently pass to the quotient sheaves since $\F_i$ is saturated in $\E_i$). It suffices to prove that the families of the $\F_{i'}$, $i\neq i'$, occurring in $S_{ij}$, are bounded as well.\\
We concentrate on the case of $S_{1j}$. In particular, the values of $\alpha_e^{L_j}(\F_1)$ are confined to a finite set, and for each of value in this set, the expression
$$\hmu^{\sigma_1}(\F_1)=\frac{\alpha^{\sigma_1}_{d-1}(\F_1)}{\alpha^{\sigma_1}_{d}(\F_1)}$$
defines a well--defined and continuous function in $\sigma\in \Sigma$. Furthermore, this function is invariant under simultaneous scaling, and must thus be bounded; note that, modulo scaling, $\Sigma$ yields a compact set. Hence,
$$\mu'\geq \hmu^{\sigma_1}(\F_1)$$
uniformly in $\F_1$ and $\sigma$ for some real number $\mu'$. Note that by the dimension assumption the $\alpha^{\sigma_i}_d(\F_i)$ are non--negative, so that we have
\begin{align*}
\frac{\sum_{1\neq i\in Q_0}\alpha^{\sigma_i}_{d-1}(\F_i)}{\sum_{1\neq i\in Q_0}\alpha^{\sigma_i}_{d}(\F_i)} & \geq \frac{\mu\left(\sum_{i\in Q_0}\alpha^{\sigma_i}_{d}(\F_i)\right)-\alpha_{d-1}^{\sigma_1}(\F_1)}{\sum_{1\neq i\in Q_0}\alpha^{\sigma_i}_{d}(\F_i)}\\
&\geq \frac{\mu\left(\sum_{1\neq i\in Q_0}\alpha^{\sigma_i}_{d}(\F_i)\right)+(\mu-\mu')\alpha_d^{\sigma_1}(\F_1)}{\sum_{1\neq i\in Q_0}\alpha^{\sigma_i}_{d}(\F_i)}\\
&=\mu + (\mu-\mu')\frac{\alpha_d^{\sigma_1}(\F_1)}{\sum_{1\neq i\in Q_0}\alpha^{\sigma_i}_{d}(\F_i)}.
\end{align*}
The left hand side is the slope of the restriction of $\F$ to the quiver where the vertex $1$, together with all arrows starting and ending in it, are removed. We claim that the right hand side is bounded from below, which would finish the proof by induction. Indeed, the fraction on the right hand side can be estimated as
\begin{align*}
\frac{\alpha^{\sigma_1}_d(\F_1)}{\sum_{1\neq i\in Q_0}\alpha^{\sigma_i}_d(\F_i)}&=\frac{\alpha^{\sigma_1}_d(\F_1)}{\sum_{i\in I}\sum_{j=1}^N\sigma_{ij}\alpha^{L_j}_d(\F_i)}\\
& \leq \frac{\alpha^{\sigma_1}_d(\F_1)}{\sum_{i\in I}\sum_{j=1}^N \sigma_{ij}},
\end{align*}
where $I\subset Q_0$ is the subset of vertices $i\in Q_0$ not equal to $1$ for which $\F_i$ does not vanish. For these vertices, $\alpha^{L_j}_d(\F_i)$ is integral and non--vanishing, giving us the inequality. Now the right hand side is also a well--defined and continuous function which is invariant under scaling, and is hence bounded.
\end{proof}
\end{lemma}

To talk about stability in families we consider
$$\Lul'=\left(L_1',\ldots,L_N'\right),$$
the pullback of $\Lul$ to the space $X\times S$ where the family lives on. Note that $\Lul'$ consists of relatively ample line bundles.

\begin{proposition}\label{QS:P:openness}
Being pure, semistability and stability are open properties for flat families of quiver sheaves.

\begin{proof}
Let $\E$ denote a flat family of quiver sheaves of topological type $\tau$ on $X$ over some noetherian base scheme $S$, and consider the family of relatively ample line bundles $\Lul'$ as in the preceding remark.\\
We know by definition that $\E_s$, where $s\in S$ is a geometric point, is pure if and only if all the sheaves $\E_{s,i}$ for $i\in Q_0$ are pure. Thus openness is implied by the openness of the sheaf version as provided by \cite{huybrechtslehn} Proposition 2.3.1.\\

To prove openness of stability and semistability we also follow the reasoning of \cite{huybrechtslehn}. For the sake of this proof we introduce the notation $\hmu^\sigma(\tau)$ and $p^\sigma(\tau)$ for the slope and reduced Hilbert polynomial fixed by the topological type and the stability condition. Consider the set $A\subset B(X)_\Q^{Q_0}$ of topological types $\tau'$ such that $\hmu^\sigma(\tau')\leq \hmu^\sigma(\tau)$ and such that there exists a geometric point $s\in S$ and a pure quotient $\E_s\to \E'$ with the property $\tau(\E')=\tau'$ and $p^{\sigma}(\tau')<p^\sigma(\tau)$.
By the Grothendieck Lemma for quiver sheaves \ref{QS:L:grothendieck}, the family of quotient quiver sheaves $\E'$ underlying $A$ is bounded, so in particular $A$ is finite.\\
Note that $\E_s$ is semistable if and only if $s$ is contained in the complement of the union of the finitely many closed images of the morphisms
$$\Quot^{\tau'}_{\E/X/S}\to S,$$
where $\tau'\in A$. Here we use the Quiver Quot--scheme as introduced in Section \ref{S:quiverquot}. For stability, we use a similar argument using the inequality $p^{\sigma}(\tau')\leq p^\sigma(\tau)$ instead of strict inequality in the definition of $A$.
\end{proof}
\end{proposition}

A necessary condition for the construction of the moduli space is the boundedness of the family of semistable quiver sheaves. In contrast to the case of Gieseker--semistable sheaves this is not even automatic in the case of a curve. However, in Theorem \ref{T: semistablebounded} we establish a partial result concerning this problem. Again, we think of $\Lul$ and also of $\tau$ as fixed. Additionally, we fix a (possibly empty) set $I$ of relations on $Q$.

\begin{definition}
A subset
$$\Sigma\subset \left(\R^{Q_0\times N}\right)_+$$
is called a bounded set of stability conditions if the family of quiver sheaves $\E$, which are of topological type $\tau$, satisfy the relations $I$, and are $\sigma$--semistable for some $\sigma\in \Sigma$, is bounded. In case that $\Sigma$ contains a single element $\sigma$, we simply say that $\sigma$ is bounded.
\end{definition}

\begin{remark}
Clearly, imposing more relations can only improve on the boundedness of some set $\Sigma$.
\end{remark}

Suppose we have a fixed projective scheme $X$, a stability condition $\left(\Lul,\sigma\right)$, a topological type $\tau$ and a set of relations $I$ on a quiver $Q$.

\begin{definition}
The moduli functor of semistable quiver sheaves
$$\M^\sst=\M^{\sigma-\sst}_\tau(Q,I,X):(Sch/k)^\op\to Sets$$
sends a scheme $S$ to the set of isomorphism classes of families which are flat over $S$ and consist of quiver sheaves on $X$ of topological type $\tau$ which are $\sigma$--semistable and satisfy the relations $I$. There is a similar notion of a moduli functor
$$\M^\st=\M^{\sigma-\st}_\tau(Q,I,X)$$
for stable quiver sheaves.
\end{definition}

Following \cite{simpson}, Section 1, we introduce the notion of moduli spaces.
\begin{definition}
A scheme $M^\sst=M^{\sigma-\sst}_\tau(Q,I,X)$ which corepresents $\M^\sst$ is called the (coarse) moduli space of semistable quiver sheaves of topological type $\tau$. Similarly, the moduli space $M^\st=M^{\sigma-st}_\tau(Q,I,X)$ of stable quiver sheaves is required to corepresent $\M^\st$.\\
If $M^\sst$ or $M^\st$ represent the respective moduli functor, they are called fine moduli spaces.
\end{definition}

\section{The embedding functor}

According to our program of construction, we need a functor that embeds the category of quiver sheaves into the category of representations of a quiver $Q(Q')$, which we need to construct.\\

\subsection{The twisted quiver}~\\

Our embedding functor maps quiver sheaves to representations of the twisted quiver $Q(Q')$, where the twisting, or auxiliar, quiver is given as in \cite{grt}.\\
Fix a projective scheme $X$ and a collection $\Lul=(L_1,\ldots,L_N)$ of ample line bundles on $X$. Furthermore, we fix a topological type $\tau$ of a quiver sheaf and two natural numbers $m>n$.

\begin{definition}
The auxiliary quiver with $N$ rows is defined as follows.
\begin{align*}
Q_0'&=\left\{v_1,\ldots,v_N,w_1,\ldots,w_N\right\}\\
Q_1'&=\left\{\varphi_{kl}:v_k\to w_l \mid k,l=1,\ldots,N\right\}.
\end{align*}
The labels $H_{ij}=H_{\varphi_{ij}}$, dependent on $\Lul,m$ and $n$, attached to the arrows $\varphi_{ij}:v_i\to w_j$ are given as
$$H_{ij}=H^0\left(X,L_i^{-n}\otimes L_j^{m}\right)=\Hom\left(L_i^n,L_j^m\right).$$
\end{definition}

For example, the auxiliary quiver $Q'$ for $N=3$ looks as follows (ignoring the labels).
$$
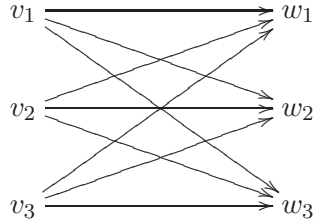
\xymatrix{v_1 \ar[rrr] \ar[rrrd] \ar[rrrdd] &&& w_1 \\
v_2 \ar[rrru] \ar[rrr] \ar[rrrd] &&& w_2 \\
v_3 \ar[rrruu] \ar[rrru] \ar[rrr] &&& w_3}$$

For integers $m>n$ consider the sheaf
$$T=\bigoplus_{j=1}^N L_j^{-n}\oplus L_j^{-m}$$
and the $k$--algebra
$$A'=L\oplus H \subset \End_X(T),$$
where $L$ is generated by the projections onto the summands $L_i^{-n}$ and $L_j^{-m}$ and
$$H=\bigoplus_{i,j=1}^N H_{ij},~ H_{ij}=H^0(X,L_i^{-n}\otimes L_j^m)=\Hom_X(L_j^{-m},L_i^{-n}).$$

The algebra $A'$ is realized as the path algebra of the auxiliary quiver $Q'$ with $N$ rows and labels $H_{ij}$.\\

\begin{definition}
Let $Q$ denote any finite quiver, and consider the auxiliary quiver $Q'$ for some given $N,m$ and $n$ with given labels $H_{ij}$. Then the twisted quiver $Q(Q')$ is defined as follows.
\begin{align*}
Q(Q')_0=&\left\{v_{ij},w_{ij}\mid i\in Q_0, j=1,\ldots,N\right\}\\
Q(Q')_1=&\left\{\varphi_{ikl}:v_{ik}\to w_{il} \mid i\in Q_0, k,l=1,\ldots,N\right\}\\
\cup&\left\{\alpha_{k}^{\links}:v_{ik}\to v_{jk}\mid \left(\alpha:i\to j\right)\in Q_1,k=1,\ldots,N\right\}\\
\cup&\left\{\alpha_{k}^{\rechts}:w_{ik}\to w_{jk}\mid \left(\alpha:i\to j\right)\in Q_1,k=1,\ldots,N\right\}
\end{align*}
The label attached to an arrow $\varphi_{ikl}$ is $H_{kl}=H^0\left(L_k^{-n}\otimes L_l^m\right)$, independent of $i\in Q_0$, and the arrows $\alpha^\links_k$ and $\alpha^\rechts_k$ remain unlabeled.
Further, for any path $\gamma=\alpha_l\ldots\alpha_1$ in $Q$ and any integer $1\leq k\leq N$ we define
$$\gamma^\rechts_k=\left(\alpha_l\right)^\rechts_k\ldots\left(\alpha_1\right)^\rechts_k,$$
and $\gamma^\links_k$ is defined in a similar way. The relations associated to the twisted quiver are then given by
$$I'(I)=I'_1\cup I'_2,$$
where
\begin{align*}
I'_1&=\left\{\alpha^\rechts_l\varphi_{ikl}-\varphi_{jkl}\alpha^{\links}_k \mid \left(\alpha:i\to j\right)\in Q_1;k,l=1,\ldots,N\right\},\\
I'_2&=\left\{\sum_{r=1}^l\lambda_r \left(\gamma_r\right)_k^z\mid \left(\sum_{r=1}^l\lambda_r \gamma_r\right)\in I;1\leq k\leq N;z\in\{\links,\rechts\}\right\}.
\end{align*}
\end{definition}

We think of the twisted quiver as a copy of the auxiliary quiver $Q'$ at each vertex of $Q$, where the arrows in $Q$ are copied for each vertex in $Q'$. The relations $I'_1$ then tell us that in following an arrow in $Q$ and an arrow in $Q'$ the order does not matter (up to changing the copy of the arrow). The relations $I'_2$ tell us that for each vertex in $Q'$ the corresponding copies of paths in $Q$ still satisfy the relations $I$.\\

For instance, let
$$Q=\bullet \xrightarrow{a} \bullet$$
denote the $a$--morphism quiver, and let $Q'$ denote the auxiliary quiver for $N=2$. Then $Q(Q')$, which might as well be interpreted as the quiver $Q'$ doubled, looks as follows.
$$
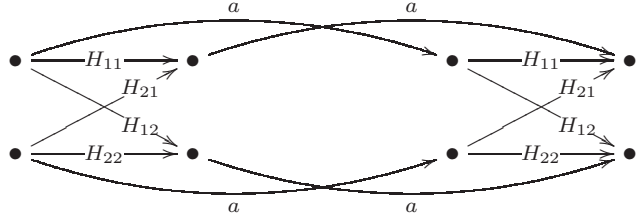
\xymatrix{&&&&&&&&\\
\bullet \ar[rr]|-{H_{11}} \ar@/^15pt/[rrrrr]^{a} \ar[rrd]|-(.75){H_{12}}&& \bullet \ar@/^15pt/[rrrrr]^{a}  &&& \bullet \ar[rr]|-{H_{11}} \ar[rrd]|-(.75){H_{12}} && \bullet\\
\bullet \ar@/_15pt/[rrrrr]_{a}  \ar[rr]|-{H_{22}} \ar[rru]|-(.75){H_{21}}  && \bullet \ar@/_15pt/[rrrrr]_{a}  &&& \bullet  \ar[rr]|-{H_{22}} \ar[rru]|-(.75){H_{21}} && \bullet\\
&&&&&&&}$$

By construction, the oriented cycles in $Q(Q')$ are inherited from $Q$, but not from $Q'$. Hence, we have the following result.

\begin{lemma}\label{EF:L:twisted_cycles}
 If the quiver $Q$ has no oriented cycles, the same holds for the twisted quiver $Q(Q')$.
\end{lemma}

The natural target category of the extension of the embedding functor, which we construct in the next Section \ref{SS:embeddingfunctor}, can be identified with the category of representations of the twisted quiver. 
\begin{lemma}\label{EF:L:target_categories}
There is a canonical identification of categories
$$(Q,I)-\rep_{A'-\mod}\simeq \left(Q(Q'),H,I'(I)\right)-\rep.$$

\begin{proof}
A representation of $(Q(Q'),H)$ in particular consists of a representation of the auxiliary quiver for each vertex in $Q$. Also, for some fixed arrow $\alpha:i\to j$ of $Q$ the arrows $\alpha^\links_k$ and $\alpha^\rechts_k$ together with the relations $I_1'$ comprise the data of a morphism of representations of $(Q',H)$. Moreover, the relations $I_2'$ imply that these morphisms satisfy the relations $I$. The identification of $(Q',H)-\rep$ with $A'-\mod$ thus induces the required equivalence.
\end{proof}
\end{lemma}
~\\

\subsection{The functor}\label{SS:embeddingfunctor}~\\

The functor $\Hom(T,*)$ of \cite{grt} extends to map quiver sheaves to representations of $Q$ in the category of representations of $(Q',H)$, which we identify with representations of $(Q(Q'),H,I'(I))$.\\

Fix a topological type $\tau$ for a quiver sheaf. First, we recall Theorem 5.7 of \cite{grt}.

\begin{theorem}\label{EF:T:fullyfaithful_sheaf}
For $m\gg n\gg 0$, the functor
$$\Hom(T,*):\Coh(X)\to A'-\mod$$
is fully faithful on the full subcategory of $(n,\Lul)$--regular sheaves of topological type $\tau_i$ for some $i\in Q_0$. More precisely, for $E$ in this subcategory the evaluation
$$\eta_E:\Hom(T,E)\otimes_{A'} T\to E$$
is a natural isomorphism, and thus $\Hom(T,*)$ is left adjoint to the functor $*\otimes_{A'} T$.
\end{theorem}

\begin{remark}
Theorem 5.7 of \cite{grt} is stated for one fixed topological type $\tau$. The main reason for this is to ensure the boundedness of the family of sheaves involved. The condition $m\gg n\gg 0$ can be satisfied for all of the finitely many $\tau_i$ simultaneously, providing us with an adjoint pair in this slightly more general case.
\end{remark}

A module in the image of this functor can be written as
$$\Hom(T,E)=\bigoplus_{j=1}^N H^0\left(E\otimes L_j^n\right)\oplus H^0\left(E\otimes L_j^m\right),$$
where the arrow $\varphi_{ij}$ acts via the canonical map
$$H^0\left(E\otimes L_i^n\right)\otimes_k H^0\left(L_i^{-n}\otimes L_j^m\right)\to H^0\left(E\otimes L_j^m\right).$$
Hence, the image of the functor (for a single $\tau$) lies in the full subcategory of representations of dimension vector $d'=d'(\tau,m,n)$ which reads as
$$d'_{v_k}=h^0\left(E\otimes L_k^n\right)=P^{L_k}_E(n),~d'_{w_k}=h^0\left(E\otimes L_k^m\right)=P^{L_k}_E(m).$$
Note that the Hilbert polynomials $P^{L_k}_E$, and hence the dimension vector, are fixed by the topological type $\tau$.\\

By abstract nonsense, using the naturalness of the evaluation, we obtain an extension of Theorem \ref{EF:T:fullyfaithful_sheaf} to quiver sheaves. The target category of the induced embedding functor can be identified with the category of representations of the twisted quiver by Lemma \ref{EF:L:target_categories}.
\begin{theorem}\label{EF:T:fullyfaithful_quiversheaf}
For $m\gg n\gg 0$ the induced functor
$$\Hom\left(T,*\right):(Q,I)-\Coh(X)\to \left(Q(Q'),H,I'(I)\right)-\rep$$
is fully faithful and exact on the full subcategory of $(n,\Lul)$--regular quiver sheaves of topological type $\tau$. More precisely, for a quiver sheaf $\E$ in this subcategory the evaluation
$$\eta_\E:\Hom\left(T,\E\right)\otimes_{A'}T\to \E$$
is a natural isomorphism, and thus $\Hom\left(T,*\right)$ is left adjoint to the functor $*\otimes_{A'}T$. The image of the embedding lies in the full subcategory of representations of dimension vector $d\left(\tau,m,n\right)$.

\begin{proof}
By general nonsense, the functor of Theorem \ref{EF:T:fullyfaithful_sheaf} extends to a functor
$$\Hom\left(T,*\right):(Q,I)-\Coh(X)\to (Q,I)-\rep_{A'-\mod},$$
and the target category is identified via Lemma \ref{EF:L:target_categories}. By definition, the functor is left--exact on $Q-\Coh(X)$, and on the subcategory under consideration it additionally is right--exact as it is a left--adjoint.
\end{proof}
\end{theorem}

For a quiver sheaf $\E$ the decomposition reads as
$$\Hom(T,\E)=\bigoplus_{i\in Q_0}\bigoplus_{j=1}^N H^0\left(\E_i\otimes L_j^n\right)\oplus H^0\left(\E_i\otimes L_j^m\right),$$
where the arrow $\varphi_{ikl}$ acts via the canonical map
$$H^0\left(\E_i\otimes L_k^n\right)\otimes_k H^0\left(L_k^{-n}\otimes L_l^m\right)\to H^0\left(\E_i\otimes L_l^m\right),$$
and the arrows $\alpha_k^\links$ and $\alpha_k^\rechts$ for some arrow $\left(\alpha:i\to j\right)\in Q_1$ act via
$$H^0\left(\E_i\otimes L_k^p\right)\to H^0\left(\E_j\otimes L_k^p\right),$$
where $p=n$ or $p=m$ respectively. Hence, $\Hom(T,\E)$ has dimension vector $d=d(\tau,m,n)$ given as
$$d_{v_{ik}}=h^0\left(\E_i\otimes L_k^n\right)=P^{L_k}_{\E_i}(n),~d_{w_{ik}}=h^0\left(\E_i\otimes L_k^m\right)=P^{L_k}_{\E_i}(m).$$

Investigating the moduli functor requires a version of $\Hom(T,*)$ globalized to a functor between categories of flat families. First, recall the globalized functors in the setting of $A'$--modules and ordinary sheaves as given in \cite{grt}.

\begin{definition}
Let $S$ denote any scheme.\\
Denote by $\flat_{A'}(S)$ the category of $\Ocal_S\otimes {A'}$--modules on $S$ which are locally free as $\Ocal_S$--modules. By $\flat_S(X\times S)$ we denote the category of $\Ocal_{X\times S}$--modules which are flat over $S$.\\
The globalized functors
\begin{align*}
\Hom'(T,*)& :\flat_S\left(X\times S\right)\rightarrow \flat_{A'}(S)\\
*\otimes_{A'}' T& :\flat_{A'}(S) \rightarrow \flat_S(X\times S)
\end{align*}
are defined as
$$\Hom'(T,\Ebb)=(p_S)_* \Hom(p_X^*T,\Ebb),~ \Mbb\otimes_{A'}' T=p_S^* \Mbb\otimes_{\Ocal_{X\times S}\otimes A'} p_X^*T,$$
where $p_S:X\times S\to S$ and $p_X:X\times S\to X$ are the canonical projections.
\end{definition}

The globalized functor $\Hom'(T,*)$ is also fully faithful, as shown in \cite{grt} Proposition 5.8 by reducing to the fibers.
\begin{proposition}
For $m\gg n\gg 0$ the functor $\Hom'(T,*)$ is fully faithful on the full subcategory of $\flat_S(X\times S)$ consisting of flat families of $(n,\Lul)$--regular sheaves of topological type $\tau$.
\end{proposition}

Next, we want to extend these globalized functors to the case of a quiver $Q$ with relations $I$. Let $A$ denote the path algebra of the twisted quiver $(Q(Q'),H,I'(I))$.
\begin{definition}
Denote the category of representations of $Q$ in flat families as
$$\flat_S(X\times S,Q,I)=(Q,I)-\rep_{\flat_s\left(X\times S\right)}.$$
By $\flat_{A}(S)$ we denote the category of $\Ocal_S\otimes {A}$--modules on $S$ which are locally free as $\Ocal_S$--modules.
\end{definition}

\begin{lemma}\label{EF:L:rep(A'-mod)=A-mod}
There is an equivalence of categories
$$(Q,I)-\rep_{\flat_{A'}(S)}\simeq flat_A(S),$$
i.e. the category of representations of $(Q,I)$ in the category $\flat_{A'}(S)$ is equivalent to the category $\flat_A(S)$.

\begin{proof}
This is the same identification as in Lemma \ref{EF:L:target_categories}, but twisted by $\Ocal_S$.
\end{proof}
\end{lemma}

\begin{proposition}\label{EF:P:faithful}
For $m\gg n\gg 0$ the functor $\Hom'(T,*)$ induces a fully faithful functor
$$\Hom'(T,*):\flat_S(X\times S,Q,I)_\tau\to \flat_A(S)$$
from the full subcategory of $(n,\Lul)$--regular representations of topological type $\tau$ in $\flat_S(X\times S)$ to $\flat_A(S)$, which we will also denote by $\Hom'(T,*)$. The corresponding functor induced by $*\otimes'_{A'}T$ will be denoted by $*\otimes'_A T$.

\begin{proof}
The functorial nature of $\Hom'(T,*)$ allows an extension to the category of representations of $Q$. The target of this extension can be identified with $\flat_A(S)$ by Lemma \ref{EF:L:rep(A'-mod)=A-mod}, and it is clear by abstract nonsense that the functor remains fully faithful.
\end{proof}
\end{proposition}

We can give a description of the image of this functor.
\begin{proposition}\label{EF:P:imageofembedding}
For $m\gg n\gg 0$ as in Proposition \ref{EF:P:faithful} and a flat family $\Mbb$ of right--$A$--modules of dimension vector $d=d(\tau,m,n)$ over some scheme $B$ there exists a unique locally closed subscheme $i:B^\imgreg_\tau\subset B$ such that the following conditions hold.
\begin{enumerate}
\item $i^*\Mbb\otimes'_{A'} T$ is a family of $(n,\Lul)$--regular quiver sheaves of topological type $\tau$ and satisfying the relations $I$ on $X$ which is flat over $B^\imgreg_\tau$, and the unit map
$$\eta_{i^*\Mbb}:i^*\Mbb\to \Hom'\left(T,i^*\Mbb\otimes'_{A'} T\right)$$
is an isomorphism.
\item If $\sigma:S\to B$ is a subscheme such that $\sigma^*\Mbb\simeq \Hom'\left(T,\Ebb\right)$ for some family $\Ebb$ of $(n,\Lul)$--regular quiver sheaves of topological type $\tau$ which satisfy the relations $I$ on $X$ which is flat over $S$, then $\sigma$ factors through the embedding morphism $i$ and
$$\Ebb\simeq \sigma^* \Mbb\otimes'_{A'} T.$$
\end{enumerate}

\begin{proof}
The identification of Lemma \ref{EF:L:rep(A'-mod)=A-mod} allows us to consider $\Mbb$ as a representation of $Q$ in the category of flat families of $A'$--modules satisfying the relations $I$. Consider Proposition 5.9 in \cite{grt}, which is analogous to the present Proposition, and take $B^\imgreg_\tau$ to be the intersection of the corresponding closed subschemes provided for these families of $A'$--modules.\\
To see that this subscheme satisfies the intended assertions we can again trace back the identification of Lemma \ref{EF:L:rep(A'-mod)=A-mod} and observe that all assertions may be checked at each vertex of $Q$, since the preservation of the relations is built into the functors. For the second assertion we also note that a morphism factors through a set of locally closed subschemes if and only if it factors through the intersection.
\end{proof}
\end{proposition}

It seems natural to think of the relations as defining a closed subscheme. This is made precise in the following result. Just for the purpose of formulating this corollary, we let $A_0$ denote the path algebra for the case $I=\emptyset$, i.e. the path algebra of the quiver $(Q(Q'),H,I'_1)$.
\begin{corollary}\label{EF:C:imageofembedding_relations}
Suppose that we are given a topological type $\tau$, a set of relations $I$, and integers $m\gg n\gg 0$ such that Proposition \ref{EF:P:faithful} holds for $I=\emptyset$. For a flat family $\Mbb$ of right--$A_0$--modules of dimension vector $d=d(\tau,m,n)$ over some scheme $B$ which restricts to a family $\Mbb'$ of modules which satisfy the relations $I_2'=I'_2(I)$ over some closed subscheme $B'\subset B$ we have
$$(B')_\tau^{\imgreg,I}=B_\tau^{\imgreg}\cap B'.$$
Here, $B_\tau^{\imgreg}\subset B$ is the locally closed subscheme as in Proposition \ref{EF:P:faithful} in a version without relations, and $(B')_\tau^{\imgreg,I}$ is the locally closed subscheme as in Proposition \ref{EF:P:faithful} in a version incorporating the relations $I$.

\begin{proof}
By construction, the restriction of the family $\Mbb$ allows a description
$$\Mbb\mid_{B^{\imgreg}}\simeq \Hom'(T,\Ebb),$$
where $\Ebb$ is a family of quiver sheaves over $B^{\imgreg}=B_\tau^{\imgreg}$. Further restricted to $B_\tau^\imgreg\cap B'$, the modules satisfy the relations $I'_2$, so that, due to the construction of the functor, the quiver sheaves satisfy the relations $I$. This shows that
$$B_\tau^\imgreg\cap B'\subset (B')_\tau^{\imgreg,I}.$$
Conversely, $\Mbb'$ restricted to $(B')^\imgreg=(B')_\tau^{\imgreg,I}$ allows a description
$$\Mbb'\mid_{(B')^\imgreg}=\Hom'(T,\Ebb')$$
for a family of quiver sheaves $\Ebb'$ which satisfy the relations $I$. Over the intersection $(B')^\imgreg\cap B^\imgreg$, the families $\Mbb$ and $\Mbb'$ coincide, so that, by faithfulness of the functor $\Hom'(T,*)$, the families $\Ebb$ and $\Ebb'$ coincide as well. Thus, $\Ebb$ and $\Ebb'$ glue to a flat family of right--$A$--modules over the subscheme
$$(B')^\imgreg \cup B^\imgreg.$$
But by maximality, this subscheme must equal $B^\imgreg$.
\end{proof}
\end{corollary}

Consider the functor
$$\M^\reg(X,Q,I):(Sch/k)^{\op}\to Sets,$$
which assigns to a scheme $S$ the set of isomorphism classes of families of $(n,\Lul)$--regular quiver sheaves on $X$ satisfying the relations $I$ of topological type $\tau$ which are flat over $S$.\\
For brevity denote $B=R_d(Q(Q'),H,I'(I))^\imgreg_\tau$, where we use the tautological family of modules $\Mbb$ as explained in Appendix \ref{SS:quiver_moduli}. Sending a morphism $f:S\to B$ to the family $f^*\left(\Mbb\mid_B\otimes'_A T\right)$ defines a natural transformation
$$g:\underline{B}\to \M^\reg(X,Q,I).$$
As in \cite{ack} Theorem 4.5, the existence of the locally closed subscheme parameterizing the image of the embedding functor implies that this moduli functor is locally a quotient functor.

\begin{proposition}\label{EF:P:local_iso_reg_qsheaves}
There is a local isomorphism of functors
$$\M^\reg(X,Q,I)\simeq \underline{B}/\underline{G},$$
where $G=G_d$. This functor is induced by $g$.

\begin{proof}
The proof of \cite{ack} Theorem 4.5 holds verbatim, where we use our version of the local isomorphism
$$h':\underline{R}/\underline{G}\to \M_A$$
as in Proposition \ref{Q:P:A-modfunctor_quotient} and our Proposition \ref{EF:P:imageofembedding}.
\end{proof}
\end{proposition}
~

\section{Stability under embedding}\label{S:stabilityembedding}

The next step in our program is to make sure that the embedding functor $\Hom(T,*)$ we constructed preserves stability. The technical cornerstone of this is a variant of the Le Potier--Simpson theorem for quiver sheaves and multi--Gieseker stability. We basically follow the reasoning of \cite{grt}, Sections 7 and 8.\\

\subsection{The Le Potier--Simpson theorem}~\\

Suppose we are given a quiver $Q$ together with a (possibly empty) set of relations $I$, a stability condition $(\Lul,\sigma)$ and a topological type $\tau=(\tau_i)_{i\in Q_0}$.\\

We assume that $\sigma$ is a bounded stability condition, i.e. that the family of semistable quiver sheaves with respect to $\sigma$ which satisfy the relations is bounded. Using this we choose $p\gg0$ such that all such quiver sheaves are $(p,\Lul)$--regular. Furthermore, we assume that $\dim(E)=d$ for all sheaves of type $\tau_i$ for some $i\in Q_0$.

\begin{remark}
Basically, the relations $I$ and $I'_2$ are not of great importance in this section. This is due to the fact that the property of satisfying some given relation is inherited by quiver subsheaves, so that the notion of semistability is insensitive to relations. The only relevance these relations have is because of the question of boundedness, which might only hold on the subfamily of semistable quiver sheaves which satisfy some relations.\\
On the other hand, the relations $I'_1$ are essential for the technical step of Lemma \ref{CSS:L:every_subord_tight}.
\end{remark}

Let us recall the Le Potier--Simpson estimate for the dimension of the space of global sections of a sheaf twisted by an ample line bundle. Because we use similar notation, we refer to the formulation in \cite{grt} Theorem 7.1.\\
Here, for a real number $x$ we use the notation $\left[x\right]_+=\max\left(x,0\right)$. 

\begin{theorem}\label{LP:T:lp_estimate}
Let $X$ denote a projective scheme, and $L$ a very ample line bundle on $X$. Let $E$ denote a purely $d$--dimensional sheaf on $X$ and define the number
$$C^L_E=\left(r^L\right)^2+\frac{1}{2}\left(r^L+d\right)-1,$$
where $r^L$ is the rank of $E$ with respect to $L$. Then, for any $n>0$ we have
$$h^0(E\otimes L^n)\leq \frac{r^L-1}{d!}\left[\hmu^L_{max}(E)+C^L_E+n\right]^d_++\frac{1}{d!}\left[\hmu^L(E)+C^L_E+n\right]^d_+.$$
\end{theorem}

To distinguish between strictly semistable and stable quiver sheaves we need more control over the destabilizing quiver subsheaves (compare with Lemma 2.13 in \cite{grt}). Note that boundedness of $\sigma$ is not needed for the following proof.
\begin{lemma}\label{LP:L:destab_saturated}
Suppose that $\F\subset \E$ is a destabilizing quiver subsheaf of a semistable quiver sheaf $\E$ of topological type $\tau$. Then $\F\subset \E$ is saturated and
$$\F\oplus \E/\F$$
is semistable with Hilbert polynomial $p^\sigma_\E$ and of topological type $\tau$.

\begin{proof}
$\F$ has the same Hilbert polynomial as $\E$ and its quiver subsheaves are quiver subsheaves of $\E$. It is thus again semistable. A similar reasoning using quotients shows that $\E/\F$ is semistable with Hilbert polynomial $p^\sigma_\E$ as well, which is a property inherited by the direct sum. The assertion about the topological type is implied by its additivity.\\

To see that $\F$ is saturated consider the inequalities
$$p^\sigma_{\F}(n)=\frac{\sum_{i,j}\sigma_{i,j}h^0\left(\F_i\otimes L_j^n\right)}{r^\sigma(\F)}\leq \frac{\sum_{i,j}\sigma_{i,j}h^0\left(\G_i\otimes L_j^n\right)}{r^\sigma(F)}=p^\sigma_\G(n)\leq p^\sigma_\E(n)$$
for $n\gg 0$, where $\G\subset \E$ is the saturation of $\F$, and the last inequality holds by semistability. The inequalities are thus actually equalities, and in particular for each i$\in Q_0$ we have
$$H^0\left(\F_i\otimes L_j^n\right)=H^0\left(\G_i\otimes L_j^n\right)$$
for at least one index $j=j(i)$ such that $\sigma_{ij}$ does not vanish. For $n\gg 0$ the sheaves $\G_i\otimes L_j^n$ are globally generated, so $\G_i\subset \F_i$ for all $i\in Q_0$, which shows $\F=\G$.
\end{proof}
\end{lemma}

We are ready to proof the Le Potier--Simpson theorem for quiver sheaves (compare with \cite{grt}, Theorem 7.2). Recall that we assume $\sigma$ to be bounded, and $\sigma$--semistable quiver sheaves to be $\left(p,\Lul\right)$--regular.
\begin{theorem}\label{LP:T:lp_theorem}
For an integer $n\gg p \gg 0$ the following assertions are equivalent for a purely $d$--dimensional quiver sheaf $\E$ of topological type $\tau$ which satisfies the relations $I$.
\begin{enumerate}
\item $\E$ is semistable.
\item $\E$ is $(p,\Lul)$--regular and for all quiver subsheaves $\F\subset \E$ we have
$$\frac{\sum_{i\in Q_0}\sum_{j=1}^N\sigma_{ij}h^0(\F_i\otimes L_j^n)}{r^\sigma(\F)}\leq p^\sigma_\E(n).$$
\item $\E$ is $(p,\Lul)$--regular and the above inequality holds for all saturated quiver subsheaves $\F\subset \E$ such that
$$\hmu^\sigma(\F)\geq \hmu^\sigma(\E).$$
\end{enumerate}
\begin{normalsize}
The same statement holds for stable quiver sheaves and strict inequality. Moreover, equality holds for a quiver subsheaf $\F\subset \E$ if and only if $\F$ is destabilizing.
\end{normalsize}

\begin{proof}
The set of $(p,\Lul)$--regular sheaves of topological type $\tau_i$ for some $i\in Q_0$ is bounded. Hence we can bound
$$\hmu^{L_j}_{\max}(F)\leq C_1$$
for all such sheaves $F$ by a constant $C_1>0$. Recall that
$$C^L_E=\frac{1}{2}r^{L_j}(E)\left(r^{L_j}(E)+\dim(E)\right),$$
the constant used in the Le Potier--Simpson estimate, is fixed by the topological type of the sheaf $E$. Hence we may set
$$\Cbar=\max\left(C^{L_j}_{\E_i}\right),$$
where $\E$ is any quiver sheaf of topological type $\tau$.\\
We now claim that there exists a constant $C_2>0$ which satisfies the following assertions.
\begin{enumerate}
\item For all $(p,\Lul)$--regular quiver sheaves $\E$ of topological type $\tau$ we have
$$-\hmu^\sigma(\E)+1\leq C_2.$$
\item For any $(p,\Lul)$--regular quiver sheaf $\E$ of topological type $\tau$, all non--empty subsets $I\subset Q_0$ and all integers $0\leq e_{i,j}\leq \rk^{L_j}(\E_i)$ the inequality
\begin{align*}
&\sum_{i\in I^c,j} \left(\sigma_{ij}e_{ij}C_1\right)+\sum_{i\in I,j} \left(\frac{\sigma_{ij}}{d!}\left(e_{ij}-1\right)\left(C_1+\Cbar\right)+\left(-C_2+\Cbar\right)\right)\\
&\leq \left(\hmu^{\sigma}(\E)-1\right)\sum_{i\in Q_0,j}\sigma_{ij}e_{ij}
\end{align*}
holds, where $I^c=Q_0\setminus I$ denotes the complement, and the sums are simultaneously taken over $j=1,\ldots,N$.
\end{enumerate}
This is possible because of the boundedness of the family of involved quiver sheaves, and because the parameters occurring in the second inequality, besides $C_2$, are either fixed beforehand or only vary within a finite set.\\

Consider the set $S$ of sheaves $F$ which are saturated subsheaves $F\subset E$ of a $(p,\Lul)$--regular sheaf $E$ of topological type $\tau$ and such that $\hmu^{L_j}(F)\geq -C_2$ for some $j$. This set is bounded by the Grothendieck Lemma for sheaves (because the subsheaves are saturated we can equivalently consider the quotients and apply \cite{huybrechtslehn} Lemma 1.7.9).\\

Furthermore, consider the set $S'$ of quiver sheaves $\F$ which are saturated quiver subsheaves $\F\subset \E$ of some $(p,\Lul)$--regular quiver sheaf $\E$ of topological type $\tau$ and such that $\hmu^{\sigma}(F)\geq -C_2$. This set is bounded as well according to the Grothendieck Lemma \ref{QS:L:grothendieck} for quiver sheaves.\\
To shorten notation we introduce the numbers
$$D_1=C_1+\Cbar,~D_2=-C_2+\Cbar.$$

Now we further claim that for $n\gg p$ the following assertions hold.
\begin{enumerate}
\item Given any quiver sheaf $\F$ which consists of sheaves in $S$ or is in $S'$ and any quiver sheaf $\E$ which is $(p,\Lul)$--regular and of topological type $\tau$ we have that
$$p^\sigma_\F(n) \sim p^\sigma_\E(n) \Leftrightarrow p^\sigma_\F \sim p^\sigma_\E,$$
where $\sim$ is one of the relations $=,\leq$ or $<$.
\item All $F\in S$ are $(n,\Lul)$--regular.
\item The quiver sheaves $\F\in S'$ are all $(n,\Lul)$--regular.
\item $n>C_2-\Cbar$
\item For all quiver subsheaves $\F\subset \E$ of some $(p,\Lul)$--regular quiver sheaf $\E$ of topological type $\tau$ and all non--empty subsets $I\subset Q_0$ we have
\begin{align*}
&\sum_{i\in I^c,j}\sigma_{ij}P^{L_j}_{\F_i}(n)+\sum_{i\in I,j}\frac{\sigma_{ij}}{d!}\left(\left(r^{L_j}(\F_i)-1\right)\left(D_1+n\right)^d+\left(D_2+n\right)^d\right)\\
&\leq r^\sigma(\F)\left(p^\sigma_\E(n)-1\right).
\end{align*}
\end{enumerate}
The first three assertions are true because the involved families of quiver sheaves are bounded, and the fourth assertion holds because the right hand side is just a constant. To see that the last assertion holds note that, after bringing $r^\sigma(\F)$ to the other side, the left hand side is a polynomial in $n$ with leading coefficient $\frac{1}{d!}$ and second coefficient
$$\frac{\sum_{i\in I^c,j} \left(\sigma_{ij}\alpha^{L_j}_{d-1}(\F_i)\right)+\sum_{i\in I,j} \left(\frac{\sigma_{ij}}{d!}\left(\rk^{L_j}(\F_i)-1\right)D_1+D_2\right)}{r^\sigma(\F)}.$$
Then we may estimate
$$\alpha^{L_j}_{d-1}(\F_i)=\alpha^{L_j}_d(\F_i)\hmu^{L_j}(\F_i)\leq \rk^{L_j}(\F_i)\hmu^{L_j}_{\max}(\E_i)\leq \rk^{L_j}(\F_i) C_1$$
to arrive at an expression as in the second condition for $C_2$. Thus for large $n$ the inequality of polynomials holds as claimed, and there are only finitely many such conditions.\\

We now prove $1.)\Rightarrow 2.)$. Let $\E$ denote some semistable quiver sheaf satisfying the relations, which is hence $(p,\Lul)$--regular by assumption. Consider a quiver subsheaf $\F\subset \E$. Note that if $\G$ denotes the saturation of $\F$ in $\E$ we have
$$\frac{\sum_{i\in Q_0}\sum_{j=1}^N\sigma_{ij}h^0(\F_i\otimes L_j^n)}{r^\sigma(\F)}\leq \frac{\sum_{i\in Q_0}\sum_{j=1}^N\sigma_{ij}h^0(\G_i\otimes L_j^n)}{r^\sigma(\G)},$$
so we may assume that $\F$ is saturated without loss of generality.\\

Let $I\subset Q_0$ denote the set of vertices such that
$$\hmu^{L_j}(\F_i)< -C_2$$
for all $j=1,\ldots,N$. Again, let $I^c$ denote the complement of this set.\\
Note that for $i\in I^c$ we clearly have $\F_i\in S$, so $\F_i$ is $(n,\Lul)$--regular by the choice of $n$. If $i\in I$ we know
$$\hmu^{L_j}_{\max}(\F_i)\leq \hmu^{L_j}_{\max}(\E_i)\leq C_1$$
and $C^{L_j}_{\F_i}\leq C^{L_j}_{\E_i}\leq \Cbar$. Thus the Le Potier--Simpson estimate (Theorem \ref{LP:T:lp_estimate}) tells us
\begin{align*}
h^0(\F_i\otimes L_j^n)&\leq \frac{r^{L_j}(\F_i)-1}{d!}\left[\hmu^{L_j}_{\max}(\F_i)+C^{L_j}_{\F_i}+n\right]^d_+\\
&+\frac{1}{d!}\left[\hmu^{L_j}(\F_i)+C^{L_j}_{\F_i}+n\right]^d_+\\
&\leq \frac{r^{L_j}(\F_i)-1}{d!}\left(D_1+n\right)^d+\frac{1}{d!}\left(D_2+n\right)^d
\end{align*}
by our choice of constants and $n$.\\

Now we distinguish two cases.\\
If $I=\emptyset$, i.e. $I^c=Q_0$, we have that
$$\frac{\sum_{i\in Q_0}\sum_{j=1}^N\sigma_{ij}h^0(\F_i\otimes L_j^n)}{r^\sigma(\F)}=p^\sigma_\F(n)\leq p^\sigma_\E(n),$$
where the first equality holds because all $\F_i$ are $(n,\Lul)$--regular and the inequality is equivalent to $p^\sigma_\F\leq p^\sigma_\E$, which holds by semistability of $\E$.\\
If $\E$ is stable, the inequality is strict by the same argument.\\

On the other hand, if $I\neq \emptyset$ we may rewrite
\begin{align*}
&{\sum_{i\in Q_0}\sum_{j=1}^N\sigma_{ij}h^0(\F_i\otimes L_j^n)}{}={\sum_{i\in I^c,j}\sigma_{ij}h^0(\F_i\otimes L_j^n)+\sum_{i\in I,j}\sigma_{ij}h^0(\F_i\otimes L_j^n)}{}\\
&\leq \sum_{i\in I^c,j}\sigma_{ij}P^{L_j}_{\F_i}(n)+\sum_{i\in I,j}\frac{\sigma_{ij}}{d!}\left(\left(r^{L_j}(\F_i)-1\right)\left(D_1+n\right)^d+\left(D_2+n\right)^d\right)
\end{align*}
and the right hand side is strictly smaller than $p^\sigma_\E(n)r^\sigma(\F)$ by choice of $n$.\\

That $2.)\Rightarrow 3.)$ holds is very obvious. It thus remains to check the direction $3.)\Rightarrow 1.)$.\\
Let $\E$ denote any $(p,\Lul)$--regular quiver sheaf of topological type $\tau$, and let $\F\subset \E$ denote a quiver subsheaf. Without loss of generality we assume that $\F$ is saturated.\\
If $\hmu^\sigma(\F)<\hmu^\sigma(\E)$ then clearly $\F$ does not destabilize. If on the other hand $\hmu^\sigma(\F)\geq \hmu^\sigma(\E)\geq -C_2$, where the latter inequality holds by choice of $C_2$, we know that $\F$ is contained in $S'$ and is hence $(n,\Lul)$--regular. Thus
$$p^\sigma_\F(n)=\frac{\sum_{i\in Q_0}\sum_{j=1}^N\sigma_{ij}h^0(\F_i\otimes L_j^n)}{r^\sigma(\F)}\leq p^\sigma_\E(n)$$
for $n\gg 0$, which implies $p^\sigma_\F\leq p^\sigma_\E$.\\
If the inequality in $2.)$ is strict the above inequality is also strict, and thus $\E$ is stable.\\

It remains to show the addendum.\\
Consider some quiver subsheaf $\F\subset \E$ and its saturation $\G\subset \E$. Recall the argument of $1.)\Rightarrow 2.)$, and note that most of it does not actually make use of the assumption of $1.)$. Only the case $I=\emptyset$ is relevant though. Because if $I\neq \emptyset$ holds for $\G$, the inequality in $2.)$ is strict, so $\G$ and thus $\F$ can not be destabilizing.\\
If $\F$ is destabilizing it is saturated by Lemma \ref{LP:L:destab_saturated}, and hence consists of sheaves in the family $S$. This implies that the desired equality in $2.)$ is equivalent to the equality $p^\sigma_\F=p^\sigma_\E$. Conversely, if equality holds in $2.)$ we have that
\begin{align*}
p^\sigma_\E(n)=p^\sigma_\F(n)&=\frac{\sum_{i\in Q_0}\sum_{j=1}^N\sigma_{ij}h^0(\F_i\otimes L_j^n)}{r^\sigma(\F)}\\
&\leq \frac{\sum_{i\in Q_0}\sum_{j=1}^N\sigma_{ij}h^0(\G_i\otimes L_j^n)}{r^\sigma(\G)} =p^\sigma_\E(n).
\end{align*}
Note that as $\G$ consists of sheaves in $S$ it in particular is $(n,\Lul)$--regular, so that an argument as in the proof of Lemma \ref{L:L:destab_saturated} shows that $\F$ is saturated. Hence $\F=\G$ and equality in $2.)$ is equivalent to $p^\sigma_\F=p^\sigma_\E$.
\end{proof}

\end{theorem}

Actually, a slight reformulation of the Theorem of Le Potier--Simpson is needed.
\begin{corollary}\label{LP:C:lp_corollary}
For $n\gg p \gg 0$, the following assertions are equivalent for any pure quiver sheaf $\E$ which is of topological type $\tau$ and which satisfies the relations $I$.
\begin{enumerate}
\item $\E$ is semistable.
\item $\E$ is $(p,\Lul)$--regular and for all quiver subsheaves $\F\subset \E$ it holds
$$\sum_{i,j}\sigma_{ij}h^0(\F_i\otimes L_j^n)P^\sigma_\E\leq P^\sigma_\E(n)P^\sigma_\F.$$
\item $\E$ is $(p,\Lul)$--regular and the above inequality holds for all saturated quiver subsheaves $\F\subset \E$ such that $\hmu^\sigma(\F)\geq \hmu^\sigma(\E)$.
\end{enumerate}
Moreover, for semistable $\E$ and destabilizing $\F\subset \E$, equality in $2)$ holds if and only if $\F$ is destabilizing.

\begin{proof}
As in the proof of \cite{grt} Corollary 7.3, we may rewrite the inequality concerning the polynomials $P^\sigma_\E$ and $P^\sigma_\F$ as an inequality concerning the leading coefficients. The reformulated assertions are then equivalent by the Le Potier--Simpson Theorem \ref{LP:T:lp_theorem}.
\end{proof}
\end{corollary}
~\\

\subsection{Semistability under embedding}~\\\label{SS:preserving_semistable}
We now want to show that a quiver sheaf $\E$ is semistable if and only if its image under the embedding functor $\Hom(T,\E)$ is semistable. Recall that we still assume $\sigma$ to be a bounded stability condition given the relations $I$.

To that end, we start by fixing integers $m\gg n\gg p\gg 0$ that satisfy certain technical conditions. Later, we will show that the embedding functor defined by these integers has the desired property of preserving stability.\\
To formulate the technical conditions we first need some definitions.
\begin{definition}
Let $E$ denote any $(n,\Lul)$--regular sheaf of topological type $\tau_i$ for some $i\in Q_0$. Consider the evaluation map
$$\ev_j:H^0(E\otimes L_j^n)\otimes L_j^{-n}\to E \to 0.$$
For any subspace $V_j'\subset H^0(E\otimes L_j^n)$ denote the image and kernel of the map induced by $\ev_j$ as
$$0\to F_j'\to V_j'\otimes L_j^{-n}\to E_j' \to 0.$$
Furthermore, consider
$$0\to K \to \bigoplus_{j=1}^N V_j'\otimes L_j^{-n} \to \sum_{j=1}^N E_j'\to 0.$$
Denote by $S_\ev$ the family of sheaves $E_j',F_j',\sum_{j=1}^N E_j'$ and $K$ that arise in such a way.
\end{definition}

\begin{definition}
Let $S$ denote the family of quiver sheaves $\F$ such that $\F\subset \E$ is a saturated quiver subsheaf, where $\E$ is $(p,\Lul)$--regular of topological type $\tau$, and such that $\hmu^\sigma(\F)\geq \hmu^\sigma(\E)$.
\end{definition}

We claim that there are integers $m\gg n\gg p \gg 0$ satisfying the following assertions, which we will keep fixed in the following arguments.
\begin{enumerate}
\item All semistable quiver sheaves of topological type $\tau$ which satisfy the relations $I$ are $(p,\Lul)$--regular.
\item The Le Potier--Simpson corollary \ref{LP:C:lp_corollary} holds.
\item $L_j^{-n}$ is $(m,\Lul)$--regular for all $j=1,\ldots, N$.
\item All quiver sheaves in $S$ and all sheaves in $S_\ev$ are $(m,\Lul)$--regular.
\item For all quiver sheaves $\E$ of topological type $\tau$, for all integers $c_{ij}$ between $0$ and $P^{L_j}_{\E_i}(n)$ and for all quiver sheaves $\F$ which are in $S$ or consist of sheaves in $S_\ev$ the relation of polynomials
$$P^\sigma_\E\sum_{i\in Q_0}\sum_{j=1}^N \sigma_{ij}c_{ij} \sim P^\sigma_\E(m)P^\sigma_\F$$
holds if and only if the relation
$$P^\sigma_\E(m)\sum_{i\in Q_0}\sum_{j=1}^N \sigma_{ij}c_{ij} \sim P^\sigma_\E(m)P^\sigma_\F(m)$$
holds, where $\sim\in\{=,\leq,<\}$.
\item The functor $\Hom(T,*)$ is an embedding, i.e. Theorem \ref{EF:T:fullyfaithful_quiversheaf} holds.
\end{enumerate}

The first assertion can be achieved because the family of semistable quiver sheaves is bounded by assumption. The second and the sixth assertions hold by the statements of the corollary and the theorem, and the third assertion is easily achieved because there are only finitely many $L_j$. To guarantee the fourth and fifth assertion note that $S$ is bounded by the Grothendieck Lemma \ref{QS:L:grothendieck}, and $S_\ev$ is also bounded (see \cite{grt} Definition 8.3).\\

Recall that the embedding functor $\Hom(T,*)$ maps any $(n,\Lul)$--regular quiver sheaf $\E$ to a representation of the twisted quiver $Q(Q')$ of dimension vector $d$, given by
$$d_{ij1}=h^0(\E_i\otimes L_j^n)=P^{L_j}_{\E_i}(n),~d_{ij2}=h^0(\E_i\otimes L_j^m)=P^{L_j}_{\E_i}(m).$$
For brevity, we let the indices $(ij1)$ and $(ij2)$ refer to the vertices $v_{ij}$ and $w_{ij}$ respectively. We also define a stability condition $\theta=\theta(\sigma,d)$ on $Q(Q')$ by
$$\theta_{ij1}=\frac{\sigma_{ij}}{\sum_{k\in Q_0}\sum_{l=1}^N \sigma_{kl}d_{kl1}},~\theta_{ij2}=\frac{-\sigma_{ij}}{\sum_{k\in Q_0}\sum_{l=1}^N \sigma_{kl}d_{kl2}}.$$

For a representation $M$ of the twisted quiver $Q(Q')$ we use the notation
$$M=\bigoplus_{i\in Q_0}\bigoplus_{j=1}^N V_{ij}\oplus W_{ij},$$
where $V_{ij}$ is the value of $M$ at the vertex $v_{ij}$ and $W_{ij}$ is the value of $M$ at the vertex $w_{ij}$. The maps associated to the arrows $\varphi_{ikl}$ are denoted as
$$\phi_{ikl}:V_{ik}\otimes H_{kl}\to W_{il},$$
and the maps associated to the arrows $\alpha_k^\links$ and $\alpha_k^\rechts$ are denoted as $A_k$ and $B_k$ in a similar fashion. We will also use obvious variants of this notation, for example for some other representation $M'$.\\

Using this notation we may rewrite
$$\theta(M)=\sum_{i,j}\left(\theta_{ij1}\dim \left(V_{ij}\right)+\theta_{ij2}\dim\left(W_{ij}\right)\right).$$
Note that for a representation $M$ of dimension vector $d$ we have $\theta(M)=0$, so that $M$ is semistable if and only if $\theta(N)\leq 0$ holds for all subrepresentations $N\subset M$.\\

To relate stability of representations to stability of quiver sheaves a slightly different notion of the slope is useful.

\begin{definition}
Let $M=\oplus_{i,j}V_{ij}\oplus W_{ij}$ denote a representation of $Q(Q')$ such that $\sum_{i,j}\sigma_{ij}\dim(W_{ij})$ or $\sum_{i,j}\sigma_{ij}\dim(V_{ij})$ are non--zero. Then we define the (auxiliary) slope of $M$ as
$$\mu'(M)=\frac{\sum_{i,j}\sigma_{ij}\dim\left(V_{ij}\right)}{\sum_{i,j}\sigma_{ij}\dim\left(W_{ij}\right)}\in [0,\infty].$$
\end{definition}

\begin{lemma}\label{CSS:L:theta_mu_equivalent}
Let $M$ denote a representation of $Q(Q')$ of dimension vector $d$, and let $M'$ denote a representation such that $\sum_{i,j}\sigma_{ij}\dim\left(W_{ij}'\right)\neq 0$.
\begin{center}
Then $\theta(M')\leq 0$ if and only if $\mu'(M')\leq \mu'(M)$.
\end{center}
The same assertion holds if we replace $\leq$ by $<$.

\begin{proof}
This can be shown by an elementary argument as in the proof of \cite{grt} Lemma 8.6.
\end{proof}
\end{lemma}

There are subrepresentations, called non--degenerate, on which the auxiliary slope is well--defined. For representations in the image of the embedding functor, it suffices to check semistability on such subrepresentations. Degenerate subrepresentations always destabilize. 
\begin{definition}
A representation $M$ of $Q(Q')$ is called degenerate if $V_{ij}=0$ for all $i\in Q_0$ and $j\in \{1,.\ldots,N\}$ and $W_{ij}=0$ for all $i$ and $j$ such that $\sigma_{ij}=0$.
\end{definition}

\begin{lemma}\label{CSS:L:thetasst_musst}
Let $M=\Hom(T,\E)$ denote the representation of $Q(Q')$ of dimension vector $d$ given by a $(n,\Lul)$--regular quiver sheaf $\E$ of topological type $\tau$. Then the following holds.
\begin{enumerate}
\item If a representation $M'$ is non--degenerate we have
$$\sum_{i,j}\sigma_{ij}\dim\left( W_{ij}'\right)\neq 0$$
and $\mu'(M')$ is well--defined.
\item $M$ is $\theta$--semistable if and only if $\mu'(M')\leq \mu' (M)$ for all non--degenerate $M'\subset M$.
\item Suppose that $M$ is $\theta$--semistable. Then $M'\subset M$ is destabilizing if and only if $M'$ is degenerate or $M$ is non--degenerate and $\mu'(M)=\mu'(M')$.
\end{enumerate}

\begin{proof}
Write $M=\Hom(T,\E)=\bigoplus_{i,j}H^0(\E_i\otimes L_j^n)\oplus H^0(\E_i\otimes L_j^m)$. Recall that in this representation the arrows $\varphi_{ikl}$ are equipped with maps
$$\phi_{ikl}:H^0(\E_i\otimes L_k^n)\otimes H^0(L_l^m\otimes L_k^{-n})\to H^0(\E_i\otimes L_l^m).$$
By our assumption on $m$ and $n$ the sheaf $L_l^m\otimes L_k^{-n}$ is globally generated, so as in the proof of \cite{grt} Lemma 8.8, if $\phi_{ikl}(s\otimes h)=0$ for all $h$ then $s=0$. Thus the first assertion can be shown as in the case of sheaves.\\
The rest follows by Lemma \ref{CSS:L:theta_mu_equivalent} once we note that $\theta(M')=0$ if $M'$ is degenerate. Compare with the proof of \cite{grt} Lemma 8.8.
\end{proof}
\end{lemma}

We consider the notion of subordinate subrepresentations and tight representations as in \cite{grt} Definition 8.10.
\begin{definition}
Let $M'$ and $M''$ denote subrepresentations of some representation $M$ of $Q(Q')$.
\begin{enumerate}
\item We say that $M'$ is subordinate to $M''$ if
$$V_{ij}'\subset V_{ij}''~\mathrm{and}~W_{ij}''\subset W_{ij}'$$
holds for all $i\in Q_0$ and all $j=1,\ldots,N$. We denote this by
$$M'\preceq M''.$$
\item A subrepresentation $M'$ is called tight if whenever $M'\preceq M''$ for another subrepresentation of $M$ we have
$$V_{ij}'=V_{ij}'' ~\mathrm{and}~ W_{ij}'=W_{ij}''$$
for all indices $i$ and $j$ such that $\sigma_{ij}\neq 0$.
\end{enumerate}
\end{definition}

\begin{lemma} \label{CSS:L:subordinate_smallerslope}
Let $M'$ and $M''$ denote subrepresentations of some representation $M$ such that $\mu'(M')$ and $\mu'(M'')$ are well--defined. If $M'\preceq M''$ we have
$$\mu'(M')\leq \mu'(M'').$$
Moreover, if $M'$ is tight equality holds.

\begin{proof}
This is elementary. Compare with \cite{grt} Lemma 8.11.
\end{proof}
\end{lemma}

Interestingly, the next lemma needs the relations $I'_1$ on the twisted quiver.
\begin{lemma}\label{CSS:L:every_subord_tight}
Let $M$ satisfy the relations $I'_1$, and consider any subrepresentation $M'\subset M$. Then
$$M'\preceq M''$$
for a tight subrepresentation $M''\subset M$.

\begin{proof}
Adapting the proof of \cite{grt} Lemma 8.12 we define $M''$ by
\begin{align*}
W_{ij}''&=\sum_{k=1}^N \phi_{ikj}(V'_{ik}\otimes H_{kj})\\
V_{ij}''&=\left\{v\in V_{ij} \mid \phi_{ijk}(v\otimes h)\in W_{ik}'' \mathrm{~for~all~}k \mathrm{~and~for~all~}h\in H_{jk} \right\}.
\end{align*}
It is immediately clear that these subspaces get respected by the maps $\phi_{ikl}$, and because of the relations $I'_1$ they are also respected by the maps $A_k$ and $B_k$ and thus define a subrepresentation. Following the remainder of the elementary proof of \cite{grt} Lemma 8.12 we can see that $M'$ is subordinated to $M''$ and that $M''$ is tight.
\end{proof}
\end{lemma}

Putting the results obtained so far together, we arrive at the following criterion for semistability of a representation.
\begin{lemma}
Let $M$ denote a representation of $Q(Q')$ of dimension vector $d$ which satisfies the relations $I'_1$. Then $M$ is semistable if and only if
$$\mu'(M')\leq \mu'(M)$$
for all tight non--degenerate subrepresentations $M'\subset M$.

\begin{proof}
This works just as \cite{grt} Lemma 8.13:
Using Lemma \ref{CSS:L:thetasst_musst} it suffices to check the claimed inequality for non--degenerate subrepresentations. Now combining Lemma \ref{CSS:L:every_subord_tight} with Lemma \ref{CSS:L:subordinate_smallerslope} gives the result.
\end{proof}
\end{lemma}

We need to compare subrepresentations and quiver subsheaves. The key construction for this is given by the following Lemma.
\begin{lemma}\label{CSS:L:sumsarequiversubsheaf}
Let $\E$ denote an $(n,\Lul)$--regular quiver sheaf of topological type $\tau$. Given a subrepresentation $M'\subset M=\Hom(T,\E)$ we consider the subsheaves
$$\E_i'=\sum_{j=1}^N (\E_i)'_j,$$
where $(\E_i)'_j=\ev_j(V_{ij}'\otimes L_j^{-n})$ is given as the image of the evaluation. These subsheaves are respected by the morphisms $\E_\alpha:\E_i\to \E_j$ and hence form a quiver subsheaf $\E'=\E'(M')\subset \E$.

\begin{proof}
Choose any arrow $\alpha:a\to b$ in $Q$. The induced maps
$$H^0(\E_\alpha\otimes L_j^{-n}):H^0(\E_a\otimes L_j^{-n})\to H^0(\E_b\otimes L_j^{-n})$$
map $V_{aj}'$ to $V_{bj}'$ because $M'$ is a subrepresentation by assumption. Hence the diagram
$$\xymatrix{V_{aj}'\otimes L_j^{-n}\ar[rr] \ar[d] && (\E_a)_j' \ar[d]\\
V_{bj}'\otimes L_j^{-n} \ar[rr]  && (\E_b)_j'}$$
commutes because morphisms of sheaves commute with restriction. Summing up over $j$ gives the result.
\end{proof}
\end{lemma}

\begin{proposition}\label{CSS:P:subreps_subord_sumsheaves}
Suppose that $\E$ is an $(n,\Lul)$--regular quiver sheaf of topological type $\tau$ and $M'\subset M=\Hom(T,\E)$ is a subrepresentation. Consider the quiver subsheaf
$$\E'\subset \E$$
as in Lemma \ref{CSS:L:sumsarequiversubsheaf}. Then
$$M'\preceq \Hom(T,\E').$$
If $M'$ is tight and non--degenerate $M'$ satisfies the equality
$$\mu'\left(M'\right)=\mu'\left(\Hom\left(T,\E'\right)\right).$$

\begin{proof}
The proof of \cite{grt} Proposition 8.14 applies to each vertex. This is sufficient to obtain a quiver sheaf version.\\
This remark also applies to the case of a tight subrepresentation, as the proof of \cite{grt} shows that $\Hom\left(T,\E'\right)$ is non--degenerate and we can apply Lemma \ref{CSS:L:subordinate_smallerslope}.
\end{proof}
\end{proposition}

\begin{lemma}
Let $\E$ denote an $(n,\Lul)$--regular quiver sheaf of topological type $\tau$. Then the following assertions are equivalent.
\begin{enumerate}
\item $\Hom(T,\E)$ is semistable.
\item For all quiver subsheaves $\F\subset \E$ we have
$$\sum_{i\in Q_0}\sum_{j=1}^N h^0(\F_i\otimes L_j^n)P^\sigma_\E(m)\leq \sum_{i\in Q_0}\sum_{j=1}^N \sigma_{ij}h^0(\F_i\otimes L_j^m)P^\sigma_\E(n).$$
\item The above inequality holds for all quiver subsheaves $\E'\subset \E$ of the form as in Lemma \ref{CSS:L:sumsarequiversubsheaf} for subrepresentations $M'\subset \Hom(T,\E)$.
\end{enumerate}

\begin{proof}
The proof of \cite{grt} Lemma 8.15 applies:\\
The inequality in the second assertions is equivalent to $\theta(\Hom(T,\F))\leq 0$ after unwrapping the definitions. By Proposition \ref{CSS:P:subreps_subord_sumsheaves} it suffices to check this inequality on subrepresentations of the form $\Hom(T,\E')$.
\end{proof}
\end{lemma}

Finally, we are ready to prove that semistability is respected by the embedding functor.
\begin{theorem}\label{CSS:T:semistability}
Let $\E$ denote a quiver sheaf of topological type $\tau$ which satisfies the relations $I$. Then $\E$ is semistable if and only if $\E$ is pure, $(p,\Lul)$--regular and $\Hom(T,\E)$ is semistable.

\begin{proof}
With the appropriate modifications this can be proven in the same way as the case of sheaves (\cite{grt} Theorem 8.16) by combining the preceding results.
\end{proof}
\end{theorem}
~

\subsection{$S$--equivalence under embedding}~\\

In this subsection, we strengthen the comparison result established in the preceding subsection. For this to work, we need to get rid of degenerate subrepresentations. One way to ensure this is to assume that $\sigma$ is positive, i.e. all its entries are strictly positive. We remind ourselves that we still assume $\sigma$ to be bounded given a set of relations $I$, and semistable quiver sheaves which satisfy these relations to be $\left(p,\Lul\right)$--regular.\\

Now that we exclude degenerate subrepresentations of $\Hom(T,\E)$, the only remaining destabilizing subrepresentations are given by destabilizing quiver subsheaves.
\begin{lemma}
Let $\E$ denote a semistable quiver sheaf which satisfies the relations $I$ for a positive stability parameter $\sigma$. Any destabilizing quiver subsheaf $\F\subset \E$ is $\left(p,\Lul\right)$--regular and $\Hom\left(T,\F\right)\subset\Hom\left(T,\E\right)$ is destabilizing as well.

\begin{proof}
By Lemma \ref{LP:L:destab_saturated}, $\E'$ is $(p,\Lul)$--regular since it is a direct summand
$$\E'\subset \E'\oplus \E/\E'$$
of a semistable quiver sheaf of topological type $\tau$ satisfying the relations $I$, which is thus $\left(p,\Lul\right)$--regular by our assumption. In particular, it is $\left(m,\Lul\right)$--regular and $\left(n,\Lul\right)$--regular as well. This yields the equations
\begin{align*}
\mu'\left(\Hom\left(T,\F\right)\right)&=\frac{\sum_{ij}\sigma_{ij}h^0\left(\F_i\otimes L_j^n\right)}{\sum_{i,j}\sigma_{ij}h^0\left(\F_i\otimes L_j^m\right)}=\frac{P^\sigma_\F(n)}{P^\sigma_\F(m)}=\frac{p^\sigma_\F(n)}{p^\sigma_\F(m)}\\
&=\frac{p^\sigma_\E(n)}{p^\sigma_\E(m)}=\mu'\left(\Hom\left(T,\E\right)\right),
\end{align*}
which proves that $\Hom\left(T,\F\right)$ is destabilizing.
\end{proof}
\end{lemma}

\begin{lemma}
Suppose that $M$ is a $\theta$--semistable representation, where $\theta$ is defined using a positive stability condition $\sigma$. A destabilizing subrepresentation $M'\subset M$ is then tight.

\begin{proof}
The elementary proof of \cite{grt} Lemma 8.18 applies word for word.
\end{proof}
\end{lemma}

We are now ready to prove that subrepresentations of $\Hom(T,\E)$ which are destabilizing can be deduced from destabilizing quiver subsheaves.
\begin{lemma}
Suppose that $\sigma$ denotes a positive stability condition and let $\E$ be a semistable quiver sheaf of topological type $\tau$ which satisfies the relations $I$. For any destabilizing subrepresentation
$$M'\subset \Hom\left(T,\E\right)$$
the quiver subsheaf $\E'\subset \E$ given as in Proposition \ref{CSS:P:subreps_subord_sumsheaves} is either equal to $\E$ or destabilizing as well.

\begin{proof}
Consider our version of the Le Potier--Simpson theorem \ref{LP:T:lp_theorem} and Condition 5 on $m$ and $n$ for quiver sheaves. The proof of Lemma 8.19 in \cite{grt} then carries over.
\end{proof}
\end{lemma}

We end this subsection with the desired result.
\begin{theorem}\label{CSS:T:S-equivalence}
Suppose that $\sigma$ is a positive stability condition. For semistable quiver sheaves $\E$ of topological type $\tau$ which satisfy the relations $I$ the identity
$$\Hom\left(T,\gr\left(\E\right)\right)\simeq \gr\left(\Hom\left(T,\E\right)\right)$$
holds. Hence, $\Hom\left(T,*\right)$ respects $S$--equivalence and $\E$ is stable if and only if $\Hom\left(T,\E\right)$ is.

\begin{proof}
Essentially, the proof of \cite{grt} Theorem 8.20 applies, which we repeat here as a sketch. Consider a Jordan--H\"older filtration
$$0=\E^0\subsetneq \E^1\subsetneq\ldots\subsetneq \E^l=\E$$
in the category of quiver sheaves. That is, the quiver subsheaves $\E^i$ are destabilizing and the filtration is maximal with this property. Let $M^i=\Hom\left(T,\E^i\right)$ and $M=\Hom\left(T,\E\right)$ denote the image of the filtration under the embedding, so that
$$0=M^0\subsetneq M^1\subsetneq \ldots \subsetneq M^l=M.$$
If we assume that this filtration allows a refinement, i.e. $M^p\subset M'\subset M^{p+1}$ for a destabilizing subrepresentation $M'\subset M$, we get that $\E^p\subset \E'\subset \E^{p+1}$ by an argument as in \cite{grt}, where $\E'$ is given as in Lemma \ref{CSS:L:sumsarequiversubsheaf}. But this contradicts the maximality of the Jordan--Hölder filtration of $\E$.\\
The exactness of the embedding functor now implies that it respects $S$--equivalence since
\begin{align*}
\gr\left(\Hom\left(T,\E\right)\right)&=\bigoplus_{k=1}^l\Hom(T,\E^k)/\Hom(T,\E^{k-1})\\
&=\Hom\left(T,\bigoplus_{k=1}^l \E^k/\E^{k-1}\right)=\Hom(T,\gr(\E)).
\end{align*}
This also implies that $\Hom(T,*)$ can distinguish stability from semistability; because $\E$ is stable if and only if $\gr(\E)=\E$ and $\Hom(T,\E)$ is stable if and only if $\gr\left(\Hom(T,\E)\right)=\Hom(T,\E)$.
\end{proof}
\end{theorem}
~

\section{Construction of the moduli space}\label{S:QSconstruction}

The actual construction of the moduli space is the final step in our program, which can be summarized as follows. The existence of the tautological family $\Mbb$ on the representation variety gives a subscheme which is the image of the embedding functor, and openness of semistability provides an open subscheme of it. Finally, the GIT quotient of the representation variety descends to this subscheme because stability of representations and quiver sheaves are compatible, and the quotient thus obtained is the moduli space of quiver sheaves.\\

Suppose that we are given a bounded and positive stability parameter $\left(\Lul,\sigma\right)$ for some fixed set of relations $I$. We fix integers $m\gg n \gg p\gg 0$ such that the following assertions hold.
\begin{enumerate}
\item All $\sigma$--semistable quiver sheaves which satisfy the relations $I$ are $\left(p,\Lul\right)$--regular.
\item The functor $\Hom\left(T,*\right)$ is an embedding, i.e. Theorem \ref{EF:T:fullyfaithful_quiversheaf} holds.
\item Stability is preserved by $\Hom\left(T,*\right)$, i.e. Theorem \ref{CSS:T:semistability} holds.
\item $S$--equivalence is preserved as well, i.e. Theorem \ref{CSS:T:S-equivalence} holds.
\end{enumerate}

Consider the representation variety
$$R=R_d\left(Q(Q'),H,I'(I)\right)$$
for the twisted quiver of dimension vector $d=d(\tau,m,n)$ and recall that we have a group action of $G=G_d$ on it as well as a stability condition $\theta=\theta(d,\sigma)$. The tautological family $\Mbb$ of right--$A$--modules, where $A$ is the path algebra of the twisted quiver with labels and relations, gives us a locally closed subscheme
$$i:B=R^{\imgreg}_\tau \subset R$$
according to Proposition \ref{EF:P:imageofembedding}. In particular, the fibers of the family
$$i^*\Mbb\otimes'_{A'} T$$
are $(n,\Lul)$--regular quiver sheaves of topological type $\tau$. Recall (Proposition \ref{EF:P:local_iso_reg_qsheaves}) that the moduli functor $\M^\reg(X,Q,I)$ of $(n,\Lul)$--regular quiver sheaves on $X$ is locally isomorphic to the quotient functor $\underline{B}/\underline{G}$. The openness of regularity, stability and semistability (Proposition \ref{QS:P:openness}) provides us with open subschemes
$$B^\imgst\subset B^\imgsst\subset B^{\imgreg}\subset B,$$
where the members of the family $i^*\Mbb\otimes'_{A'} T$ are stable, semistable and $(p,\Lul)$--regular respectively.

\begin{proposition}
There is a local isomorphism of functors
$$\M_\tau^{\sst}(Q,X,I)\simeq \underline{B^\imgsst}/\underline{G},$$
and a similar local isomorphism for the case of stable quiver sheaves.

\begin{proof}
We have a natural transformation
$$\underline{B^\imgsst}/\underline{G}\to \M_\tau^\sst(X,Q,I),$$
given by pulling back the restriction of the family $i^*\Mbb\otimes'_{A'} T$ to $B^\imgsst$. As in the proof of \cite{ack} Theorem 6.1, this is the restriction of the natural transformation which induced the local isomorphism of Proposition \ref{EF:P:local_iso_reg_qsheaves}.\\
The same argument applies to the stable case.
\end{proof}
\end{proposition}

\begin{proposition}\label{C:P:moduli_as_quotient}
There exists a commuting diagram
$$\xymatrix{B^{[\sigma-\sst]}\ar@{^(->}[r] \ar[d]_{q'} & R^{\theta-sst} \ar[d]^{q} \\
K \ar@{^(->}[r]_j & M^{\theta-sst}_d,}$$
where $M^{\theta-\sst}_d=M^{\theta-\sst}_d(Q(Q'),H,I'(I))$, $q$ is the GIT quotient as in Theorem \ref{Q:T:quivermoduli}, the restriction $q'$ is a good quotient and $j$ is the embedding of a locally closed subscheme.

\begin{proof}
Let $Y$ denote $B^\imgsst$. By definition, the pullback of the tautological family $\Mbb$ to $Y$ satisfies
$$i^*\Mbb\simeq \Hom'\left(T,\Ebb\right)$$
for a family $\Ebb$ of $(p,\Lul)$--regular semistable quiver sheaves of topological type $\tau$ which satisfy the relations $I$.\\
Consider $Z=\overline{Y}\cap R^{\theta-sst}$, where $\overline{Y}$ is the closure in $R$. This is a closed and $G$--invariant subscheme in $R^{\theta-sst}$. By our assumption that $k$ has characteristic zero, we know that $q$ induces a good quotient on $Z$, and it thus suffices to show that $Y$ lies in $Z$ and that $Y\subset Z$ satisfies the assumptions of Lemma \ref{C:L:goodquotient_opensubset}. To see the inclusion, note that for a point $x\in Y$ the fiber
$$\Mbb_x\simeq \Hom(T,\E)$$
is $\theta$--semistable due to Theorem \ref{CSS:T:semistability}. A similar argument yields the technical condition involving the closed orbits in the closure. Indeed, as a point $x\in Q^\imgsst$ corresponds to some representation $M=\Hom(T,\E)$ for a semistable quiver sheaf $\E$, the closed orbit in the closure of the orbit of $x$ corresponds to the module
$$\gr(M)=\gr\left(\Hom(T,\E)\right)\simeq \Hom\left(T,\gr(\E)\right),$$
where the latter isomorphism is given by Theorem \ref{CSS:T:S-equivalence}. Since $\gr(E)$ is again semistable of topological type $\tau$, the closed orbit is also contained in the image of the semistable quiver sheaves under the embedding functor.
\end{proof}
\end{proposition}

We repeat the technical lemma used in the preceding theorem as given in \cite{ack}, Lemma 6.2.
\begin{lemma}\label{C:L:goodquotient_opensubset}
Suppose that the action of a reductive group $G$ on a scheme $Z$ admits a good quotient $\pi:Z\to Z//G$. Further, suppose that $Y\subset Z$ denotes an open and $G$--invariant subset such that for each orbit $\Ocal\subset Y$ the unique closed orbit $\Ocal'$ contained in the closure $\overline{\Ocal}$ is also contained in $Y$. Then, $\pi$ restricts to a good quotient
$$\pi\mid_Y:Y\to \pi(Y)$$
of $Y$, where $\pi(Y)\subset Z//G$ is open.
\end{lemma}

\begin{theorem} \label{T: existmoduli}
The scheme $K$ as in Proposition \ref{C:P:moduli_as_quotient} is the coarse moduli space for semistable quiver sheaves of topological type $\tau$ which satisfy the relations $I$. The closed points of
$$K=M^{\sigma-sst}_\tau(X,Q,I)$$
correspond to $S$--equivalence classes of semistable quiver shaves. Furthermore, there exists an open subscheme
$$M^{\sigma-\st}_\tau(X,Q,I)\subset M^{\sigma-\sst}_\tau(X,Q,I)$$
which is the coarse moduli space for stable quiver sheaves of topological type $\tau$. Its closed points correspond to isomorphism classes of stable quiver sheaves.

\begin{proof}
Since $q: B^\imgsst \to K$ is a good quotient, the corresponding natural transformation
$$\underline{q}:\underline{B^\imgsst}/\underline{G}\to \underline{K}$$
corepresents the quotient functor $\underline{B^\imgsst}/\underline{G}$ (compare with the remark to Definition \cite{ack}, Definition 4.6). But this quotient functor is locally isomorphic to the moduli functor of semistable quiver sheaves by Proposition \ref{EF:P:local_iso_reg_qsheaves}. Hence $M$ also corepresents this moduli functor (where we use \cite{ack}, Lemma 4.7).\\

The closed points of $K$ correspond to closed orbits in $B^\imgsst$. These orbits further correspond to $S$--equivalence classes of $A$--modules (by the GIT--construction of \cite{king}, transferred to the labeled case), which are of the form $\Hom(T,\E)$ for some semistable quiver sheaf $\E$ of topological type $\tau$, as explained in the proof of Proposition \ref{C:P:moduli_as_quotient}. We also know by Theorem \ref{CSS:T:S-equivalence} that $S$--equivalence classes of such modules correspond to $S$--equivalence classes of semistable quiver sheaves.\\

A semistable quiver sheaf $\E$ is stable if and only if the $A$--module $\Hom(T,\E)$ is stable. Hence
$$B^\imgst=B^\imgsst\cap R^{\theta-\st}.$$
The geometric quotient of $R^{\theta-st}$ then restricts to a geometric quotient $B^\imgst\to q(B^\imgst)=K'\subset K$, which has open image inside $K$.\\
Repeating the argument above we see that $K'$ corepresents the moduli functor of stable quiver sheaves, and that its closed points correspond to $S$--equivalence classes of stable quiver sheaves. But $S$--equivalence reduces to isomorphism for stable objects.
\end{proof}
\end{theorem}

\begin{corollary}\label{C:C:relations_closed}
The moduli space of semistable quiver sheaves satisfying the relations
$$M^\sst_\tau(X,Q,I)\subset M^\sst_\tau(X,Q)$$
is a closed subscheme of the moduli space without relations, and the same assertion holds for the moduli space of stable quiver sheaves.

\begin{proof}
The representation variety with relations $I'(I)=I'_1\cup I'_2$ is a closed subvariety
$$R_d(Q(Q'),H,I'(I))\subset R_d(Q(Q'),H,I'_1),$$
and the tautological family $\Mbb$ of right--$A$--modules on $R=R_d(Q(Q'),H,I'_1)$ restricts to a family $\Mbb'$ of modules additionally satisfying the relations $I'_2$ on $R'=R_d(Q(Q'),H,I'(I))$. According to Corollary \ref{EF:C:imageofembedding_relations} we thus have
$$(R')^{\imgreg}=R'\cap R^\imgreg,$$
i.e. $(R')^\imgreg$ is a closed subscheme of $R^\imgreg$. Further, the property of stability and semistability is insensitive to the question whether relations are imposed, so that the corresponding open subschemes parametrizing stable and semistable quiver sheaves are closed inside the corresponding open subschemes without the relations $I'_2$. This inclusion is respected by the quotient.
\end{proof}
\end{corollary}
~

\section{Projectivity of the moduli space}

In this section we want to show that the moduli space of quiver sheaves is projective, if the quiver $Q$ does not contain oriented cycles. In face of the similar result for moduli spaces of representations (compare with Theorem \ref{Q:T:quivermoduli}) it seems unlikely that one can get rid of this condition. Furthermore, we require that $\sigma$ is a rational stability condition for a crucial technical step of Theorem \ref{L:T:langton} below.\\
Since the moduli space with relations is a closed subscheme of the moduli space without relations (consider Corollary \ref{C:C:relations_closed}), it is sufficient to show the projectivity for the case without relations, up to the question of boundedness. We consider the case of a smooth and projective scheme $X$.

\begin{theorem}
Assume that $Q$ does not contain oriented cycles, and that $\sigma$ is rational. Then, the moduli space of multi--Gieseker semistable quiver sheaves of topological type $\tau$ is projective over $k$.

\begin{proof}
Consider the proof of \cite{ack} Proposition 6.6, which holds almost word for word in our setting. Nevertheless, we give a brief sketch to point out the relevant references.\\
Let $R$ denote a discrete valuation ring over $k$ with field of fractions $K$. The valuative criterion then requires us to construct an extension
$$x:\Spec(R)\to M^\sst$$
for any given morphism $x_0:\Spec(K)\to M^\sst=M^\sst_\tau(Q,X)$. The identification of the moduli space with a quotient as in Proposition \ref{C:P:moduli_as_quotient} allows a lift
$$\xymatrix{\Spec(K') \ar[d] \ar[r]^{y_0}& B^\sst \ar[d]\\
\Spec(K) \ar[r]_{x_0} & M^\sst,}$$
where $K\subset K'$ is a suitable finite extension of fields. Denote by $R'$ a discrete valuation ring with field of fractions $K'$ which dominates $R$.\\
The remainder of the proof then reduces, modulo borrowing arguments from \cite{ack}, to showing that a flat family over $\Spec(K')$ of semistable quiver sheaves of topological type $\tau$ on $X$ extends to a flat family of such quiver sheaves over $\Spec(R')$.\\
By Proposition \ref{P:P:extension_dvr} we can extend the family to a flat one since $Q$ does not contain oriented cycles. Our version of Langton's theorem \ref{L:T:langton} then allows to modify the family to obtain a family of semistable quiver sheaves as needed.
\end{proof}
\end{theorem}

\begin{remark}
For the strategy of proof presented here, both the condition that $Q$ contains no cycles, and the condition that $\sigma$ is rational, are essential. Interestingly, for extensions of families of quiver sheaves (as in Subsection \ref{SS:extensions}), the rationality of $\sigma$ plays no role, while for the proof of Langton's theorem (consider Appendix \ref{S:langton}), we may allow cycles in $Q$.
\end{remark}
~

\subsection{Extensions}~\label{SS:extensions}\\

In this subsection we construct extensions of flat families of quiver sheaves from the base $\Spec(K)$ to $\Spec(R)$, where $K$ is the quotient field of a discrete valuation ring $R$. A condition for this to work is that the quiver does not contain oriented cycles.\\

Let $U\subset X$ denote an open subset of a noetherian scheme over $k$ with inclusion map $i$. It is well--known that coherent sheaves on $U$ can be extended to coherent sheaves on $X$, but we need to ensure that morphisms extend as well. For the case of a quiver without oriented cycles we can show that this even holds for quiver sheaves.

\begin{proposition}\label{P:P:extension_opensubset}
Consider a quiver sheaf $\E$ on $U$ over a quiver $Q$ without oriented cycles. There exists a quiver sheaf $\E'$ on $X$ such that 
$$\E'\mid_U=\E.$$
This extension is given as a quiver subsheaf of the quasi coherent quiver sheaf $i_*\E$.

\begin{proof}
We recall the construction of the extension for a single coherent sheaf $F$ on $U$. As a first step, suppose that $X=\Spec(A)$ is affine. The push forward under the inclusion map $i$ is quasi coherent and hence given by an $A$--module
$$i_*F=\widetilde{M_F}.$$
The module $M_F$ can be written n as a projective limit
$$M_F=\sum_N N=\overrightarrow{\lim}_N N,$$
where the sum and limit are taken over the projective system of finitely generated submodules $N\subset M_F$. Both the restriction functor and the tilde functor preserve colimits, which yields
$$F=i_*F\mid_U=\overrightarrow{\lim}_N \widetilde{N}\mid_U=\sum_N \widetilde{N}.$$
But this sum is equal to one of the summands because $F$ is coherent, so that $F=\widetilde{N^0}\mid_U$. In other words $F'=\widetilde{N^0}$ is the desired extension.\\
For the general case take an affine open covering $X=V_1\cup\ldots\cup V_n$ and use induction on $n$. If an extension $F'$ is already given on $X_n=V_1\cup\ldots\cup V_n$, that is $F'|_{X_n\cap U}=F|_{X_n\cap U}$, we apply the above construction to the open subset $(X_n\cup U)\cap V_{n+1} \subset V_{n+1}$ to obtain an extension to $V_{n+1}$.\\

To construct the extension of a quiver sheaf we recursively apply a slight variation of this argument. It is easy to see that the same argument for the reduction step to an affine scheme also holds in the quiver sheaf setting; hence assume $X=\Spec(A)$.\\
If $i\in Q_0$ is a source of the quiver, we use the construction to obtain some extension $\E_i'=\left(N^0_i\right)^\sim$ for a finitely generated submodule $N^0_i\subset M_i=M_{\E_i}$.\\
Suppose we have a vertex $j\in Q_0$ such that for all its predecessors, i.e. vertices which allow an arrow $\alpha_k:i_k\to j$, the extensions are already constructed. That is
$$\E_{i_k}'=\left(N^0_{i_k}\right)^\sim$$
for certain finitely generated submodules $N^0_{i_k}\subset M_{i_k}$. Under the tilde functor the push forward of the morphisms for the arrows corresponds to morphisms of $A$--modules
$$\left(i_*\E_{\alpha_k}:i_*\E_{i_k}\to i_*\E_j\right)=\left(f_k:M_{i_k}\to M_j\right)^{\sim}.$$
To construct the extension of $\E_j$ we do not consider the full system of finitely generated submodules of $M_j$ but the subsystem of finitely generated submodules which contain the finitely generated submodule
$$\sum_k f_k\left(N^0_{i_k}\right)\subset M_j.$$
Obviously, the sum over this system again has $M_j$ as a limit, and the rest of the construction works as above.\\

The fact that $Q$ does not contain oriented cycles guarantees that this recursion extends the sheaves at all vertices and the morphisms at the arrows in finitely many steps.
\end{proof}
\end{proposition}

We need the application of this result to the case of families over a discrete valuation ring.

\begin{proposition}\label{P:P:extension_dvr}
Let $R$ denote a discrete valuation ring with field of fractions $K$, and suppose that $Q$ is a quiver without oriented cycles. Any family $\E$ of quiver sheaves over $K$ extends to a flat family of quiver sheaves over $R$.

\begin{proof}
An application of Proposition \ref{P:P:extension_opensubset} to the inclusion
$$U=X\times \Spec(K)\subset X\times \Spec(R)$$
provides the existence of an extension. It remains to show flatness.\\
Locally on $\Spec(A)\times \Spec(K)$, the sheaf $\E_i$ is given by an $A\otimes K$--module $M_i$, and the push forward $i_*\E_i$ is given by the same vector space equipped with the structure of an $A\otimes R$--module via the canonical map
$$A\otimes R\to A\otimes K.$$
According to Proposition \ref{P:P:extension_opensubset} the extension $\E_i'$ is given by some finitely generated submodule $N_i^0$. Because the canonical map is injective, $M$ is torsion--free as an $R$--module, and so $N_i^0$ is torsion--free as well. This is sufficient to show flatness over the principal ideal domain $R$.
\end{proof}
\end{proposition}

\section{Variation of stability conditions}\label{S:variation}
In this section we study the question of the dependence of the moduli spaces on the stability parameter $\sigma$. We consider some given (possibly empty) set of relations $I$.\\

\subsection{Walls and chambers}~\label{SS:wallschambers}\\

Fix a natural number $p$, a topological type $\tau$, a dimension $d$ and a tuple $\Lul$ of ample line bundles on some projective scheme $X$. Furthermore, consider a fixed subset
$$\Sigma\subset \left(\R^{Q_0\times N}\right)_+$$
which is separated from the boundary of $\left(\R^{Q_0\times N}\right)_+$. That is, we require the existence of a closed polyhedral cone
$$\Sigma\subset \Sigma'\subset \R^{Q_0\times N}$$
which, except for the origin, is contained in $\left(\R^{Q_0\times N}\right)_+$. We recall that a chamber structure on $\Sigma'$, in the sense of \cite{grt} Definition 4.1, is a collection of real hypersurfaces
$$W=\left(W_i,~i\in J\right).$$
The chambers are the maximal connected subsets $C\subset \R^n$ such that for each wall $W_i$ either $C\subset W_i$ or $C\cap W_i=\emptyset$ holds. Note that, by definition, the chambers are connected components of real semialgebraic sets, and so are path--connected as well (since they even admit triangulations according to \cite{lojasiewicz}).\\

We consider the family $S$ of quiver subsheaves $\F\subset \E$ which satisfy the conditions
\begin{enumerate}
\item $\E$ is pure of dimension $d$, $(p,\Lul)$--regular, of topological type $\tau$, and satisfies the relations $I$
\item $\F\subset \E$ is saturated,
\item $\hmu^\sigma(\F)\geq \hmu^\sigma(\E)$ for some $\sigma\in \Sigma'$.
\end{enumerate}

We need to ensure that this family is bounded, and the proof of this fact employs the properties of the auxiliary cone $\Sigma'$.
\begin{lemma}
The family $S$ is bounded.

\begin{proof}
Since the topological type $\tau$ is fixed, the expression
$$\hmu^\sigma(\E)=\frac{\sum_{i\in Q_0}\sum_{j=1}^N \sigma_{ij}\alpha^{L_j}_{d-1}(\E_i)}{\sum_{i\in Q_0}\sum_{j=1}^N \sigma_{ij}\alpha^{L_j}_{d}(\E_i)}$$
does only depend on $\sigma\in \Sigma'$. By an argument as in the proof of the Grothendieck lemma \ref{QS:L:grothendieck}, this is bounded from below as a function in $\sigma$. Hence, $S$ is bounded by the Grothendieck lemma.
\end{proof}
\end{lemma}

To each quiver sheaf $\F\in S$, and each integer $0\leq e \leq d-1$, we can associate a wall. It is defined by the equation

\begin{align}
W_{e,\F}: \sum_{i,k\in Q_0}\sum_{j,l=1}^N\sigma_{ij}\sigma_{kl}\left(\alpha^{L_j}_e(\F_i)\alpha^{L_l}_d(\E_k)-\alpha^{L_j}_e(\E_i)\alpha^{L_l}_d(\F_k)\right)=0. \label{E: quadrwall}
\end{align}
The chamber structure we consider is given by the equations which are non--trivial, i.e. which carve out neither the empty set nor the whole of $\Sigma'$.\\
Note that $\sigma$ is contained in the 
wall $W_{e,\F}$ if $p^\sigma_\E$ and $p^\sigma_\F$ have the same coefficient in degree $e$. Also note that the boundedness implies that there are finitely many distinct walls, as the coefficients of the equations are determined by the topological type of $\F$.

\begin{lemma}\label{V:L:chamber_inequalites}
Let $\E$ denote a purely $d$--dimensional and $(p,\Lul)$--regular quiver sheaf of topological type $\tau$ which satisfies the relations $I$, and consider a quiver subsheaf $\F\subset \E$ contained in the family $S$. Suppose that $\sigma$ and $\sigma'$ are stability conditions in $\Sigma'$ which are contained in the same chamber $C$. Then
\begin{center}
$p^\sigma_\F\leq p^\sigma_\E$ if and only if $p^{\sigma'}_\F\leq p^{\sigma'}_\E$.
\end{center}
A similar assertion holds for strict inequality.

\begin{proof}
This is a variant of the proof of Lemma 4.6 in \cite{grt}. For brevity, we concentrate on the case of non--strict inequality.\\
We may write
$$p^\sigma_\F-p^\sigma_\E=\sum_{s=0}^d \frac{c_s^\sigma}{s!}T^s\in \R[T].$$
Assuming that $p^\sigma_\F\leq p^\sigma_\E$ but $p^{\sigma'}_\F> p^{\sigma'}_\E$, we can choose a minimal index $e$ such that $c_t^{\sigma}=c^{\sigma'}_t=0$ for all $t>e$, so that
\begin{center}
$c_e^{\sigma}\leq 0$ and $c_e^{\sigma'}\geq 0$,
\end{center}
but not both vanish at the same time. The function
$$f(\sigma)=\frac{\sum_{i\in Q_0}\sum_{j=1}^N \sigma_{ij}\alpha^{L_j}_e(\F_i)}
{\sum_{k\in Q_0}\sum_{l=1}^N \sigma_{kl}\alpha^{L_l}_d(\F_k)}-\frac{\sum_{i\in Q_0}\sum_{j=1}^N \sigma_{ij}\alpha^{L_j}_e(\E_i)}{\sum_{k\in Q_0}\sum_{l=1}^N \sigma_{kl}\alpha^{L_l}_d(\E_k)}$$
is continuous in $\sigma$, has the wall $W_{e,\F}$ as vanishing locus, and satisfies $c_e^\sigma=f(\sigma)$. Hence, either $\sigma$ or $\sigma'$ is not an element of $W_{e,\F}$, so that $C$ is 
not contained in the wall. But if neither $c_e^\sigma$ nor $c_e^{\sigma'}$ vanishes, they have differing signs so that $f(\sigma'')=0$ for some $\sigma''$ on a connecting path 
from $\sigma$ to $\sigma'$ inside $C$. In any case, we deduce that $C\cap W_{e,\F}\neq \emptyset$, which is a contradiction.
\end{proof}
\end{lemma}

\begin{proposition}\label{V:P:chambers}
Consider $\Sigma,\Sigma',p,d,\tau$ and $\Lul$ as explained above. Then, the finite chamber 
structure $W=\left(W_{e,\F}\right)$ (cf. \eqref{E: quadrwall}) encodes the equivalence of stability conditions. More precisely, for any stability conditions $\sigma$ and $\sigma'$ inside $\Sigma'$ which belong to the same chamber $C$, and for any purely $d$--dimensional and $(p,\Lul)$--regular quiver sheaves $\E$ and $\E'$ of topological type $\tau$ which satisfy the relations $I$, the following assertions hold.
\begin{enumerate}
\item $\E$ is $\sigma$--semistable if and only if $\E$ is $\sigma'$--semistable, and the same assertion holds for stability.
\item Let $\E$ and $\E'$ be semistable with respect to both $\sigma$ and $\sigma'$. Then $\E$ is $S$--equivalent to $\E'$ with respect to $\sigma$ if and only if it is with respect to $\sigma'$.
\end{enumerate}

\begin{proof}
The proof is the same as the corresponding proof of Proposition 4.2 in \cite{grt}, where we use our Lemma \ref{V:L:chamber_inequalites}.
\end{proof}
\end{proposition}

\begin{corollary}\label{V:C:chambers_induce}
The chamber structure on $\Sigma'$ as in Proposition \ref{V:P:chambers} induces a finite partition of $\Sigma$ such that the behavior of stability and $S$--equivalence remains unchanged as $\sigma$ varies within any fixed block of the partition. 
\end{corollary}

Proposition 4.2 in \cite{grt} is a special case of our Proposition \ref{V:P:chambers} in the case of the trivial quiver $Q=\bullet$ and torsion--free sheaves on an integral scheme. In contrast to their result, 
we can not guarantee the existence of rational points in the chambers, though in the special case of symmetric stability conditions, discussed in Section \ref{S:symmetric}, we can. For our variation results however, we need to assume the existence of rational points.

\begin{assumption}\label{V:A:allchambersrationalpt}
Each chamber on $\Sigma'$ contains a rational point. In that case, we can replace any stability condition $\sigma$ by a rational one without loss of generality.
\end{assumption}

\begin{remark}
The relations we used in our definition of $S$ select a subfamily from the corresponding family $S'$ wich we get without imposing the relations. The chamber structure without relations is thus a refinement of the chamber structure with relations.\\
As in \cite{grt}, we think of $\Sigma'$ as a bounded set of stability parameters, so that the integer $p$ could then be chosen so that all semistable sheaves with respect to some $\sigma\in \Sigma'$ are $(p,\Lul)$--regular.
\end{remark}

The cone $\Sigma'$ is just an auxiliary object. Its purpose is to ensure the boundedness of the family $S$ (compare with the proof of the Grothendieck Lemma \ref{QS:L:grothendieck}), so that there are only finitely many walls. A priori it seems possible that for the whole set $\left(\R^{Q_0\times N}\right)_+$ this is no longer true, i.e. that the walls accumulate at the boundary.\\

\subsection{Variation of moduli spaces}\label{SS:variation_moduli}~\\

We now want to describe the behavior of the moduli spaces when changing the stability condition. For this to work, we need the 
moduli spaces $M^{\sigma-sst}_\tau(X,Q,I)$ to be projective. Hence, we assume that $Q$ contains no oriented cycles, which also implies 
that $Q(Q')$ does not contain oriented cycles (Lemma \ref{EF:L:twisted_cycles}). Furthermore, we assume that $X$ is smooth, and that 
Assumption \ref{V:A:allchambersrationalpt} holds, i.e., that each chamber contains a rational point.\\

Let $\Sigma \subset \R_{>0}^{Q_0 \times N}$ be a finite subset of rational, positive, and bounded stability conditions under a given set of relations $I$, and choose $m\gg n\gg p\gg 0$ so that 
the conditions (1)--(4) in the beginning of Section \ref{S:QSconstruction} are satisfied for all $\sigma \in \Sigma$. Let $$R_{I'}=R_d(Q(Q'), H, I'(I))$$
be the representation variety of $Q(Q')$ with labeling $H$ and relations $I'(I)$, where $d$ is the dimension vector determined by $\tau,m$ and $n$.
This embeds as a closed subvariety into the representation variety $R=R_d(Q(Q'), H)$. 
Moreover, for $\sigma \in \Sigma$ let $\theta=\theta(d, \sigma)$ be the corresponding parameter defined in Section \ref{S:stabilityembedding}, and let
\begin{align*}
 B^{[\sigma-sst]} \subseteq R_{I'} 
\end{align*}
denote the locally closed subset of $R_{I'}$ that parametrizes $\theta$--semistable points in $R_{I'}$ that lie in the image of the 
embedding functor $\mbox{Hom}(T, *)$ from Section \ref{SS:embeddingfunctor}, i.e., $\theta$--semistable points that come from 
$\sigma$--semistable quiver sheaves of topological type $\tau$ which satisfy the relations $I$. Then, $B^{[\sigma-sst]}$ is also a locally closed, and 
$G$--invariant, subset of $R$. Define 
 \begin{align*}
  Z=\overline{\bigcup_{\sigma \in \Sigma} B^{[\sigma-sst]}}.
 \end{align*}
 
 \begin{theorem}\label{V:T:masterspace}
  Let $Z$ be as above. Then, 
  \begin{align*}
   Z^{\theta-sst}=R^{\theta-sst} \cap Z=B^{[\sigma-sst]},
  \end{align*}
and the moduli space $M_\tau^{\sigma-sst}(X,Q,I)$ is given by the GIT quotient 
\begin{align*}
 M_\tau^{\sigma-sst}(X,Q,I)=Z^{\theta-sst}//G.
\end{align*}
\end{theorem}
 
 \begin{proof}
First of all, by the same argument as in the proof of Theorem 10.1 of \cite{grt}, it follows that $B^{[\sigma-sst]}$ is a dense open subset of $Z^{\theta-sst}$. 
  
Now, by the assumption that $Q$ has no oriented cycles, the same holds for the twisted quiver $Q(Q')$, and hence the quotient 
$R^{\sigma-sst}//G$ is projective. Since $Z$ is a closed and $G$--invariant subscheme, the GIT quotient $Z^{\theta-sst}//G$ embeds as a closed 
 subscheme of $R^{\theta-sst}//G$ and is thus also projective. Since $M^{\sigma-sst}_\tau(X,Q,I)=B^{[\sigma-sst]}//G$ and $Z^{\theta-sst}//G$ are 
 both projective, the image of $M^{\sigma-sst}_\tau(X,Q,I)=B^{[\sigma-sst]}//G$ in $Z^{\theta-sst}//G$ by the morphism induced by the 
 inclusion $B^{[\sigma-sst]} \subseteq Z^{\theta-sst}$ is closed. Being also dense, this image has to coincide with $Z^{\theta-sst}//G$.
 
 Let $\pi: Z^{\theta-\sst} \to Z^{\theta-ss}//G$ denote the quotient morphism. By the comparison of $S$--equivalence (Theorem \ref{CSS:T:S-equivalence}), for each $G$--orbit $\mathcal{O}$ in 
 $B^{[\sigma-sst]}$, the closed orbit $\mathcal{O}'$ in the orbit closure $\overline{\mathcal{O}}$ lies in $B^{[\sigma-sst]}$. Hence, we can 
 apply \cite{ack}, Lemma 6.2, to conclude that $\pi$ restricts to a good quotient $B^{[\sigma-sst]} \to B^{[\sigma-sst]}//G$, and that 
 $\pi^{-1}(\pi(B^{[\sigma-sst]}))=B^{[\sigma-sst]}$. Together with the above established identity $\pi(B^{[\sigma-sst]})=Z^{\theta-sst}$, this implies 
 that $B^{[\sigma-sst]}=Z^{\theta-sst}$.
 \end{proof}

We obtain a Mumford--Thaddeus principle for $\sigma$--semistable quiver sheaves, with essentially the same proof as in \cite{grt}. Note that the Thaddeus flips occurring here are to be interpreted in the context of VGIT of $Z$.

 \begin{corollary}
 Assume that Assumption \ref{V:A:allchambersrationalpt} holds. 
 Let $\sigma$ and $\sigma'$ be bounded stability parameters under some given set of relations $I$. Then the moduli spaces
\begin{center}
$M^{\sigma-sst}_\tau(X,Q,I)$ and $M^{\sigma-sst}_\tau(X,Q,I)$
\end{center} 
  are related by a finite sequence of Thaddeus flips induced by VGIT of the variety $Z$.
\end{corollary}
~

\section{Symmetric stability conditions}\label{S:symmetric}

The results about construction and variation of the moduli spaces, as outlined in Sections \ref{S:QSconstruction} and \ref{S:variation}, demand that the family of semistable sheaves under consideration is bounded. In general, we can not guarantee that this is true, though in the case that $\sigma_{ij}$ does not depend on the vertex $i\in Q_0$, we can build upon the boundedness results of \cite{grt}. Furthermore, the walls encoding the change of stability are rational hyperplanes in this case, which also solves the problem of finding rational points in the chambers.\\

Recall that
$$\sigma_i=(\sigma_{ij})_{j\in\{1,\ldots,N\}}$$
is the restriction of $\sigma$ to the vertex $i\in Q_0$.
For the remainder of this section we make the following assumption, which could also be understood as $\sigma$ being symmetric under permutations of $Q_0$.
\begin{assumption} \label{A: equalrest}
The tuples $\sigma_i \in \R_{\geq 0}^N \setminus 0$, for $i \in Q_0$, all coincide. 
\end{assumption}
Under this assumption on $\sigma$, we let $\hat{\sigma} \in \R^N$ denote this common restriction of $\sigma$. As a technical tool, we consider $Q_0$--tuples
$$(\E_i)_{i \in Q_0}$$
of coherent sheaves on $X$, following an idea of Schmitt (cf. \cite{schmitt-tuple}). Tupels can be identified with quiver sheaves 
for the quiver with the set of vertices $Q_0$ and with no arrows. In particular, if $\mathcal{E}$ is a quiver sheaf, we obtain a $Q_0$--tuple $(\E_i)_{i \in Q_0}$ 
of coherent sheaves by forgetting the morphisms $\E_\alpha: \E_i \to \E_j$ given by the arrows $\alpha: i \to j$ of $Q$. 
By a subtuple of a $Q_0$--tuple $(\E_i)_{i \in Q_0}$ we mean a $Q_0$--tuple $(\F_i)_{i \in Q_0}$, where $\F_i$ is a coherent subsheaf of $\E_i$ for each $i$. 
If $\sigma\in (\R^{Q_0 \times N})_+$, we say that a $Q_0$--tuple 
$(\E_i)_{i \in Q_0}$ is $\sigma$--semistable if it is pure of some dimension and the inequality
\begin{align*}
p^\sigma_\F \leq p^\sigma_\E
\end{align*}
holds for all nontrivial subtuples $(\F_i)_{i \in Q_0}$ of $(\E_i)_{i \in Q_0}$.\\


Let $E$ denote any coherent sheaf on $X$, and pick any vertex $i_0\in Q_0$. Then we define the quiver sheaf $\delta_{i_0}(E)$ by
$$\delta_{i_0}(E)_i=\left\{\begin{array}{ll}
E & i=i_0,\\
0 & i\neq i_0
\end{array}\right.,\delta_{i_0}(E)_\alpha=0.$$

It is very easy to verify that
$$p^\sigma_{\delta_{i_0}(E)}=p^{\sigma_{i_0}}_E.$$

\begin{lemma}\label{L:tuplesst_sheafsst}
Let $\E$ denote a $\sigma$--semistable tuple of sheaves on $X$. Then
$$p^\sigma_\E=p^{\sigma_i}_{\E_i}$$
for all vertices $i\in Q_0$. Moreover, all $\E_i$ are $\sigma_i$--semistable.

\begin{proof}
Clearly, $\delta_{i_0}(\E_{i_0})$ is both a subobject and a quotient of $\E$ in the category of tuples. Thus the equality of the Hilbert polynomials follows.\\
If $U\subset \E_i$ is any subsheaf, then $\delta_i(U)$ is also a subobject of the tuple $\E$. Hence
$$p^{\sigma_i}_U=p^\sigma_{\delta_i(U)}\leq p^\sigma_\E=p^{\sigma_i}_{\E_i}.$$
\end{proof}
\end{lemma}

\begin{remark}\label{R:positivecoeff}
The condition
$$p^\sigma_\E=p^{\sigma_i}_{\E_i}$$
of course implies the equality of degrees of these polynomials. Since we require that $\sigma \in (\R^{Q_0 \times N})_+$, this implies that
the dimensions
$$\dim(\E_i), i \in Q_0,$$
all coincide.
\end{remark}

\begin{lemma} \label{L:onlyzerohom}
 Let $(\E_i)_{i \in Q_0}$ and $(\F_i)_{i \in Q_0}$ be $\sigma$--semistable tuples and assume that $p^\sigma_\mathcal{E}>p^\sigma_\mathcal{F}$. 
 Then $$\Hom(\mathcal{E}_i, \mathcal{F}_k)=0$$ for all $i, k \in Q_0$.
\end{lemma}

\begin{proof}
 Let $i, k \in Q_0$. 
 By Lemma \ref{L:tuplesst_sheafsst},  $\mathcal{E}_i$ and $\mathcal{F}_k$ are both semistable with respect to $\hat{\sigma}$, and 
 $p^{\hat{\sigma}}_{\E_i}>p^{\hat{\sigma}}_{\F_k}$. The claim then follows from general properties of stability conditions (cf. \cite{huybrechtslehn}, Proposition 1.2.7 and its proof).
\end{proof}

Although we are interested in quiver sheaves and tuples, we formulate our next lemma in the general setting of stability conditions on abelian categories 
(cf. \cite{rudakov}).
\begin{lemma}\label{L:HNweightslargerslope}
Let
$$0=\E_0\subset \E_1 \subset \cdots \subset \E_r=\E$$
denote the Harder--Narasimhan filtration for some object $\E$ with respect to some semistability condition $\mu$ (in the sense of \cite{rudakov}).\\ 
Then, for all $0<i<r$ we have
$$\mu(\E_i)>\mu(\E).$$
\begin{proof}
We proceed by induction on $i$.\\
If the filtration is nontrivial, we clearly have $\mu(\E_1)>\mu(\E)$. Now suppose that $\mu(\E_i)>\mu(\E)$ is already settled, that is, 
$\mu(\E)>\mu(\E/\E_i)$. The Harder--Narasimhan filtration 
of $\E/\E_i$ is given by 
\begin{align*}
 0 \subset \E_{i+1}/\E_i \subset \cdots \subset \E/\E_i, 
\end{align*}
and by definition $\E_{i+1}/\E_i$ is then a maximal destabilizing subobject of $\E/\E_i$, so that 
\begin{align}
 \mu(\E_{i+1}/\E_i)>\mu(\E/\E_i). \label{E: minslope}
\end{align}
The short exact sequence 
$$0\to \E_{i+1}/\E_i \to \E/\E_i \to \E/\E_{i+1} \to 0$$
then yields 
\begin{align*}
 \mu(\E/\E_{i})>\mbox{min}\{\mu(\E_{i+1}/\E_i), \mu(\E/\E_{i+1})\}=\mu(\E/\E_{i+1}),
\end{align*}
where the equality holds by virtue of \eqref{E: minslope}. Hence,
$$\mu(\E)>\mu(\E/\E_i)>\mu(\E/\E_{i+1}),$$
which then implies that $\mu(\E_{i+1})>\mu(\E)$. The claim thus follows by induction.
\end{proof}
\end{lemma}

\begin{proposition} \label{P: semistab-tuple-quivsheaf}
Assume that $\sigma$ satisfies Assumption \ref{A: equalrest}.
Let $\E$ denote any pure quiver sheaf. Then $\E$ is $\sigma$--semistable as a quiver sheaf if and only if $\E$ is $\sigma$--semistable as a 
$Q_0$--tuple of sheaves on $X$.\\
In that case, all sheaves $\E_i$ are $\hat{\sigma}$--semistable. Moreover, the numbers $\dim(\E_i)$ all coincide.

\begin{proof}
We first assume that $\E$ is semistable as a quiver sheaf and show that it is then semistable as a tuple. Let therefore
$$0=\F^0\subset \F^1 \subset \cdots \subset \F^r=\E$$
denote the Harder--Narasimhan filtration in the category of tuples. We assume that this filtration is not trivial, i.e., $r\geq 2$.\\

Because $p^{\sigma}_{\F^k}>p^{\sigma}_\E$ for all $k$, by Lemma \ref{L:HNweightslargerslope}, none of the $\F^k$ can be a quiver subsheaf of $\E$, so 
that for each index $k$ there 
exists at least one arrow $\left(\alpha:i\to j\right)=\left(\alpha^k:i^k\to j^k\right)$ such that
$$\E_\alpha(\F^k_i)\not\subseteq \F^k_j.$$
Fix one such $k$ and $\alpha:i\to j$ and choose $l$ minimal such that $\E_\alpha(\F^k_i)\subseteq \F^l_j$. By choice of $\alpha$ we clearly 
have $l\geq k+1$. Furthermore, choose $m$ maximal such that $\F_i^m\subseteq \ker(\E_\alpha)$. Again, the choice of $\alpha$ implies $m\leq k-1$. Then 
we have a well--defined induced morphism
$$f:\F^{m+1}_i/\F^m_i\to \F^l_j/\F^{l-1}_j$$
which does not vanish by the choice of $m$ and $l$. Note that
$$\F^{m-1}_i/\F^m_i=(\F^{m-1}/\F^m)_i$$
is $\hat{\sigma}$--semistable according to Lemma \ref{L:tuplesst_sheafsst}, with the same multi--Hilbert polynomial as the corresponding subquotient 
of tuples. Now,
\begin{align*}
p^{\sigma}_{\F^k/\F^{k-1}}&\leq p^{\sigma}_{\F^{m+1}/\F^m}\\
&\leq p^{\sigma}_{\F^l/\F^{l-1}}\\
&\leq p^{\sigma}_{\F^{k+1}/\F^k},
\end{align*}
where the first and third inequalities follow by the properties of the Harder--Narasimhan filtration, and the second inequality follows from 
Lemma \ref{L:onlyzerohom} because of 
the existence of the nonzero morphism $f$. Now, the inequality
$$p^{\sigma}_{\F^k/\F^{k-1}}\leq p^{\sigma}_{\F^{k+1}/\F^{k}}$$
is a contradiction. Hence, $\E$ is already semistable as a tuple.
\\

The other implication is trivial, and the question of semistability of the $\E_i$ and their dimensions is 
taken care of by Lemma \ref{L:tuplesst_sheafsst} and Remark \ref{R:positivecoeff}.
\end{proof}
\end{proposition}

Under the assumption of this section we can now show a boundedness result for families of semistable quiver sheaves for certain varieties $X$, building on the work of \cite{grt}. Note that this also implies the boundedness of the subfamily defined by any set of relations $I$.

\begin{theorem} \label{T: semistablebounded}
 Let $X$ be a smooth projective variety, and let $\tau \in B(X)_{\mathbb{Q}}^{Q_0}$ be a fixed topological type. Assume further that one of the following conditions holds:
 \begin{enumerate}
  \item $\dim(X) \leq 3$, 
  \item the Picard rank of $X$ is at most two, 
  \item all the sheaves under consideration are of rank at most two. 
 \end{enumerate}
Then, for any subset $\Sigma \subseteq (\R^{Q_0 \times N})_+$ of parameters $\sigma$ satisfying Assumption \ref{A: equalrest}, 
the family of all quiver sheaves of topological type $\tau$ which are semistable with respect to some $\sigma \in \Sigma$ is bounded. 
\end{theorem}

\begin{proof}
By Proposition \ref{P: semistab-tuple-quivsheaf}, the claim follows from the corresponding boundedness result for $\hat{\sigma}$--semistable sheaves in \cite{grt}, Corollary 6.12. 
\end{proof}

Clearly, the set of stability conditions $\sigma$ which satisfy Assumption \ref{A: equalrest} form a subset $\Sigma_0$ of $\left(\R^{Q_0\times N}\right)_+$ which is convex. Hence, it is possible to move from one such condition to another without leaving this subset. At least for torsion--free quiver sheaves, the walls encoding the change of stability are hyperplanes.

\begin{proposition}\label{SYM:P:wallslinear}
Let $\tau$ denote a topological type of torsion--free quiver sheaves on an integral and projective scheme $X$, and let $p$ be a natural number. Then, for any closed polyhedral cone $\Sigma'\subset\Sigma_0$, the walls, as given by Subsection \ref{SS:wallschambers}, are rational hyperplanes. In particular, each chamber in $\Sigma'$ contains a rational point.

\begin{proof}
Consider the wall $W_{e,\F}$ associated to the quiver subsheaf $\F\subset \E$ and integer $e$, which is defined by the equation
$$\frac{\sum_{i\in Q_0}\sum_{j=1}^N \sigma_{ij}\alpha^{L_j}_e(\F_i)}{\sum_{k\in Q_0}\sum_{l=1}^N \sigma_{kl}\alpha^{L_l}_d(\F_k)}-\frac{\sum_{i\in Q_0}\sum_{j=1}^N \sigma_{ij}\alpha^{L_j}_e(\E_i)}{\sum_{k\in Q_0}\sum_{l=1}^N \sigma_{kl}\alpha^{L_l}_d(\E_k)}=0.$$
Because $X$ is integral, and each sheaf $\E_k$ is torsion--free, we can write $\alpha^{L_l}_d(\F_k)=\alpha^{L_l}_d(\Ocal_X)\rk(\F_k)$, and similarly for the $\E_k$ (compare with \cite{huybrechtslehn}, Definition 1.2.2). Note that the rank does not depend on the choice of an ample line bundle. Further, by Assumption \ref{A: equalrest}, we can write $\sigma_{ij}=\hat{\sigma}_j$, so that we may rewrite the equation as
$$f(\sigma)\frac{\sum_{i\in Q_0}\sum_{j=1}^N\sigma_{ij}(\rk(\E)\alpha^{L_j}_e(\F_i)-\rk(\F)\alpha^{L_j}_e(\E_i))}{\rk(\F)\rk(\E)}=0.$$
Here, we introduce the notation
\begin{align*}
\rk(\E)&=\sum_{i\in Q_0}\rk(\E_i),\\
f(\sigma)&=\left(\sum_{k\in Q_0}\sum_{l=1}^N\hat{\sigma}_l\alpha_d^{L_l}(\Ocal_X)\right)^{-1}.
\end{align*}
The function $f$ is always positive, and hence can be ignored for the definition of the wall. The remaining equation is linear in the entries $\sigma_{ij}$.
\end{proof}
\end{proposition}

\begin{corollary} \label{C: linearwalls}
Let $X$ denote a smooth projective variety, and let $\tau$ denote a fixed topological type for torsion--free quiver sheaves. Assume that one of the conditions of Theorem \ref{T: semistablebounded} holds. Then, for any stability conditions $\sigma,\sigma'\in \Sigma_0$ the moduli spaces
\begin{center}
$M^{\sigma-\sst}_\tau(X,Q)$ and $M^{\sigma'-\sst}_\tau(X,Q)$
\end{center}
exist. Moreover, if $Q$ does not contain oriented cycles, these moduli spaces are projective, and are related by a finite sequence of Thaddeus flips, induced by VGIT of the variety $Z$ as in Theorem \ref{V:T:masterspace}. 
\end{corollary}

\begin{appendices} 

\section{Quivers and their representations}

In this section, we outline how the notions of representations and stability conditions (as introduced in \cite{king} for unlabeled quivers) can be generalized to the case of labeled quivers with relations, i.e. ordinary quivers with labeling vector spaces assigned to the arrows. This is a mostly straightforward procedure, but even though special cases have been studied (e.g. in \cite{ack} or \cite{grt}), there seems to be no such concise treatment in the literature.\\
Apart from the foundational paper \cite{king}, we also refer the reader to the survey article \cite{reineke} for an introduction to the case of an unlabeled quiver.\\

We recall that an (unlabeled) quiver $Q=(Q_0,Q_1)$ consists of a set of vertices $Q_0$ and a set of arrows $\alpha:i\to j$. We assume both sets to be finite.

\begin{definition}
A labeling for a quiver $Q$ is a collection of vector spaces
$$H=\left(H_\alpha \mid \alpha\in Q_1\right)$$
of finite dimension for each arrow in $Q$. An arrow $\alpha$ with label $H_\alpha=k$ is considered to be unlabeled. The pair $(Q,H)$ is called a labeled quiver.
\end{definition}
~

\subsection{The category of representations}\label{SS:quiver_cat}~\\

A representation $M$ of an (unlabeled) quiver $Q$ in some category $\C$ consists of a tuple of objects $\left(M_i\mid i\in Q_0\right)$ for each vertex, and a tuple of morphisms
$$\left(M_\alpha:M_i\to M_j \mid (\alpha:i\to j)\in Q_1\right)$$
for each arrow. Together with the appropriate notion of morphisms, these representations form a category.
\begin{definition}
The category of representations of $Q$ in the category $C$ is denoted as $Q-\rep_\C.$
\end{definition}

The special cases of representations in the category $k-\vect$ of vector spaces of finite dimension, and in the category of coherent sheaves $\Coh(X)$ on some scheme $X$ over $k$ deserve the special notations
$$Q-\rep=Q-\rep_{k-\vect},\quad Q-\Coh(X)=Q-\rep_{\Coh(X)}.$$
Objects in $Q-\Coh(X)$ are called quiver sheaves.\\

A representation of a labeled quiver is given in a slightly different manner.
\begin{definition}
A representation $M$ of a labeled quiver $(Q,H)$ consists of a tuple $\left(M_i \mid i\in Q_0\right)$ of finite--dimensional vector spaces, and a tuple
$$\left(M_\alpha: M_i\otimes_k H_\alpha\to M_j \mid (\alpha:i\to j)\in Q_1\right)$$
of linear maps. A morphism of representations $\varphi: M\to N$ consists of a tuple of linear maps $\left(\varphi_i:M_i\to N_i\mid i\in Q_0\right)$, such that for all arrows $\alpha:i\to j$ in $Q$ the diagram
$$\xymatrix{M_i\otimes_k H_\alpha \ar[rr]^{M_\alpha} \ar[d]_{\varphi_i\otimes id} && M_j \ar[d]^{\varphi_j} \\
N_i\otimes_k H_\alpha \ar[rr]_{N_\alpha} && N_j}$$
commutes.
\end{definition}

Again, the representations form a category.
\begin{definition}
The category of representations of $(Q,H)$ is denoted as $(Q,H)-\rep$.
\end{definition}

We also need the dimension vector of a representation.
\begin{definition}
The dimension vector of a vector space representation $M$ is defined as
$$\dimvec(M)=\left(\dim(M_i)\mid i\in Q_0\right),$$
both for the labeled and unlabeled case.
\end{definition}

In the unlabeled case, relations on the quiver are a well--established concept. To the best knowledge of the authors, there is no previous treatment of relations in the case of a labeled quiver in the literature.\\
For the construction of the moduli space, we need such relations. More precisely, we restrict ourselves to the discussion of one very special case involving non--trivial labels, and relations only involving unlabeled arrows.

\begin{convention}
Suppose that $(Q,H)$ contains a subquiver
$$\xymatrix{1 \ar[d]_\beta \ar[rr]^H_\alpha && 2 \ar[d]^\gamma\\
3 \ar[rr]^H_\delta && 4,}$$
that is two unlabeled arrows $\beta,\gamma$ and two labeled arrows $\alpha,\delta$, with the same label $H$ on opposite sides of a square. We then say that a representation $M$ satisfies the relation $\gamma\alpha-\delta\beta$ if there is a commuting diagram
$$\xymatrix{M_1\otimes_k H \ar[rr]^{M_\alpha} \ar[d]_{M_\beta\otimes id} && M_2 \ar[d]^{M_\gamma}\\
M_3\otimes_k H \ar[rr]_{M_\delta} && M_4.}$$
\end{convention}

Another special case we need is that of relations which only involve unlabeled arrows. They can be treated in exactly the same way as in the well--known case of relations on an unlabeled quiver.\\

We remark that relations in the unlabeled case also make sense for arbitrary $k$--linear categories. In particular, we can consider quiver sheaves satisfying a set of relations $I$.
\begin{definition}
The full subcategory of quiver sheaves satisfying the relations $I$ is denoted as
$$(Q,I)-\Coh(X)\subset Q-\Coh(X).$$
\end{definition}

A labeled quiver with relations is now defined as a triple $(Q,H,I)$, where $(Q,H)$ is a labeled quiver and $I$ is a set of relations of a form as discussed above.
\begin{definition}
The category 
$$(Q,H,I)-\rep\subset (Q,H)-\rep$$
is the full subcategory of representations which satisfy the relations $I$.
\end{definition}

The connection of the notion of labeled representations to the unlabeled ones is via a choice of basis for each label $H_\alpha$. This also allows us to inherit many already established results.\\
Specifically, let $(Q,H,I)$ denote any labeled quiver with relations. Construct a new (unlabeled) quiver $Q'$ by setting
$$Q_0'=Q_0,~Q_1'=\left\{\alpha_k:i\to j \mid (\alpha:i\to j)\in Q_1,~k=1,\ldots,\dim(H_\alpha)\right\}.$$
Roughly speaking, we replace each arrow by $\dim(H_\alpha)$ copies. For a relation $\gamma\alpha-\delta\beta$ in a form as explained above, we equip $Q'$ with $\dim(H)$ relations of the form
$$\gamma\alpha_k-\delta_k\beta,\quad k=1,\ldots,\dim(H).$$
Relations only involving unlabeled arrows can be imposed on $Q'$ in a straight--forward way. The set of all such relations is denoted as $I'$.\\
Now, choose a basis for each label $H_\alpha$. We restrict this choice by assuming the following.
\begin{enumerate}
\item If the label is trivial, i.e. $H_\alpha=k$, the canonical basis $1\in k$ is chosen.
\item If two labels are exactly the same, i.e. $H_\alpha=H_\beta$, the same bases are chosen.
\end{enumerate}

It is straightforward to check that these two quivers essentially have the same categories of representations.
\begin{proposition}\label{Q:P:identification_reps}
The choice of bases induces an isomorphism of categories
$$(Q,H,I)-\rep\to (Q',I')-\rep.$$
This identification respects dimension vectors.
\end{proposition}

The paths in a labeled quiver $(Q,H)$ are simply paths in the underlying quiver $Q$, and the label of a path $\gamma=\alpha_1\alpha_2\ldots\alpha_l$ is defined as
$$H_\gamma=H_{\alpha_l}\otimes_k\ldots\otimes_k H_{\alpha_1}.$$
By convention, the paths $e_i$ of length zero should be labeled by $k$.

\begin{definition}
The path algebra of $(Q,H)$ is
$$A=A(Q,H)=\bigoplus_{\gamma\mathrm{~path}}k\cdot\gamma\otimes_k H_\gamma$$
as a vector space, and multiplication is given on homogeneous elements by
$$\left(\gamma\otimes h\right)\cdot\left(\gamma'\otimes h'\right)=\gamma\gamma'\otimes(h'\otimes h)$$
if the concatenation $\gamma\gamma'$ is possible, and zero otherwise. Relations of the form $\delta\beta-\gamma\alpha$ in $(Q,H)$ give rise to the relations
$$\left(\delta\otimes h_\delta\right)\cdot \left(\beta\otimes h_\beta\right)-\left(\gamma\otimes h_\gamma\right)\cdot\left(\alpha\otimes h_\alpha\right)=0$$
independent of the elements $h_\alpha,h_\delta\in H$ and $h_\beta,h_\gamma\in k$. The path algebra with relations is then simply the path algebra of $(Q,H)$ modulo the ideal generated by these relations.
\end{definition}

We remark that there is an isomorphism
$$A\simeq A'$$
between the path algebra $A$ of $(Q,H,I)$ and the path algebra $A'$ of $(Q',I')$.\\

\subsection{Stability conditions}\label{SS:quiver_stability}~\\

A stability condition on a quiver, labeled or unlabeled, is a tuple $\theta\in \R^{Q_0}.$
\begin{definition}
The slope of a representation with respect to $\theta$ is given as
$$\mu(M)=\frac{\sum_{i\in Q_0}\theta_i \dim(M_i)}{\sum_{i\in Q_0}\dim(M_i)},$$
and $M$ is said to be $\theta$--semistable if and only if
$$\mu(N)\leq \mu(M)$$
holds for all non--trivial subrepresentations $N\subset M$. If strict inequality holds for all subrepresentations we say that $M$ is stable. Consequently, for some representation $M$ we say that $N\subset M$ is destabilizing if $\mu(N)\geq \mu(M)$.
\end{definition}

Just as in the unlabeled case, we can show the existence of Jordan--Hölder filtrations for semistable representations, and thus we obtain the notion of $S$--equivalence.\\

Consider the identification of Proposition \ref{Q:P:identification_reps}. Since $Q$ and $Q'$ have the same set of vertices, a stability condition $\theta$ can be used on both simultaneously. Subrepresentations and dimension vectors are preserved, so that stability is preserved as well.
\begin{corollary}
The identification
$$(Q,H,I)-\rep \simeq (Q',I')-\rep$$
preserves stability, semistability and $S$--equivalence.
\end{corollary}
~

\subsection{Moduli spaces}\label{SS:quiver_moduli}~\\

We briefly sketch the construction of moduli spaces of semistable representations. This is a slight variation of the program carried out by \cite{king} for the case of unlabeled quivers.\\

Let $(Q,H,I)$ denote a labeled quiver with relations.
\begin{definition}
The representation variety for dimension vector $d\in \N^{Q_0}$ is given as
$$R_d(Q,H)=\bigoplus_{\alpha:i\to j}\Hom\left(k^{d_i}\otimes H_\alpha,k^{d_j}\right),$$
and the relations $I$ define a closed subvariety $R_d(Q,H,I)\subset R_d(Q,H)$.
\end{definition}

On $R_d(Q,H)$ we have an action of the group
$$G_d=\prod_{i\in Q_0}\GL(d_i,k)$$
via conjugation, that is
$$\left(g_i\mid i\in Q_0\right)*\left(f_\alpha\mid \alpha:i\to j\right)=\left(g_j^{-1}f_\alpha\left(g_i\otimes id\right)\mid \alpha:i\to j\right).$$
Note that the diagonally embedded scalars $\Gm\subset G_d$ act trivially, so that we can equivalently pass to the action of the group
$$PG_d=G_d/\Gm.$$

Furthermore, there is a tautological bundle $\Mbb$ of $A$--modules on $R_d(Q,H,I)$, where $A$ denotes the path algebra of $(Q,H,I)$. As a vector bundle,
$$\Mbb=\left(\bigoplus_{i\in Q_0}k^{d_i}\right)\times R_d(Q,H,I)\to R_d(Q,H,I)$$
is trivial. Each fiber $\Mbb_x$ carries the structure of an $A$--module, by declaring that each arrow acts via the morphisms encoded in $x$.\\

The identification of representations of $(Q,H,I)$ with those of an unlabeled quiver via choice of bases implies that the GIT--construction of the moduli space transfers to the labeled case. Here, we use the abbreviations $\st=\theta-\st$ and $\sst=\theta-\sst$.

\begin{theorem}\label{Q:T:quivermoduli}
There is a commuting diagram
$$\xymatrix{R_d(Q,H,I)^\st \ar[d] \ar@{^(->}[r] & R_d(Q,H)^\sst \ar[d] \ar@{^(->}[r] & R_d(Q,H,I)\ar[d] \\
M^\st_d(Q,H,I) \ar@{^(->}[r] & M^\sst_d(Q,H,I) \ar[r]^p& M^\ssp_d(Q,H),}$$
and the following assertions hold.
\begin{enumerate}
\item The vertical arrows are good quotients, and the leftmost quotient is even geometric.
\item The inclusions are open, and $p$ is a projective morphism.
\item The varieties $M^\st_d(Q,H,I), M^\sst_d(Q,H,I)$ and $M^\ssp_d(Q,H,I)$ are the moduli spaces of stable, semistable and semisimple representations of $(Q,H,I)$ of dimension vector $d$, respectively.
\item The variety $M^\ssp_d(Q,H,I)$ is affine. It is trivial if and only if $Q$ does not contain oriented cycles.
\item Closed points in $M_d^\sst(Q,H,I)$ correspond to $S$--equivalence classes of semistable representations, and closed points in $M^\st_d(Q,H,I)$ correspond to isomorphism classes of stable representations.
\end{enumerate}

\begin{proof}
The special case of an unlabeled quiver without relations is settled by \cite{king}, and the general theory of GIT (see eg. \cite{reineke}, Section 3.5), and the question when $M^\ssp_d(Q)$ is affine is settled by \cite{lebruynprocesi}.\\
Relations in the unlabeled case can be handled, once we note that $S$--equivalence respects relations. This is because orbit closure respects $S$--equivalence and 
$$R_d(Q,I)^\sst=R_d(Q,I)\cap R_d(Q)^\sst.$$
Finally, our identifications allow a transfer to the labeled case. Note that $Q'$ contains oriented cycles if and only if $Q$ does.
\end{proof}
\end{theorem}

Denote by $A$ the path algebra of $(Q,H,I)$, by $R$ the representation variety $R_d(Q,H,I)$, and by $G$ the group $G_d$. Consider the functor
$$\M_A:(Sch/k)^\op\to Sets,$$
which maps a scheme $S$ to the set of isomorphism classes of $A\otimes \Ocal_S$--modules which are locally free as $\Ocal_S$--modules.\\
By sending a morphism $f:S\to R$ to the pullback $f^*\Mbb$ of the tautological family, we obtain a natural transformation
$$h:\underline{R}\to \M_A.$$
Just like in \cite{ack}, Proposition 4.4, we can prove that $\M_A$ is locally isomorphic to a quotient functor.
\begin{proposition}\label{Q:P:A-modfunctor_quotient}
The natural transformation $h$ induces a local isomorphism
$$h':\underline{R}/\underline{G}\to \M_A.$$
\end{proposition}

\section{Quiver Quot--Schemes}\label{S:quiverquot}

We construct the quiver sheaf version of the Quot--scheme, which is helpful as a technical tool throughout this chapter. This section is mostly independent of the other sections.\\

First, we need a version of the Quot--scheme for sheaves which uses the topological type instead of the Hilbert polynomials.\\

Recall the strategy of proof in the construction of the sheaf version of the Quot--scheme $\Quot^P_{E/X/S}$, parametrizing flat quotients of the sheaf $E$ with Hilbert polynomial $P$ (eg. consider \cite{huybrechtslehn}, Theorem 2.2.4). Reducing to the case
$$X=\P^n\to S=\Spec(k),$$
one constructs an embedding of functors
$$\Quotfunc^P_{E/X/S}\to \Grassfunc_k\left(H^0\left(E(m)\right),P\right),$$
where the right hand side is represented by the Grassmannian scheme. The Quot--scheme is then given as the component corresponding to $P$ of the flattening stratification of some suitable sheaf $F$ on the Grassmannian.\\

This proof carries over to the case of the Quot--scheme $\Quot^\tau_{E/X/S}$, parameterizing flat quotients of $E$ with topological type $\tau\in B(X)_\Q$. More precisely, the same construction as above gives an embedding
$$\Quotfunc^\tau_{E/X/S}\to \Grassfunc_k\left(H^0\left(E(m)\right),P\right),$$
where $P=P(\tau)$ is the Hilbert polynomial determined by $\tau$ (and the implicitly fixed relatively very ample line bundle). Since $\tau$ is locally constant in flat families, we may consider the component of the flattening stratification corresponding to $\tau$. The same argument as above shows that this component represents the Quot--functor.\\

We will mainly need the generalization to the Quiver Quot--scheme in the version using the topological type. Though all arguments hold for the version using Hilbert polynomials equally well.
\begin{definition}
Consider a projective scheme $X\to S$ over a noetherian base scheme $S$, a topological type $\tau \in B(X)_\Q^{Q_0}$ and a quiver sheaf $\E$ on $X$.
The Quiver Quot--functor
$$\Quotfunc^\tau_{\E/X/S}:(Sch/S)^\op\to Sets$$
assigns to a scheme $T\to S$ the set of equivalence classes of quotient quiver sheaves $q:\E_T\to\F$ on $X_T$ where $\F$ is flat over $T$ and the topological type of the fiberwise quiver sheaves $(\E)_t$ are equal to $\tau$.
\end{definition}

\begin{definition}
Consider a projective scheme $X\to S$ over a noetherian base scheme $S$, a tuple of polynomials $P \in \Q[T]^{Q_0}$ and a quiver sheaf $\E$ on $X$. Also, implicitly fix a relatively very ample line bundle on $X$.
The Quiver Quot--functor
$$\Quotfunc^P_{\E/X/S}:(Sch/S)^\op\to Sets$$
assigns to a scheme $T\to S$ the set of equivalence classes of quotient quiver sheaves $q:\E_T\to\F$ on $X_T$ where $\F$ is flat over $T$ and the Hilbert polynomials of the fibrewise sheaves at the vertices $(\E_i)_t$ are equal to $P_i$.
\end{definition}

\begin{proposition}\label{QQ:P:pullbackfunctor_representable}
Let $X\to S$ denote a projective scheme over a noetherian base scheme $S$, and let $E$ and $F$ denote coherent sheaves on $X$ such that $F$ is flat over $S$. Consider a morphism $f:E\to F$ and the functor
$$F:(Sch/S)^\op\to Sets,~ \left(T\to S\right) \mapsto \left\{\phi\in \Hom_S(T,X) \mid \Phi^*(f)=0\right\}.$$
This functor is represented by a closed subscheme of $X$.

\begin{proof}
Consider the functor
$$F':(Sch/S)^\op\to Sets,~\left(T\to S\right)\mapsto \Hom_{X_T}(E_T,F_T)$$
as in \cite{nitsure}, Theorem 5.8. It is represented by a scheme $V=\Spec \left(\Sym_{\Ocal_S}(Q)\right)$ for a suitable sheaf $Q$ on $S$.\\
In our situation, $f\in F'(X)$ corresponds to a morphism $\phi_f:X\to V$, so that Remark 5.9 in \cite{nitsure} implies that our functor $F$ is represented by the closed subscheme $\phi_f^{-1}(V_0)$, where $V_0\subset V$ is the image of the zero section $S\to V$.
\end{proof}
\end{proposition}

\begin{theorem}
The Quiver Quot--functor is represented by a projective scheme $\Quot^\tau_{\E/X/S}$, the so called Quiver Quot--scheme. It is given as the closed subscheme
$$\Quot^\tau_{\E/X/S}\subset \prod_{i\in Q_0} \Quot^{\tau_i}_{\E_i/X/S}$$
which satisfies the conditions to be compatible with the morphisms $\E_\alpha$.\\
A similar statement holds for the version using Hilbert polynomials.

\begin{proof}
By projection to the quotient sheaves at the vertices we get a natural transformation of functors
$$\Quotfunc^\tau_{\E/X/S}\to \prod_{i\in Q_0} \Quotfunc^{\tau_i}_{\E_i/X/S},$$
where the right hand side is clearly represented by the projective scheme
$$G=\prod_{i\in Q_0}G_i=\prod_{i\in Q_0} \Quot^{\tau_i}_{\E_i/X/S}.$$
Consider the universal quotients at the vertices
$$q_i:(\E_i)_{G_i}\to \U_i$$
on the Quot--schemes $G_i$, and denote by $e_i:\K_i\to (\E_i)_{G_i}$ their kernels.\\
For a morphism
$$\phi:T\to G$$
of noetherian schemes over $S$ and any arrow $\alpha:i\to j$ we obtain the following exact diagram via pullback.
$$\xymatrix{\Phi^*\K_i \ar[r]^{\Phi^*e_i} & (\E_i)_T \ar[r]^{\Phi^*q_i} \ar[d]_{(\E_\alpha)_T} & \Phi^*\U_i \ar[r] & 0 \\
\Phi^*\K_j \ar[r]^{\Phi^*e_j} & (\E_j)_T \ar[r]^{\Phi^*q_j} & \Phi^*\U_j \ar[r] & 0}$$
Hence the morphism $(\E_\alpha)_T$ descends to the quotients if and only if
$$0=\Phi^*q_j(\E_\alpha)_T\Phi^* e_i=\Phi^*(q_j(\E_\alpha)_G e_i),$$
which shows that the Quiver Quot--functor is given as the subfunctor of $\Hom_S(*,G)$ which satisfies finitely many pullback equations. It is thus represented by the intersection of the finitely many closed subschemes provided by Proposition \ref{QQ:P:pullbackfunctor_representable}.
\end{proof}
\end{theorem}

\section{Langton's theorem}\label{S:langton}

This rather technical section is devoted to proving a version of Langton's theorem for multi--Gieseker semistable quiver sheaves. Suppose that we have a flat family of quiver sheaves over the spectrum of a discrete valuation ring. The theorem then states that if the fiber over the open point is semistable, the family can be modified over the closed point such that the fiber there becomes semistable as well.\\

For the proof we first introduce some technicalities, which are perhaps not of general interest outside of this section. For the proof of the theorem itself, we mostly follow the reasoning of \cite{huybrechtslehn}, Theorem 2.B.1.\\

Fix a quiver $Q$, a projective and smooth scheme $X$ over $k$, a dimension $d\leq \dim(X)$ and a stability condition $\left(\Lul,\sigma\right)$ on $Q$.\\

For a quiver sheaf $\E$ we consider the number
$$s(\E)=\max(\dim(\E_i))\in \N.$$
Following \cite{huybrechtslehn} Definition 1.6.1, the category $Q-\Coh(X)_s$ is given as the full subcategory of the category $Q-\Coh(X)$ of quiver sheaves $\E$ with $s(\E)\leq s$. Clearly,
$$Q-\Coh(X)_s\subset Q-\Coh(X)_t$$
for $s\leq t$ is a full subcategory, which obviously is closed under subobjects, quotients and extensions. Hence it forms a Serre subcategory, and we can consider the quotient category
$$Q-\Coh(X)_{d,d'}=Q-\Coh(X)_d/Q-\Coh(X)_{d'-1}.$$
For background on this construction we refer to \cite{gabriel}. The quotient category is again abelian and the canonical functor
$$Q-\Coh(X)_d\to Q-\Coh(X)_{d,d'}$$
is exact (\cite{gabriel}, Lemma III.1). Note that two objects $\E,\F$ in $Q-\Coh(X)_{d,d'}$ are isomorphic if and only if there exists an ordinary morphism $\phi:\E\to \F$ such that kernel and cokernel of $\phi$ are contained in $Q-\Coh(X)_{d'-1}$ (\cite{gabriel} Lemma III.4). In this case we say that they are isomorphic in dimension $d'$.\\
By additivity of Hilbert polynomials on exact sequences we thus get a well--defined map
$$P^\sigma:Q-\Coh(X)_{d,d'}\to \R[T]_{d,d'},$$
assigning to a quiver sheaf its multi--Hilbert polynomial with respect to the fixed stability condition. Here, $\R[T]_{d,d'}$ denotes the ordered vector space of polynomials of degree at most $d$ modulo polynomials of degree at most $d'-1$. By
$$p(\E)=p^\sigma(\E)$$
we denote the reduced version.\\

We say that $\E$ is pure in $Q-\Coh(X)_{d,d'}$ if $T_{d-1}(\E)=T_{d'-1}(\E)$. The definition of semistability and stability applies to the relative setting, where we replace the multi--Hilbert polynomial by its class in $\R[T]_{d,d'}$. This satisfies the properties of a stability condition in the sense of \cite{rudakov}.

\begin{remark}\label{L:R:slope_is_relative}
Clearly, for $d'=0$ we recover the definition of semistability in $Q-\Coh(X)$, and for $d'=d-1$ we obtain slope semistability. The case $d'=d$ is trivial because the reduced Hilbert polynomial is just the monomial $\frac{1}{d!}T^d$.
\end{remark}

\begin{remark}
Unless emphasized differently, we are concerned with ordinary quiver sheaves in the proof of Theorem \ref{L:T:langton}. That is, in writing $\F\in Q-\Coh(X)_{d,d'}$ we refer to an object in $Q-\Coh(X)_{d,d'}$ which is represented by the quiver sheaf $\F$.
\end{remark}

By Proposition 1.9 of \cite{rudakov}, there exist maximally destabilizing subobjects for objects in $Q-\Coh(X)_{d,d'}$. For technical reasons we want to make sure that these subobjects are represented by saturated quiver subsheaves.
\begin{lemma}\label{L:L:destab_saturated}
The maximally destabilizing subobject
$$\G\subset \F$$
of any quiver sheaf $\F\in Q-\Coh(X)_{d,d'}$ is represented by an actual quiver subsheaf $\G\subset \F$ which is saturated.\\
\begin{proof}
First consider any representative and the map $i:\G\to \F$ giving the inclusion. This provides the diagram
$$\xymatrix{\G \ar[d] \ar[r]^i & \F. \\
\G/\ker(i) \ar@{^(->}[ur]}$$
Since $\ker(i)$ is small we may replace $\G$ by $\G/\ker(i)$. In the second step consider the saturation
$$\G\subset \G_{\sat}\subset \F.$$
Assuming $\G_{\sat}/\G$ not to be small, we have a strict inclusion $\G\subset \G_{sat}$ in the quotient category. But this contradicts the maximality of $\G$ because $\G_{\sat}$ has larger multi--Hilbert polynomial (both ordinary and modulo smaller degrees). Hence, $\G$ and $\G_{\sat}$ are isomorphic in $\Coh_{d,d'}(X)$.
\end{proof}
\end{lemma}

We need another preparatory lemma.
\begin{lemma}\label{L:L:duals_isomorphic}
Suppose that $\E$ and $\F$ are quiver sheaves projective scheme $X$ of dimension $n$ such that $s(\E),s(\F)\leq d$ and such that there is an isomorphism $\varphi:\E\to \F$ in dimension $d-1$. Then the induced morphism
$$\varphi^D:\F^D\to \E^D$$
is an isomorphism.

\begin{proof}
It is sufficient to prove the corresponding assertion about sheaves on $X$ in a functorial way. To that end, consider the exact sequence
$$0\to \ker(\varphi) \to F \to \im(\varphi) \to 0,$$
where by assumption the codimension $c$ of the kernel is greater equal to $n-d$. Consider the induced exact sequence
\begin{align*}
\Ext^{c-1}(\ker(\varphi),\omega_X) &\to \Ext^{c}(\im(\varphi),\omega_X)\\
&\to \Ext^{c}(F,\omega_X)\to \Ext^{c}(\ker(\varphi),\omega_X).
\end{align*}
$$$$
By \cite{huybrechtslehn} Proposition 1.1.6 both terms involving the kernel vanish. Hence there is an isomorphism
$$\varphi^D:\Ext^{c}(\im(\varphi),\omega_X)\to F^D.$$
Applying the same argument to the cokernel sequence shows that
$$G^D \to \Ext^{c}(\im(\varphi),\omega_X)$$
is also an isomorphism, which finishes the proof.
\end{proof}
\end{lemma}

Let $X$ denote a projective scheme over an algebraically closed field $k$ of characteristic 0, and consider a field extension $k\subset K$. Denote by $\E_K$ the base change of a quiver sheaf $\E$ on $X$, and by $\Lul_K$ the base change of the tuple $\Lul$. We now show that semistability is preserved by field extensions. This is a variant of \cite{huybrechtslehn}, Theorem 1.3.7.
\begin{proposition}\label{L:P:extensions_semistability}
Let $\E$ denote a pure quiver sheaf on a projective scheme $X$ over $k$ which is semistable with respect to some stability condition $(\Lul,\sigma)$. Consider a finitely generated field extension $k\subset K$. The pullback $\E_K$ is a semistable quiver sheaf on $X_K$ with respect to $\left(\Lul_K,\sigma\right)$.

\begin{proof}
We even claim that the Harder--Narasimhan filtrations are compatible in the sense that
$$\HN_i\left(\E_K\right)=\HN_i\left(\E\right)_K.$$
First note that any morphism $l\to L$ of fields induces a flat morphism $\Spec(L)\to \Spec(l)$, so that
$$h^0\left(X_L,E_L\right)=h^0\left(X_l,E_l\right)$$
for any coherent sheaf $E_l$ on $X_l$ (compare with the proof of \cite{hartshorne} Proposition III.9.3). In particular, this implies that the Hilbert polynomials for quiver sheaves $\E$ on $X=X_k$ remain identical in the sense that
$$p^{\left(\Lul,\sigma\right)}\left(\E\right)=p^{\left(\Lul_K,\sigma\right)}\left(\E_K\right).$$
Further, the flatness implies that quiver subsheaves and quiver subquotients of $\E$ get mapped to quiver subsheaves and quiver subquotients of $\E_K$ respectively.\\

A first consequence of these remarks is that if $\E_K$ is semistable, so is $\E$.\\
If we can then show that the Harder--Narasimhan filtration of $\E_K$ is induced by some filtration $\F^*$ of $\E$ in the sense that $\HN_i\left(\E_K\right)=\left(\F^i\right)_K$, then $\F^*$ satisfies the properties of the Harder--Narasimhan filtration of $\E$.\\

By induction on the number of generators of the field extension $k\subset K$ we reduce to the case that $K=k(x)$, where $x$ is either transcendental or algebraic and hence separable over $k$ (note that $k$ is perfect).\\

In the separable case we pass to the normal hull, so that we may assume the extension to be Galois. Thus, $\HN_i(\E_K)$ is induced by a quiver subsheaf $\F^i\subset \E$ if and only if $\HN_i(\E_K)$ is invariant under the induced action of $G=\Gal\left(K/k\right)$ on $\E_K$. To see this, we note that the corresponding descent question for sheaves is locally a question whether a submodule $N\subset M\otimes_k K$ over $R\otimes_k K$, where $M$ is a module over some $k$--algebra $R$, is induced by a submodule $N'\subset M$ if $N$ is invariant under the induced action of $G$ on $M\otimes_k K$. This is true by \cite{milne}, Proposition 16.7, with $N'=M\otimes_k k\cap N$. Clearly, these descents are also respected by induced morphisms $f\otimes_k K:M_1\otimes_k K\to M_2\otimes_k K$.\\
In our situation, we can see that for any $g\in G$ the $g*\HN_i\left(\E_K\right)$ satisfy the properties of the Harder--Narasimhan filtration by applying our initial remarks to the induced morphism $g:\Spec(K)\to \Spec(K)$. Hence the Harder--Narasimhan filtration is invariant and we are done.\\

In the case where $x$ is transcendental over $k$, i.e. $K=k(x)$ is the field of rational functions, we can use a similar argument using the relative automorphism group $G=\Aut\left(K/k\right)$, once we note that the relative automorphisms
$$x\mapsto ax$$
for $a\in k^*$ have $k$ as their fixed point field (this follows from the fact that there are no invariant polynomials by using \cite{mukai} Proposition 6.2).
\end{proof}
\end{proposition}

\begin{lemma}\label{L:L:local_flatness_criterion}
Let $X\to S$ denote a morhpism of finite type between noetherian schemes. Suppose that $S_0\subset S$ is a closed subscheme defined by a nilpotent ideal sheaf $\I\subset \Ocal_S$. Then a quiver sheaf $\F$ on $X$ is flat over $S$ if and only if it is flat over $S_0$ and the natural multiplication map
$$\I\otimes _S \F \to \I F$$
is an isomorphism.
\begin{proof}
This follows from the sheaf version \cite{huybrechtslehn} Lemma 2.1.3 once we note that the notion of flatness can be checked at each vertex and the natural multiplication map for sheaves extends to quiver sheaves.
\end{proof}
\end{lemma}

We are ready to prove Langton's theorem for families of semistable quiver sheaves.
\begin{theorem}\label{L:T:langton}
Let $R$ denote a discrete valuation ring with maximal ideal $\m=(\pi)$, field of fractions $K$, and residue field $k$.\\
Let $\F$ be an $R$--flat family of $d$--dimensional quiver sheaves on $X$ such that $\F_K=\F\otimes_R K$ is semistable in $Q-\Coh(X_K)_{d,d'}$ for some $d'<d$. Then there exists a quiver subsheaf $\E\subset \F$ such that $\E_K=\F_K$ and such that $\E_k$ is also semistable in $Q-\Coh(X)_{d,d'}$.

\begin{proof}
We prove the following (stronger) auxiliary statement:\\
Suppose that for $d'\leq \delta<d$ we have that $\F_k$ is semistable in $Q-\Coh(X)_{d,\delta+1}$ in addition to the assumptions of the theorem. Then there is a quiver subsheaf $\E\subset \F$ such that $\E_K=\F_K$ and $\E_k$ is semistable in $Q-\Coh(X)_{d,\delta}$.\\

Assuming that the auxiliary statement is true we obtain the statement of the theorem by induction on $\delta$, where the case $\delta=d-1$ is trivial (compare with Remark \ref{L:R:slope_is_relative}). From now on we assume that the auxiliary statement is false. By recursion we will define a descending sequence of quiver sheaves on $X_R$
$$\F=\F^0\supset \F^1 \supset \F^2 \supset \ldots,$$
where $\F^n_K=\F_K$ for all $n$. Under the assumption that the auxiliary statement is false each $\F^n_k$ is unstable in $Q-\Coh(X)_{d,\delta}$.\\
Suppose that $\F^n$ was already defined. Let $\K^n\subset \F^n_k$ be a saturated representative for the maximally destabilizing quiver subsheaf (given by Lemma \ref{L:L:destab_saturated}) and define $\G^n=\F^n_k/\K^n$ (note that $\G^n$ is pure). Then $\F^{n+1}$ is given as
$$\F^{n+1}=\ker\left(\F^n\to \F^n_k\to \G^n \right).$$
Note for later use that this implies $\K^{n-1}=\F^n/\pi \F^{n-1}$. Because outside of the closed point $\Spec(k)$, that is on $X_K$, every section gets mapped to zero we have $\F^{n+1}_K=\F^n_K$.\\
There is an obvious exact sequence
$$(S1): 0 \to \K^n \to \F^n_k \to \G^n \to 0.$$
Furthermore, we may now construct a second exact sequence
$$(S2): 0 \to \G^n \to \F^{n+1}_k \to \K^n \to 0.$$
Note that using induction these two exact sequences imply
$$(EQ1): P^\sigma(\F^n_k)=P^\sigma(\F_k)\in \R[T].$$

To construct the second sequence first note that $\F^n_k=\F^n/\pi\F^n$, where we denote the projection map as
$$q:\F^n\to \F^n_k=\F^n/\pi\F^n.$$
By construction of $\F^{n+1}$ and the universal property of the kernel we have an induced map $q_0:\F^{n+1}\to \K^n$. Note that $\pi\F^n\subset \F^{n+1}$, and $\pi\F^n$ clearly gets annihilated by $q$. Hence we obtain an induced morphism
$$\xymatrix{\F^{n+1}\ar[r]^{q_0} \ar[d] & \K^n \\
\F^{n+1}/\pi\F^n \ar[ur]_{q''}.}$$
Further note that $\pi\F^{n+1}\subset \pi\F^n$, so that $q''$ induces a map
$$q':\F^{n+1}_k=\F^{n+1}/\pi\F^{n+1}\to \K^n.$$
The kernel of this map consists of (classes of) sections of $\F^{n+1}$ which get annihilated by $q_0$. Hence
$$\ker(q')=\ker(q_0)/\pi\F^{n+1}=\pi\F^n/\pi\F^{n+1}.$$
By flatness, we further know that $\pi\F^n\simeq \F^n$ and $\pi\F^{n+1}\simeq \F^{n+1}$. Since the maps used for the construction of $\F^{n+1}$ are surjective we also have $\G^n\simeq \F^n/\F^{n+1}$. This yields $\ker(q')\simeq \G^n$.\\
To finish the construction of $(S2)$ observe that the surjectivity of $q'$ is inherited from $q_0$.\\

We define $\C^n=\G^n\cap \K^{n+1}$ considered as quiver subsheaves of $\F_k^{n+1}$ via the exact sequences above. Then the exact sequence $(S2)$ and the obvious inclusions $\C^n\subset \G^n$ and $\C^n\subset \K^{n+1}$ induce a map $\K^{n+1}/\C^n\to \K^n$
$$\xymatrix{0 \ar[r] & \G^n \ar[r] & \F^{n+1}_k \ar[r] & \K^n \ar[r] & 0\\
0 \ar[r] & \C^n \ar[r] \ar@{^(->}[u] & \K^{n+1} \ar[r] \ar@{^(->}[u] & \K^{n+1}/\C^n \ar[r] \ar@{-->}[u]_{i} & 0.}$$
A proof by diagram chasing shows that $i$ is an inclusion (note that $\C^n$ is a pullback). In a similar fashion, using sequence $(S1)$, we get a monomorphism $\G^n/\C^n\to \G^{n+1}$.\\

Assuming that $\C^n$ is not isomorphic to zero in $Q-\Coh(X_k)_{d,\delta}$ yields a contradiction as follows. In case $\C^n=\K^{n+1}$ we get the inequalities
$$(IE1): p^\sigma(\K^{n+1})=p^\sigma(\C^n)\leq p^\sigma_{\max}(\G^n)<p^\sigma(\K^n)\in \R[T]_{d,\delta}.$$
The rightmost inequality holds because $p^\sigma_{\max}(\G^n)$ and $p^\sigma(\K^n)$ are the second and first Harder--Narasimhan weights of $\F^n_k$ respectively. In case $\C^n\neq \K^{n+1}$ we have the inequalities
$$(IE2): p^\sigma(\C^n)< p^\sigma(\K^{n+1})< p^\sigma(\K^{n+1}/\C^n)\leq p^\sigma(\K^n)\in \R[T]_{d,\delta}.$$
The first and last inequalities hold by the defining property of a destabilizing subobject. Note that the first inequality is strict because of the assumption, and the latter inequality also uses the inclusion constructed above. The inequality in the middle is implied by the first one using the obvious exact sequence.\\

In any case we have
$$(IE3): p^\sigma(\K^{n+1})\leq p^\sigma(\K^n)\in \R[T]_{d,\delta}.$$
If $\C^n$ is not isomorphic to zero this holds by $(IE1)$ and $(IE2)$. And if $\C^n\simeq 0$ this holds because we have the inclusion $\K^{n+1}=\K^{n+1}/\C^n\subset \K^n$.\\

Since we assume $\F_k$ to be semistable in $Q-\Coh(X)_{d,\delta+1}$ we have that $p^\sigma(\K^n)\leq p^\sigma(\F_k^n)\in \R[T]_{d,\delta+1}$. But strict inequality would also imply
$$p^\sigma(\K^n)<p^\sigma(\F_k^n)\in \R[T]_{d,\delta},$$
contradicting the fact that $\F^n_k$ is not semistable. Together with $(EQ1)$ we thus arrive at
$$(EQ2): p^\sigma(\K^n)=p^\sigma(\F_k)\in \R[T]_{d,\delta+1}.$$
Hence
$$p^\sigma(\K^n)-p^\sigma(\F_k)=\beta_n T^{\delta}\in \R[T]_{d,\delta}$$
for some $\beta_n\in \R$. We need some properties of the sequence $\beta_n\in \R$.\\
Because $p^\sigma(\K^n)>p^\sigma(\F^n_k)=p^\sigma(\F_k)\in \R[T]_{d,\delta}$, which holds by unstability of $\F^n_k$ and $(EQ1)$, the $\beta_n$ are strictly positive, and by $(IE3)$ their sequence is decreasing. Finally, the possible values for $\beta_n$ are contained in discrete set (consider Lemma \ref{L:L:betadiscrete}). Hence, $\beta_n$ must become stationary for $n\gg 0$, and without loss of generality we restrict ourselves to such $n$ in the following.\\
This implies $p^\sigma(\K^n)=p^\sigma(\K^{n+1})$ in $\R[T]_{d,\delta}$, and because this contradicts both $(IE1)$ and $(IE2)$ it implies furthermore that
$$\G^n\cap \K^{n+1}=\C^n\simeq 0.$$
Observe that $\C^n$ is a quiver subsheaf of a pure quiver sheaf of dimension $d$ ($\G^n$ or $\K^{n+1}$). But since it can not have dimension $d$ it must equal zero, so that there are inclusions $i:\K^{n+1}\to \K^n$ as well as $j:\G^n\to \G^{n+1}$ in $Q-\Coh(X_k)_{d,\delta}$.\\
Note that $P^\sigma(\K^n)=P^\sigma(\K^{n+1})\in \R[T]_{d,\delta}$ as well for large enough $n$ because the rank of the $\K^n$ can not descend forever. Combined with $(EQ1)$ this gives
$$(EQ3): P^\sigma(\G^n)=P^\sigma(\G^{n+1})\in \R[T]_{d,\delta}.$$

Since $\G^n$ is pure the kernel of $j$ is either zero or of dimension $d$. Clearly the latter is absurd, and hence there are actual inclusions $j:\G^n\subset \G^{n+1}$. Thus we get exact sequences
$$0\to \G^n \to \G^{n+1} \to \G^n/\G^{n+1} \to 0,$$
where $(EQ3)$ implies that the rightmost term is isomorphic to zero in the category $Q-\Coh(X_k)_{d,\delta}$. So the $\G^n$ are isomorphic in dimension $\delta$ and thus in particular in dimension $d-1$. By Lemma \ref{L:L:duals_isomorphic} their reflexive hulls $(\G^n)^{DD}$ are all isomorphic, so that the sequence
$$\G^0\subset \G^1 \subset \G^2 \subset \ldots$$
is an increasing sequence of quiver subsheaves of the fixed quiver sheaf which is given as this reflexive hull, and must thus become stationary. Again we restrict to the case where $n\gg 0$ is such that this is the case and set $\G=\G^n$.\\
We can see that this implies that the sequences $(S1)$ and $(S2)$ split as follows. Consider the diagram
$$\xymatrix{0 \ar[r] & \K^{n+1} \ar[r] \ar@{=}[d] & \K^{n+1}\oplus \G^n \ar[r] \ar[d] & \G^n \ar[r] & 0 \\
0 \ar[r] & \K^{n+1} \ar[r] & \F^{n+1}_k \ar[r] & \G^{n+1} \ar[r] & 0,}$$
where the first row is given by $(S1)$ and the middle morphism as the sum of the injective morphisms in $(S1)$ and $(S2)$. Thus there is an induced morphism $\G^n\to \G^{n+1}$. By construction, this is exactly the inclusion morphism $\G^n\subset \G^{n+1}$, which turned out to be an isomorphism, so that
$$\F^{n+1}_k\simeq \K^{n+1}\oplus \G^n \simeq \K\oplus \G$$
for all $n$ large enough.\\

Define $\E^n=\F/\F^n$. Because $\pi\F^n\subset \F^{n+1}$, and so by induction $\pi^n\F\subset \F^n$, there is a well--defined quotient map
$$\F/\pi^n\F \to \F/\F^n=\E^n.$$
By the local flatness criterion for quiver sheaves \ref{L:L:local_flatness_criterion} we thus know that $\E^n$ is flat over $R/\pi^n$.\\

Next, we want to show $\E^n_k\simeq \G$, which gives us the topological type of $\E^n$.\\
Using Noether's isomorphism theorem we obtain
\begin{align*}
\E^n_k&=(\F/\F^n)/(\pi(\F/\F^n))=(\F/\F^n)/((\pi\F+\F^n)/\F^n)\simeq \F/(\pi\F+\F^n)\\
&\simeq (\F/\pi\F)/((\F^n+\pi\F)/\pi \F)=\F_k/\im(\alpha),
\end{align*}
where $\alpha$ is the composition of the morphisms
$$\alpha_n:\F^n_k\to \F^{n-1}_k,~f+\pi\F^n\mapsto f+\pi \F^{n-1}.$$
Note that
\begin{align*}
\ker(\alpha_n)&=\pi \F^{n-1}/\pi \F^n\simeq \F^{n-1}/\F^n\simeq \G^{n-1}\\
\im(\alpha_n)&=\F^n/\pi\F^{n-1}=\K^{n-1},
\end{align*}
which implies that we get decompositions
$$\F^n_k=\K^n\oplus \G^{n-1}=\im(\alpha_{n+1}) \oplus \ker(\alpha_n).$$
Hence $\im(\alpha)=\im(\alpha_1)=\K^0$ and $\E^n_k=\G^0\simeq \G$ as desired.\\

Summarizing these results we have shown that $\E^n$ is a quotient
$$\F_{R_n}\to \E^n \to 0,$$
which is flat over $R_n=R/\pi^n$. This corresponds to a morphism $\phi$ over $\Spec(R)$ which fits into a diagram
$$\xymatrix{\Quot^{\tau(\G)}_{\F/X_R/R}\ar[rr]^\sigma&&\Spec(R).\\
&\Spec(R_n)\ar[ul]^\phi \ar@{^(->}[ur]&}$$
Hence the image of $\sigma$ contains the closed subschemes $\Spec(R_n)$ of $\Spec(R)$ for all $n$, which is only possible if $\sigma$ is surjective.\\

The morphism $\Spec(K)\to \Spec(R)$ corresponds to a point $y\in \Spec(R)$ such that $k(y)\subset K$. By surjectivity of $\sigma$ we can find an inverse image $x\in \Quot^{\tau(\G)}_{\F/X_R/R}(\F,\tau(\G))$. Let $K'$ denote the common extension of the induced field extension $k(x)\subset k(y)$ and the extension $k(x)\subset K$, so that
$$\xymatrix{ & K' & \\
K \ar@{^(->}[ur] & & k(y) \ar@{_(->}[ul] \\
& k(x) \ar@{_(->}[ul] \ar@{^(->}[ur] &}$$
Reversing the correspondences used above the extension $k(y)\subset K'$ gives a morphism $\Spec(K')\to \Quot^{\tau(\G)}_{\F/X_R/R}$ over $\Spec(R)$ and hence a quotient
$$\F_{K'}\to \U \to 0$$
with topological type $\tau(\G)$. By Proposition \ref{L:P:extensions_semistability} we know that $\F_{K'}$ is semistable. But by our assumptions, the Hilbert polynomials satisfy the inequality
$$p^\sigma(\U)=p^\sigma(\G)>p^\sigma(\F^n_k)=p^\sigma(\F_{K'}).$$
This is a contradiction.
\end{proof}
\end{theorem}

The fact that the stability condition $\sigma$ is rational is crucial for the validity of the proof of Theorem \ref{L:T:langton}. Because of its significance, we state the relevant step as a separate lemma.
\begin{lemma}\label{L:L:betadiscrete}
With notation as in the proof of Theorem \ref{L:T:langton}, the numbers $\beta_n$ such that
$$p^\sigma(\K^n)-p^\sigma(\F_k)=\beta_nT^\delta$$
are contained in a discrete set, if $\sigma$ is rational.

\begin{proof}
Since the polynomial $p^\sigma(\F_k)$ is independent of $n$, it remains to show that the coefficients of $p^\sigma(\K^n)$ take values in a discrete set. Such a coefficient is given as
$$\frac{\sum_{i\in Q_0}\sum_{j=1}^N \sigma_{ij}\alpha^{L_j}_\delta(\K^n_i)}{\sum_{i\in Q_0}\sum_{j=1}^N \sigma_{ij}\alpha^{L_j}_d(\K^n_i)}.$$
By construction, $\K^n$ is a quiver subsheaf of $\F^n_k$, so that each $\alpha^{L_j}_d(\K^n_i)$ is an integer between $1$ and $\alpha^{L_j}_d(\K^n_i)$. The upper bound is independent of $n$ because $\F^n_K=\F^{n+1}_K$ and $\F^n$ is flat, so that there are only finitely many possible values for the denominator.\\
For $\delta<d$ it is well--known that the coefficients $\alpha^{L_j}_\delta(\K^n_i)$ take value in some lattice $(1/r!)\Z$. The numerator thus takes values in a set of the form
$$\Z a_1+\ldots+\Z a_r$$
for finitely many rational numbers $a_r$. By factoring out the denominators of the $a_i$, we see that this set is contained in a cyclic subgroup $\Z\alpha\subset \R$, and is thus discrete.
\end{proof}
\end{lemma}

\end{appendices}


\begin{thebibliography}{99999999999}

\bibitem[A09]{alvarezconsul}
{\'Alvarez-C\'onsul, L.},
{\it Some results on the moduli spaces of quiver bundles},
 Geom. Dedicata {\bf 139} (2009), 99--120\\

\bibitem[AK07]{ack}
{Álvarez-Cónsul, Luis; King, Alastair.}
\textit{A functorial construction of moduli of sheaves},
{Invent. Math. \textbf{168} (2007), no. 3, 613--666}\\


\bibitem[F98]{fulton}
{Fulton, William.}
\textit{Intersection theory, Second edition},
{Results in Mathematics and Related Areas, \textbf{3}rd Series. A Series of Modern Surveys in Mathematics, 2. Springer-Verlag, Berlin, 1998}\\

\bibitem[G90]{gabriel}
{Gabriel, Pierre.}
\textit{Des catégories abéliennes (French)},
{Bull. Soc. Math. France \textbf{90} (1962) 323--448}\\


\bibitem[GRT16]{grt}
{Greb, Daniel; Ross, Julius; Toma, Matei},
\textit{Variation of Gieseker moduli spaces via quiver GIT},
{Geom. Topol. \textbf{20} (2016), no. 3, 1539--1610}\\

\bibitem[H77]{hartshorne}
{Hartshorne, Robin},
\textit{Algebraic geometry},
{Graduate Texts in Mathematics, No. \textbf{52}. Springer-Verlag, New York-Heidelberg, 1977}\\


\bibitem[HL10]{huybrechtslehn}
{Huybrechts, Daniel; Lehn, Manfred.}
\textit{The geometry of moduli spaces of sheaves},
{Second edition. Cambridge Mathematical Library. Cambridge University Press, Cambridge, 2010}\\

\bibitem[K94]{king}
{King, A. D.},
\textit{Moduli of representations of finite-dimensional algebras},
{Quart. J. Math. Oxford Ser. (2) \textbf{45} (1994), no. 180, 515--530}\\

\bibitem[La75]{langton}
{Langton, Stacy G.}, 
\textit{Valuative criteria for families of vector bundles on algebraic varieties},
{Ann. of Math. (2) \textbf{101} (1975), 88--110}\\


\bibitem[Lo64]{lojasiewicz}
{\L ojasiewicz, S.},
\textit{Triangulation of semi-analytic sets},
{Ann. Scuola Norm. Sup. Pisa (3) \textbf{18} (1964) 449--474}\\

\bibitem[LP90]{lebruynprocesi}
{Le Bruyn, Lieven; Procesi, Claudio},
\textit{Semisimple representations of quivers},
{Trans. Amer. Math. Soc. \textbf{317} (1990), no. 2, 585--598}\\


\bibitem[Mi15]{milne}
{Milne, J.S.},
\textit{(Topics in) Algebraic Geometry, Ch. 16: Descent theory},
{lecture notes, http://www.jmilne.org/math/CourseNotes/ag.html}\\

\bibitem[Mu03]{mukai}
{Mukai, Shigeru},
\textit{An introduction to invariants and moduli},
{Translated from the 1998 and 2000 Japanese editions by W. M. Oxbury. Cambridge Studies in Advanced Mathematics, \textbf{81}, 
Cambridge University Press, Cambridge, 2003}\\


\bibitem[MW97]{matsukiwentworth}
{Matsuki, Kenji; Wentworth, Richard},
\textit{Mumford-Thaddeus principle on the moduli space of vector bundles on an algebraic surface},
{Internat. J. Math. \textbf{8} (1997), no. 1, 97--148}\\

\bibitem[MFK94]{mumford}
{Mumford, D.; Fogarty, J.; Kirwan, F.}
\textit{Geometric invariant theory. Third edition},
{Ergebnisse der Mathematik und ihrer Grenzgebiete (2) \textbf{34} Springer-Verlag, Berlin, 1994}\\

\bibitem[N05]{nitsure}
{Nitsure, Nitin},
\textit{Construction of Hilbert and Quot schemes},
{Fundamental algebraic geometry, 105--137, Math. Surveys Monogr., \textbf{123}, Amer. Math. Soc., Providence, RI, 2005}\\


\bibitem[Re08]{reineke}
{Reineke, Markus},
\textit{Moduli of representations of quivers},
{Trends in representation theory of algebras and related topics, 589--637, EMS Ser. Congr. Rep., Eur. Math. Soc., Zürich, 2008}\\


\bibitem[Ru97]{rudakov}
{Rudakov, Alexei},
\textit{Stability for an abelian category.}
{J. Algebra \textbf{197} (1997), no. 1, 231--245}\\

\bibitem[Sch00]{schmitt-wall}
{Schmitt, Alexander},
\textit{Walls for Gieseker semistability and the Mumford-Thaddeus principle for moduli spaces of sheaves over higher dimensional bases},
{Comment. Math. Helv. \textbf{75} (2000), no. 2, 216--231}\\

\bibitem[Sch05]{schmitt-tuple}
{Schmitt, Alexander},
\textit{Moduli for decorated tuples of sheaves and representation spaces for quivers},
{Proc. Indian Acad. Sci. Math. Sci. \textbf{115} (2005), no. 1, 15--49}\\


\bibitem[Sch12]{schmitt-qsheaf}
{Schmitt, Alexander},
{\it A remark on semistability of quiver bundles}, Eurasian Math. J. {\bf 3} (2012), 1, 110--138\\

\bibitem[Si94]{simpson}
{Simpson, Carlos T},
\textit{Moduli of representations of the fundamental group of a smooth projective variety}
{I. Inst. Hautes Études Sci. Publ. Math. No. \textbf{79} (1994), 47--129}\\


\bibitem[T96]{thaddeus}
{Thaddeus, Michael},
\textit{Geometric invariant theory and flips},
{J. Amer. Math. Soc. \textbf{9} (1996), no. 3, 691--723}\\


\end{thebibliography}
\end{document}